\documentclass[11pt]{amsart}
\usepackage{indentfirst,latexsym,bm}
\usepackage{amsfonts}
\usepackage{amssymb}
\usepackage{amsmath}
\usepackage{amsbsy}
\usepackage{dsfont}
\usepackage{amsthm}
\usepackage{amscd}
\usepackage[all]{xy}
\usepackage{chemarrow}
\usepackage{color}
\usepackage{multirow}
\usepackage{hyperref}
\usepackage[titletoc]{appendix}
\usepackage{eepic}
\usepackage{color}
\usepackage{graphicx}
\usepackage{ifthen}

\begin{document}

\title[Classification of finite-dimensional Hopf algebras] {Classification of finite-dimensional Hopf algebras over dual Radford algebras}
\author[Rongchuan Xiong]{Rongchuan Xiong}
\address{School of Computer Science and Artificial Inteligent, Aliyun School of Big Data, Changzhou University, Changzhou 213164, China}
\email{rcxiong@foxmail.com}
\author[Naihong Hu]{Naihong Hu$^*$}
\address{School of Mathematical Sciences, Shanghai Key Laboratory of PMMP, East China Normal University, Shanghai 200241, China}
\email{nhhu@math.ecnu.edu.cn}

\thanks{$^*$ Corresponding author.}
\subjclass[2010]{16T05, 16S35, 18D10}
\date{}
\maketitle

\newtheorem{question}{Question}
\newtheorem{defi}{Definition}[section]
\newtheorem{conj}[defi]{Conjecture}
\newtheorem{thm}[defi]{Theorem}
\newtheorem{lem}[defi]{Lemma}
\newtheorem{pro}[defi]{Proposition}
\newtheorem{cor}[defi]{Corollary}
\newtheorem{rmk}[defi]{Remark}
\newtheorem{example}[defi]{Example}

\theoremstyle{plain}
\newcounter{maint}
\renewcommand{\themaint}{\Alph{maint}}
\newtheorem{mainthm}[maint]{Theorem}

\theoremstyle{plain}
\newtheorem*{proofthma}{Proof of Theorem A}
\newtheorem*{proofthmb}{Proof of Theorem B}
\newcommand{\tabincell}[2]{\begin{tabular}{@{}#1@{}}#2\end{tabular}}

\newcommand{\C}{\mathcal{C}}
\newcommand{\D}{\mathcal{D}}
\newcommand{\A}{\mathcal{A}}
\newcommand{\De}{\Delta}
\newcommand{\M}{\mathcal{M}}

\newcommand{\K}{\mathds{k}}
\newcommand{\E}{\mathcal{E}}
\newcommand{\Pp}{\mathcal{P}}
\newcommand{\Lam}{\lambda}
\newcommand{\As}{^{\ast}}
\newcommand{\Aa}{a^{\ast}}
\newcommand{\Ab}{(a^2)^{\ast}}
\newcommand{\Ac}{(a^3)^{\ast}}
\newcommand{\Ad}{(a^4)^{\ast}}
\newcommand{\Ae}{(a^5)^{\ast}}
\newcommand{\B}{b^{\ast}}
\newcommand{\BAa}{(ba)^{\ast}}
\newcommand{\BAb}{(ba^2)^{\ast}}
\newcommand{\BAc}{(ba^3)^{\ast}}
\newcommand{\BAd}{(ba^4)^{\ast}}
\newcommand{\BAe}{(ba^5)^{\ast}}
\newcommand{\cF}{\mathcal{F}}
\newcommand{\cK}{\mathcal{K}}
\newcommand{\cX}{\mathcal{X}}
\newcommand{\Pa}{p}

\newcommand{\CYD}{{}^{\C}_{\C}\mathcal{YD}}
\newcommand{\AYD}{{}^{\A}_{\A}\mathcal{YD}}
\newcommand{\HYD}{{}^{H}_{H}\mathcal{YD}}
\newcommand{\ydH}{{}^{H}_{H}\mathcal{YD}}
\newcommand{\KYD}{{}^{K}_{K}\mathcal{YD}}
\newcommand{\DM}{{}_{D}\mathcal{M}}
\newcommand{\BN}{\mathcal{B}}
\newcommand{\cZ}{\mathcal{Z}}
\newcommand{\cO}{\mathcal{O}}
\newcommand{\I}{\mathds{I}}
\newcommand{\J}{\mathds{J}}
\newcommand{\Z}{\mathds{Z}}
\newcommand{\N}{\mathds{N}}
\newcommand{\cL}{\mathcal{L}}
\newcommand{\cW}{\mathcal{W}}
\newcommand{\G}{\mathcal{G}}
\newcommand{\roots }{\boldsymbol{\Delta }}
\newcommand{\cJ}{\mathcal{J}}

\newcommand\ad{\operatorname{ad}}
\newcommand\Ob{\operatorname{Ob}}
\newcommand{\Alg}{\Hom_{\text{alg}}}
\newcommand\Aut{\operatorname{Aut}}
\newcommand{\AuH}{\Aut_{\text{Hopf}}}
\newcommand\coker{\operatorname{coker}}
\newcommand\car{\operatorname{char}}
\newcommand\Der{\operatorname{Der}}
\newcommand\diag{\operatorname{diag}}
\newcommand\End{\operatorname{End}}
\newcommand\mult{\operatorname{mult}}
\newcommand\id{\operatorname{id}}
\newcommand\Char{\operatorname{char}}
\newcommand\gr{\operatorname{gr}}
\newcommand\GK{\operatorname{GKdim}}
\newcommand{\Hom}{\operatorname{Hom}}
\newcommand\ord{\operatorname{ord}}
\newcommand\rk{\operatorname{rk}}
\newcommand\Soc{\operatorname{soc}}
\newcommand\lgot{\operatorname{l}}
\newcommand\Top{\operatorname{top}}
\newcommand\supp{\operatorname{supp}}
\newcommand{\bp}{\mathbf{p}}
\newcommand{\bq}{\mathbf{q}}
\newcommand\Sb{\mathbb S}
\newcommand\cR{\mathcal{R}}
\newcommand\cH{\mathcal{H}}
\newcommand{\grAYD}{{}^{\gr\A}_{\gr\A}\mathcal{YD}}

\newcommand{\Dchaintwo}[3]{\xymatrix@C-4pt{\overset{#1}{\underset{x }{\circ}}\ar
@{-}[r]^{#2}
& \overset{#3}{\underset{v_1 }{\circ}}}}

\newcommand{\Dchaintwoa}[3]{\xymatrix@C-4pt{\overset{#1}{\underset{\  }{\circ}}\ar
@{-}[r]^{#2}
& \overset{#3}{\underset{\ }{\circ}}}}

\newcommand{\Dchainthree}[8]{\xymatrix@C-2pt{
\overset{#1}{\underset{#2}{\circ}}\ar  @ {-}[r]^{#3}  & \overset{#4}{\underset{#5
}{\circ}}\ar  @{-}[r]^{#6}
& \overset{#7}{\underset{#8}{\circ}} }}

\newcommand{\Dtriangle}[6]{
\xymatrix@R-12pt{  &    \overset{#1}{\underset{x}{\circ}} \ar  @{-}[dl]_{#2}\ar  @{-}[dr]^{#3} & \\
\overset{#4}{\underset{w_1}{\circ}} \ar  @{-}[rr]^{#5}  &  &\overset{#6}{\underset{v_1}{\circ}} }}
\newcommand{\DDtriangle}[6]{
\xymatrix@R-12pt{  &    \overset{#1}{{\circ}} \ar  @{-}[dl]_{#2}\ar  @{-}[dr]^{#3} & \\
\overset{#4}{{\circ}} \ar  @{-}[rr]^{#5}  &  &\overset{#6}{{\circ}} }}

\newcommand{\cDchaintwo}[4]{
\rule[-3\unitlength]{0pt}{8\unitlength}
\begin{picture}(14,5)(0,3)
\put(1,2){\ifthenelse{\equal{#1}{l}}{\circle*{2}}{\circle{2}}}
\put(2,2){\line(1,0){10}}
\put(13,2){\ifthenelse{\equal{#1}{r}}{\circle*{2}}{\circle{2}}}
\put(1,5){\makebox[0pt]{\scriptsize #2}}
\put(7,4){\makebox[0pt]{\scriptsize #3}}
\put(13,5){\makebox[0pt]{\scriptsize #4}}
\end{picture}}

\begin{abstract}
We determine and classify all finite-dimensional Hopf algebras over an algebraically
closed field of characteristic zero whose Hopf coradicals are isomorphic
to dual Radford algebras of dimension $4p$ for a prime $p>5$.
In particular, we obtain families of new examples of finite-dimensional
Hopf algebras without the dual Chevalley property.

\bigskip
\noindent {\bf Keywords:} Nichols algebra; Hopf algebra;  Dual Chevalley property; Dual Radford algebra.
\end{abstract}


\section{Introduction}
Let $\K$ be an algebraically closed field of characteristic zero. This work is to classify those finite-dimensional Hopf algebras over $\K$ without the dual Chevalley property, that is, the coradical is not a subalgebra. To date there are few classification results on such Hopf algebras without pointed duals, with some exceptions in \cite{GG16,HX17,X17,X18} by means of the generalized lifting method \cite{AC13}. More examples are needed to get a better understanding of such Hopf algebras structures.

As a generalization of the \emph{lifting method }introduced by Andruskiewitsch-Schneider in \cite{AS98}, the generalized lifting method works because of the following observation. Let $A$ be a Hopf algebra over $\K$  and $A_{[0]}$ the subalgebra generated by the coradical $A_0$ of $A$.  We will say $A_{[0]}$ is the \emph{Hopf coradical} of $A$.  Assume that $S_A(A_{[0]})\subseteq A_{[0]}$, the standard filtration $\{A_{[n]}\}_{n\geq 0}$, defined recursively by $A_{[n]}=A_{[n-1]}\bigwedge A_{[0]}$, is a Hopf algebra filtration, which implies that  the associated graded coalgebra
$\gr A=\bigoplus_{n=0}^{\infty}A_{[n]}/A_{[n-1]}$ with $A_{[-1]}=0$ is a Hopf algebra. It follows from \cite[Theorem 2]{R85} that there exists uniquely a connected graded braided Hopf algebra $R=\bigoplus_{n\geq 0}R(n)$ in ${}^{A_{[0]}}_{A_{[0]}}\mathcal{YD}$ such that $\gr A\cong R\sharp A_{[0]}$. We call $R$ and $R(1)$ the \emph{diagram} and \emph{infinitesimal braiding} of $A$, respectively. If the coradical $A_{0}$ is a Hopf subalgebra, then the generalized lifting method coincides with the lifting method.

In order to obtain new classification results, according to the principle of the generalized lifting method due to \cite{AC13}, one needs to solve the following $3$-steps questions:
\begin{itemize}
  \item {\text{\rm Question\,1. }} Let $C$ be a cosemisimple coalgebra and $\mathcal{S} : C \rightarrow C$
  an injective anti-coalgebra morphism. Classify all Hopf algebras $L$ generated by $C$, such that $S|_C = \mathcal{S}$.
  \item {\text{\rm Question\,2. }} Given $L$ as in the previous item, classify all connected graded
  Hopf algebras $R$ in ${}_L^L\mathcal{YD}$.
  \item {\text{\rm Question\,3. }} Given $L$ and $R$ as in previous items, classify all liftings,
  that is, classify all Hopf algebras $A$ such that $\text{gr}\,A\cong R\sharp L$. We call $A$ a \emph{lifting} of $R$ over $L$.
\end{itemize}

These questions are largely open for lacks of standard methods and effective tools. Basic Hopf algebras are main  examples of Hopf algebras in Question 1. Of course, not all basic Hopf algebras have this property. For Questions 2 and 3, Garcia and Giraldi \cite{GG16} first studied Hopf algebras over the smallest non-pointed basic Hopf algebra $\cH_{2,-1}$, where $R$ are Nichols algebras over simple objects in ${}_{\cH_{2,-1}}^{\cH_{2,-1}}\mathcal{YD}$;  then the first author continued to study the case where $R$ are Nichols algebras over semisimple objects in ${}_{\cH_{2,-1}}^{\cH_{2,-1}}\mathcal{YD}$ \cite{X18} or  simple objects in ${}_{\cH_{3,-1}}^{\cH_{3,-1}}\mathcal{YD}$ \cite{X17}. A striking work by \cite{AA18b} showed that a finite-dimensional Hopf algebra whose Hopf coradical is basic is a lifting of a Nichols algebra over a semisimple Yetter-Drinfeld module, under some restrictions.

Based on a series of works \cite{GG16, X17, X18, X18b, X19}, we shall fix a family of Hopf algebras $\cH_{p,-1}$ for any natural number $p$ of dimension $4p$ in Question 1 and then study Questions $2$ and $3$. This extends the work in \cite{GG16, X17, X18, X18b}.

Let $\xi$ be a primitive $2p$-th root of unity and $\Lam=\frac{\xi-1}{\xi+1}$. As an algebra, $\cH_{p,-1}=\K\langle a, b\rangle$, subject to the relations
 \begin{align*}
a^{2p}=1, \quad b^2=0,\quad ba=\xi ab;
\end{align*}
and its comultiplication is
\begin{align*}
 \De(a)=a\otimes a+ \Lam^{-1}b\otimes ba^p, \quad
 \De(b)=b\otimes a^{p+1}+a\otimes b.
\end{align*}

The Hopf algebra $\cH_{p,-1}$ is the dual of a Radford algebra $\A_{p,-1}$ \cite{R75} and its Green ring was determined in \cite{CYWX18}. In particular,  if $p$ is a prime number, then $\cH_{p,-1}$ is the only non-pointed basic Hopf algebra of dimension $4p$ \cite{AN01}.
Consider the Drinfeld double $\D(\cH_{p,-1}^{cop})$  of $\cH_{p,-1}^{cop}$. We shall determine directly all simple objects  in ${}_{\D(\cH_{p,-1}^{cop})}\mathcal{M}$. Indeed, we shall show in Theorem \ref{thmsimplemoduleD} that there are $2p$ one-dimensional objects
$\K_{\chi_{i}}$ with $i\in\I_{0,2p-1}$ and $4p^2-2p$ two-dimensional objects $V_{i,j}$ with
$(i,j)\in\Lambda_p$, where $\Lambda_p=\{(i,j)\in \I_{0,2p-1}\times \I_{0,2p-1}\mid pi-j\not\equiv 0\mod 2p\}$.

In order to study Questions 2 and 3 in the framework of the generalized lifting method, using the  braided equivalence ${}_{\cH_{p,-1}}^{\cH_{p,-1}}\mathcal{YD}\cong {}_{\D(\cH_{p,-1}^{cop})}\mathcal{M}$
\cite[Proposition\,10.6.16]{M93}, we compute simple objects in ${}_{\cH_{p,-1}}^{\cH_{p,-1}}\mathcal{YD}$ and their braidings.  Following the idea in \cite{AA18b}, by the braided equivalence ${}_{\cH_{p,-1}}^{\cH_{p,-1}}\mathcal{YD}\cong{}_{\gr\A_{p,-1}}^{\gr\A_{p,-1}}\mathcal{YD}$, under the assumption that $p$ is a prime number or $p=4$, we determine all objects $V$ in ${}_{\cH_{p,-1}}^{\cH_{p,-1}}\mathcal{YD}$ with $\dim\BN(V)<\infty$, which extends \cite[Theorem A]{X18} and partial results in \cite{BGGM}, \cite{X19}.  See Theorem \ref{thm-finite-braided-vector-spaces-H},  Theorem \ref{thm-finite-braided-vector-spaces-H-p=4} for details.
The main part of this paper arises from a part of the first author's Ph. D. Dissertation \cite{X19}.
The paper \cite{BGGM} determined finite-dimensional Nichols algebras over simple objects in
${}_{\cH_{p,-1}}^{\cH_{p,-1}}\mathcal{YD}$ with different notations.
The authors became aware of the results in \cite{BGGM} after sending \cite{X19} to Garcia.

Besides, we also determine much more examples of Nichols algebras of non-diagonal type when removing the condition that $p$ is a prime number. Furthermore, we present them by generators and relations. By computing the liftings of Nichols algebras, we construct families of finite-dimensional Hopf algebras without the dual Chevalley property and pointed duals, which constitute new examples of finite-dimensional Hopf algebras. In particular, we obtain the following classification results.

\begin{thm}[Theorem \ref{thm:lifting-braiding-object-quad-simple}]\label{thm:thm:lifting-braiding-object-quad-simple}
 Assume that $A$ is a finite-dimensional Hopf algebra over $\cH_{p,-1}$ whose infinitesimal braiding $V$ is an indecomposable object in ${}_{\cH_{p,-1}}^{\cH_{p,-1}}\mathcal{YD}$. If the diagram of $A$ admits non-trivial quadratic relations, then $A$ is isomorphic to  one of the following Hopf algebras
\begin{itemize}
  \item[(1)] $\bigwedge\K_{\chi^k}\sharp\cH_{p,-1}$, where  $k\in\I_{0,2p-1}$ is an odd number,
  \item[(2)] $\BN(V_{i,j})\sharp\cH_{p,-1}$, where $(i,j)\in\Lambda_p^1 $,
  \item[(3)] $\BN(V_{i,j})\sharp\cH_{p,-1}$, where $(i,j)\in \Lambda^2_p-\Lambda^3_p$,
  \item[(4)] $\mathfrak{A}_{i,j}(\mu)$, where $\mu\in\K$ and $(i,j)\in\Lambda_p^3$.
\end{itemize}
\end{thm}

The Hopf algebras described in $(1)$, $(2)$ and $(3)$ of Theorem \ref{thm:thm:lifting-braiding-object-quad-simple} are basic Hopf algebras of dimension $8p$, $8pN_1$ and $8pN_2$, respectively, where $N_1=\ord((-1)^i\xi^{-j})$ and $N_2=\ord(\xi^{-ij})$.  The Hopf algebras descried in $(4)$ of Theorem \ref{thm:thm:lifting-braiding-object-quad-simple} with $\mu\neq 0$ are not basic and  constitute new examples of $32p$-dimensional Hopf algebras.

\begin{thm}[Theorem~\ref{thm:f.d.HopfH4pp7-1}]\label{thm:thm:f.d.HopfH4pp7-1}
Suppose that $p>5$ is a prime number and $A$ a finite-dimensional Hopf algebra over $\cH_{p,-1}$. Then $A\cong \gr A$.
\end{thm}
Theorem \ref{thm:thm:f.d.HopfH4pp7-1} indicates that $A$ is basic.

\begin{thm}[Theorem \ref{thm:f.d.Hopf-H16}]\label{thm:thm:f.d.Hopf-H16}
Assume that $A$ is a finite-dimensional Hopf algebra over $\cH_{4,-1}$ and the diagram of $A$ admits non-trivial quadratic relations. Then $A\cong\gr A$.
\end{thm}

Theorem \ref{thm:thm:f.d.Hopf-H16} 
shows that $A$ is basic.

The organization of the paper is the following.
In section \ref{Preliminary}, some basic notations about Yetter-Drinfeld modules, Nichols algebras and other  materials that we will need are recalled.
In section  \ref{secPresentation}, the definitions of dual Radford algebras $\cH_{p,-1}$ and the category ${}_{\cH_{p,-1}}^{\cH_{p,-1}}\mathcal{YD}$ are given.
All finite-dimensional Nichols algebras in ${}_{\cH_{p,-1}}^{\cH_{p,-1}}\mathcal{YD}$ are determined in section \ref{secNicholsalg}. Consequently, in the final section \ref{secHopfalgebra}, all finite-dimensional Hopf algebras over $\cH_{p,-1}$ for a prime $p>5$ are determined and classified.

\section{Preliminaries}\label{Preliminary}
\subsection*{Conventions.} Throughout the paper, $\K$ is an algebraically closed field of characteristic zero,  $p$ is a natural number,  $\xi$
 is a primitive $2p$-th root of unity, $\lambda=\frac{\xi-1}{\xi+1}$; we denote  $\Z_n:=\Z/n\Z$,  $\I_{k,n}:=\{k,k+1,\ldots,n\}$ for $n\geq k\geq 0$, by $G_n$ the set of $n$-th root of unity, and by $C_n$ the cyclic group of order $n$.

  For a Hopf algebra $H$, we have the group of group-likes $\G(H)=\{g\in H-\{0\}\mid \Delta(g)=g\otimes g\}$ and the skew $(g,h)$-primitive vector spaces $\Pp_{g,h}(H)=\{x\in H\mid \De(x)=x\otimes g+h\otimes x\}$ for any   $g,h\in\G(H)$, where any nonzero element $v\in \Pp(H):=\Pp_{1,1}(H)$ is said to be primitive.

\subsection{Yetter-Drinfeld modules and bosonizations}
Suppose $H$ has a bijective antipode and denote by  ${}^{H}_{H}\mathcal{YD}$ the category of left Yetter-Drinfeld modules over $H$. It is known that ${}^{H}_{H}\mathcal{YD}$ (\cite{M93}, \cite{R11}) is rigid braided monoidal with the braiding $c_{V,W}$ for $V,W\in\HYD$  given by
\begin{align}\label{equbraidingYDcat}
c_{V,W}:V\otimes W\longrightarrow W\otimes V,\quad v\otimes w\mapsto v_{(-1)}\cdot w\otimes v_{(0)},\quad  (v\in V,\ w\in W).
\end{align}
For $V\in\HYD$, the left dual $V\As$ is defined by
\begin{align*}
\langle h\cdot f,v\rangle=\langle f, S(h)v\rangle,\quad f_{(-1)}\langle f_{(0)}, v\rangle=S^{-1}(v_{(-1)})\langle f, v_{(0)}\rangle.
\end{align*}

When $\dim H<\infty$, we know that as braided monoidal categories, $\HYD\cong{}_{H^{\ast}}^{H^{\ast}}\mathcal{YD}$ with the braided equivalence functor $(F,\eta)$ defined by:
 $F(V)=V$ as a vector space,
\begin{align}
\begin{split}\label{eqVHD}
f\cdot v=f(S(v_{(-1)}))v_{(0)},\quad \delta(v)=\sum_{i}S^{-1}(h^i)\otimes h_i\cdot v,\  \text{and}\\
\eta:F(V)\otimes F(W)\longrightarrow F(V\otimes W),\quad  v\otimes w\mapsto w_{(-1)}\cdot v\otimes w_{(0)}
\end{split}
\end{align}
where $\{h_i\}$ and $\{h^i\}$ are the dual bases of $H$ and $H\As$, and $f\in H^{\ast},\, v\in V,\, w\in W$ for $V,\, W\in\HYD$.
See  \cite[Proposition\,2.2.1.]{AG99} for details.

For a Hopf algebra $R\in{}^{H}_{H}\mathcal{YD}$ that is braided, set $\Delta_R(r)=r^{(1)}\otimes r^{(2)}$ for the comultiplication. By the \emph{Radford biproduct or bosonization} of $R$ by $H$ (\cite{R11}), written as $R\sharp H$, means a usual Hopf algebra, as a vector space,
$R\sharp H=R\otimes H$, whose multiplication and comultiplication are provided by the smash product and smash coproduct, respectively:
\begin{align}
(r\sharp g)(s\sharp h)=r(g_{(1)}\cdot s)\sharp g_{(2)}h,\quad
\Delta(r\sharp g)=r^{(1)}\sharp (r^{(2)})_{(-1)}g_{(1)}\otimes (r^{(2)})_{(0)}\sharp g_{(2)}.\label{eqSmash}
\end{align}

\subsection{Braided vector spaces and Nichols algebras}
 For any $(V, c)$ a braided vector space, the tensor algebra $T(V)=\bigoplus_{n\geq 0}T^n(V):=\bigoplus_{n\geq 0}V^{\otimes n}$ is a connected $\N$-graded braided Hopf algebra.
For any $x, y \in T(V)$, their braided commutator is
\begin{align*}
[x,y]_c=xy-m_{T(V)}\cdot c(x\otimes y).
\end{align*}

Denote by $\mathbb B_n$  the braid group generated by generators $(\tau_j)_{j \in \I_{1,n-1}}$
with relations
\begin{align}\label{eq:braid-rel}
\tau_i\tau_j = \tau_j\tau_i,\quad \tau_i  \tau_{i+1}\tau_i = \tau_{i+1}\tau_i\tau_{i+1}, \quad\text{ for } i\in\I_{1,n-2},
~j\not\in\{i-1,i, i+1\}.
\end{align}

It is well-known that there is a natural representation $\varrho_n$ of
$\mathbb B_n$ on $T^n(V)$ for $n \geq 2$ given by $$\varrho_n: \tau_j \mapsto c_j:=\id_{V^{\otimes (j-1)}} \otimes c \otimes \id_{V^{\otimes (n - j-1)}}.$$

Let  $M_n: \mathbb{S}_n \rightarrow \mathbb B_n$ be the (set-theoretical)
Matsumoto section, that preserves the length and satisfies $M_n(s_j) = \tau_j$.
The \emph{quantum symmetrizer} $\Omega_n:V^{\otimes n}\rightarrow V^{\otimes n}$ is then defined by
\begin{align*}
 \Omega_n = \sum_{\tau \in \mathbb{S}_n} \varrho_n (M_n(\tau)).
\end{align*}

\begin{defi}\cite[p.16-22]{AS02}\label{defi-Nicholsalgebra}
For $(V, c)$ a braided vector space, the Nichols algebra $\BN(V)$ is defined by
\begin{align}
\BN(V) =  T(V) / \J(V),\quad \text{where }\J(V) = \bigoplus_{n\geq 2}\J^n(V) \text{ and }\J^n(V) = \ker \Omega_n.
\end{align}
\end{defi}
Indeed, $\J(V)$ coincides with the largest homogeneous ideal of $T(V)$ generated by elements of degree bigger than and equal to $2$ that is also a coideal.

Recall that any object in the category of Yetter-Drinfeld modules is a braided vector space.
\begin{pro}\cite[Theorem 5.7]{T00}\label{pro-Nichols-YD-Realization}
Let $(V,c)$ be a rigid braided vector space. Then  $\BN(V)$ can be realized as a braided Hopf algebra in $\HYD$ for some Hopf algebra $H$.
\end{pro}

\begin{rmk}
 By Definition \ref{defi-Nicholsalgebra} and \cite[Theorem 5.7]{T00},
$\BN(V)$ depends only on $(V, c_{V,V})$, and as braided vector space, it might be realized in $\HYD$ for some non-unique $H$'s.
\end{rmk}

\begin{rmk}\label{rmkN-infity}
Suppose $W$ is a subspace of $(V, c)$ with $c(W\otimes W)\subset W\otimes W$, then $\dim \BN(W)=\infty$ means $\dim \BN(V)=\infty$. In particular, $\dim\BN(V)=\infty$ if the braiding $c$ has an eigenvector $v\otimes v\in V^{\otimes 2}$ with eigenvalue $1$ $($cf. \cite{G00}$)$.
\end{rmk}

For $m\in\N$ and $i\in\I_{0,m}$, let $\Delta^{i,m-i}:T^m(V)\rightarrow T^{i}(V)\otimes T^{m-i}(V)$ be the $(i,m-i)$-homogeneous component of the comultiplication $\Delta_{T(V)}$ of $T(V)$. The \emph{skew-derivation} $\partial_f\in\text{End}\,T(V)$ for  $f\in V^{\ast}$ is defined by
\begin{gather}\label{eqSkew-1}
\partial_f(v)=(f\otimes \id)\Delta^{1,m-1}(v):T^m(V)\longrightarrow T^{m-1}(V),\quad \textit{for \ } v\in T^m(V).
\end{gather}
This is useful for seeking the relations of $\BN(V)$ due to:
\begin{align}\label{Def-Nichols-IV}
\J^m(V)=\{r\in T^m(V)\mid  \partial_{f_1}\partial_{f_2}\cdots\partial_{f_m}(r)=0,\ \forall\ f_i\in V\As\}.
\end{align}
Furthermore, $\partial_f$ can be defined on $\BN(V)$ and $\mathop\cap\limits_{f\in V\As}\ker\partial_f=\K$. For details, see \cite{AS02,AHS10}.

The following result gives a relationship between $\BN(V)$ and $\BN(V\As)$ in ${}^{H}_{H}\mathcal{YD}$.
\begin{pro}\cite[Proposition\,3.2.30]{AG99}\label{proNicholsdual}
For any $V \in {}_H^H\mathcal{YD}$, when $\dim\BN(V)<\infty$, $\BN(V\As)\cong \BN(V)^{\ast\,bop}$, where
$\text{bop}$ means taking the opposite bialgebra structure.
\end{pro}

\subsection{Cocycle deformations}
 Assume that $(H,m,1,\Delta,\epsilon,S)$ is a Hopf algebra with bijective antipode, and $R\in {}^{H}_{H}\mathcal{YD}$ a Hopf algebra. Write $\rightharpoonup: H\otimes R\rightarrow R$ the action of $H$ on $R$.

A convolution invertible linear map $\sigma : H\otimes H\rightarrow k$ is called a {\em normalized Hopf  $2$-cocycle}, if
\begin{align*}
\sigma(g_{(1)},h_{(1)})\sigma(g_{(2)}h_{(2)},l)=\sigma(h_{(1)},l_{(1)})\sigma(g,h_{(2)}l_{(2)}), \quad
\sigma(h,1)=\varepsilon(h)=\sigma(1,h),\quad \forall g,h,l\in H.
\end{align*}

Given a $2$-cocycle $\sigma : H\otimes H\rightarrow k$. Denote by $\sigma^{-1}$ the convolution inverse of $\sigma$. We can construct a new Hopf algebra $(H^{\sigma},m^{\sigma},1,\Delta,\epsilon, S^{\sigma})$, where $H^{\sigma}=H$ as coalgebras, and
\begin{align*}
m^{\sigma}(a\otimes b)=\sum\sigma(a_{(1)},b_{(1)})a_{(2)}b_{(2)}\sigma^{-1}(a_{(3)},b_{(3)}),\quad\forall a,b\in H,\\
S^{\sigma}(a)=\sum\sigma^{-1}(a_{(1)},S(a_{(2)}))S(a_{(3)})\sigma(S(a_{(4))},a_5),\quad\forall a\in H.
\end{align*}

\begin{defi}
The Hopf algebra $H^{\sigma}$  is called the  \emph{cocycle deformation} of $H$ by $\sigma$.
\end{defi}

\begin{pro}\cite[Theorem\, 2.7]{MO99}.
\label{mo99}
Assume that $H$ is a Hopf algebra with bijective antipode, and $\sigma : H\otimes H\rightarrow k$ is a Hopf $2$-cocycle. Then
 ${}^{H}_{H}\mathcal{YD}\cong{}^{H^{\sigma}}_{H^{\sigma}}\mathcal{YD}$ as braided monoidal categories via the tensor functor $(G, \gamma)$ defined as follows: $G(V)=V$ as vector spaces and coactions,  transforming the action $\cdot$ to
\begin{align}
\begin{split}\label{formulaecocycle}
  h\rightharpoonup_{\sigma}v=\sigma(h_{1},v_{-2})\sigma^{-1}(h_{2}v_{-1}S(h_{4}),h_{5})h_{3}\rightharpoonup v_{0}, \quad h\in H^{\sigma}, v\in V,\text{ and}\\
  \gamma:G(V)\otimes G(W)\longrightarrow G(V\otimes W),\quad v\otimes w\mapsto\sigma(v_{-1},w_{-1})v_{0}\otimes w_{0},\quad V,W\in{}^{H}_{H}\mathcal{YD}.
\end{split}
\end{align}
\end{pro}

\section{Dual Radford algebra $\cH_{p,-1}$ and the category ${}_{\cH_{p,-1}}^{\cH_{p,-1}}\mathcal{YD}$}\label{secPresentation}
Let $p$ denote a natural number. We shall describe the dual Radford algebra $\cH_{p,-1}$ and the category ${}_{\cH_{p,-1}}^{\cH_{p,-1}}\mathcal{YD}$.

\begin{defi}\label{proStrucOfC}
Assume that $\xi$ is a primitive $2p$-th root of unity and $\Lam=\frac{\xi-1}{\xi+1}$. The Hopf algebra $\cH_{p,-1}$ is the algebra generated by elements $a$ and $b$ satisfying the relations
\begin{align*}
a^{2p}=1, \quad b^2=0,\quad ba=\xi ab,
\end{align*}
with the coalgebra structure defined by
\begin{align}\label{eqDef}
  \De(a)=a\otimes a+ \Lam^{-1}b\otimes ba^p, \quad \epsilon(a)=1,\quad
 \De(b)=b\otimes a^{p+1}+a\otimes b,\quad\epsilon(b)=0.
 \end{align}
and the antipode  given by $$S(a)=a^{2p-1}, \quad S(b)=\xi^{p+1}ba^{p-2}.$$
\end{defi}
It is straightforward to verify that $\cH_{p,-1}$ is well-defined and generated by a simple subcoalgebra $C:=\K.a\bigoplus\K.b\bigoplus\K.a^{p+1}\bigoplus\K.ba^p$. In particular, $\cH_{p,-1}$ has no the dual Chevalley property.
\begin{rmk}\label{rmkDDDDD}
\begin{enumerate}
\item $\G(\cH_{p,-1})=\{1, a^p\}$, $\Pp_{1,a^p}(\cH_{p,-1})=\K.(1{-}a^p)\bigoplus\K.ba^{p-1}$.
\item $\{a^i,ba^i\mid i\in\I_{0,2p-1}\}$ is a linear basis of $\cH_{p,-1}$.
\item $\{(a^i)^{\ast},(ba^i)^{\ast}\mid i\in\I_{0,2p-1}\}$ is the dual basis of $\cH_{p,-1}^{\ast}$ associated to the basis of $\cH_{p,-1}$ given in $(2)$.
From the multiplication table induced by the relations of $\cH_{p,-1}$, we have
\begin{align*}
\Delta(\widetilde{x})=\widetilde{x}\otimes \epsilon+\widetilde{g}\otimes \widetilde{x},\quad
\Delta(\widetilde{g})=\widetilde{g}\otimes \widetilde{g},
\end{align*}
where $\widetilde{x}=\sum_{i=0}^{2p-1}(ba^i)\As$, $\widetilde{g}=\sum_{i=0}^{2p-1}\xi^{-i}(a^i)\As$.
\item Let $\alpha\in \G(\cH_{p,-1}\As)$. Since $a^{2p}=1,b^2=0,ba=\xi ab$, it follows that $\alpha(a)$ is a $2p$th root of unity and $\alpha(b)=0$. Thus
    \begin{align*}
    \G(\cH_{p,-1}\As)=\left\{\alpha_i=\sum_{j=0}^{2p-1}\xi^{-ij}(a^j)\As\right\}.
  \end{align*}
   Note that $\alpha_0=\epsilon$, $\alpha_i=(\alpha_1)^i$, and $\G(\cH_{p,-1}\As)\simeq C_{2p}$ with the generator $\alpha_1$.
\end{enumerate}
\end{rmk}

Let $\A_{p,-1}$ denote a Radford algebra of dimension $4p$ defined by
\begin{align*}
\A_{p,-1}:=\langle g,x\mid g^{2p}=1, x^2=1{-}g^2, gx=-xg\;\rangle;\quad  \De(g)=g\otimes g, \quad\De(x)=x{\otimes} 1+g{\otimes} x.
\end{align*}
It is clear that $\gr\A_{p,-1}=\BN(X)\sharp \K[\Gamma]$, where $\Gamma\cong C_{2p}$ with the generator $g$ and $X:=\K.x\in{}_{\Gamma}^{\Gamma}\mathcal{YD}$  by $g\cdot x=-x$ and $\delta(x)=g\otimes x$.  By \cite[Proposition\,4.2]{GM10}, $\gr\A_{p,-1}\cong (\A_{p,-1})^\sigma$ with the Hopf $2$-cocycle $\sigma$ given by
\begin{align}\label{eqSigama}
\sigma=\epsilon\otimes\epsilon-\zeta, \text{ where } \zeta(x^ig^j,x^kg^l)=(-1)^{jk}\delta_{2,i+k} \text{ for }  i,k\in\I_{0,1}, j,l\in\I_{0,2p-1}.
\end{align}
Hence ${}_{\gr \A_{p,-1}}^{\gr\A_{p,-1}}\mathcal{YD}\cong {}_{ \A_{p,-1}}^{\A_{p,-1}}\mathcal{YD}$ as braided monoidal categories via the tensor functor $(G,\gamma)$ given in Proposition \ref{mo99}. Now we show that $\A_{p,-1}$ is the dual Hopf algebra of $\cH_{p,-1}$, that is, $\cH_{p,-1}$ is a basic Hopf algebra.

\begin{lem}\label{lem:A-Dual-H}
 $\A_{p,-1}\cong \cH_{p,-1}\As$ and the Hopf algebra isomorphism $\phi:\A_{p,-1}\longrightarrow \cH_{p,-1}\As$ is given by
\begin{align*}
\phi(g^i)&=\sum_{j=0}^{2p-1}\xi^{-ij}(a^j)\As,\quad
\phi(g^ix)=\theta\sum_{j=0}^{2p-1}\xi^{-i(j+1)}(ba^j)\As,\;\text{where}\;\theta^2=(1-\xi^{-2})\lambda.
\end{align*}
\end{lem}
\begin{proof}
Observe that, by Remark \ref{rmkDDDDD}, $\psi$ is a coalgebra morphism and $\psi(\A_{p,-1})$ contains properly $\G(\cH_{p,-1}\As)$.  It is straightforward to verify that $\psi$ is a bialgebra morphism by using the multiplication table induced by the defining relations of $\cH_{p,-1}$.  Since $\dim\G(\cH_{p,-1}\As)=2p$,  by the Nichols-Zoeller Theorem,  $\psi$ is an epimorphism. The Lemma follows from $\dim \A_{p,-1}=\dim \cH_{p,-1}\As=4p$.
\end{proof}
\begin{rmk}
In what follows, we set $\theta^{-2}:=(1-\xi^{-2})^{-1}\lambda^{-1}=(1-\xi^{-2})^{-1}(\xi-1)^{-1}(1+\xi)$.
\end{rmk}
Recall that the Drinfeld double $\D(\cH_{p,-1}^{cop})$ is a Hopf algebra with the tensor product coalgebra structure and the algebra structure given by
$$
(r\otimes a)(s\otimes b)=r\langle s_{(3)}, a_{(1)}\rangle s_{(2)}\otimes a_{(2)}\langle s_{(1)}, S^{-1}(a_{(3)})\rangle b
$$
for $r,s\in\A_{p,-1},a,b\in\cH_{p,-1}$. In this section, we write  $\D:=\D(\cH_{p,-1}^{cop})$ for short.

\begin{pro}
$\D$ is isomorphic as a coalgebra to the tensor coalgebra $\A_{p,-1}^{bop}\otimes \cH_{p,-1}^{cop}$, and is generated as an algebra by elements $a, b, g, x$ satisfying the relations in $\cH_{p,-1}^{cop}$, the relations in $\A_{p,-1}^{bop}$ and
\begin{align}\label{eqDD-1}
\begin{split}
ag&=ga,\quad ax-\xi xa=\Lam^{-1}\theta\xi^{p+1}(ba^p-gb),\\
bg&=-gb,\quad bx-\xi xb=\theta\xi^{p+1}(a^{p+1}-ga).
\end{split}
\end{align}
\end{pro}

\subsection{Category ${}_{\D(\cH_{p,-1}^{cop})}\mathcal{M}$}We first determine simple and some indecomposable  objects in ${}_{\D}\mathcal{M}$.
\begin{lem} \label{lemDonesimple}
Let $i\in\I_{0,2p-1}$ and $\chi$ be an irreducible character of the cyclic group $C_{2p}$. Assume that $\K_{\chi^{i}}$ is the one-dimensional left $\D$-module defined by
\begin{align*}
\chi^i(a)=\xi^i, \quad \chi^i(b)=0,\quad \chi^i(g)=(-1)^i,\quad \chi^i(x)=0,
\end{align*}
then for any one-dimensional $\D$-module $\K_{\lambda}$, there is $j\in\I_{0,2p-1}$ such that $\K_{\lambda}\cong\K_{\chi^{j}}$. Furthermore, for $i,j\in\I_{0,2p-1}$, $\K_{\chi^i}\cong\K_{\chi^j}$ if and only if $i=j$.
\end{lem}
\begin{proof}
It is straightforward to verify that $\chi^i$ ($i\in\I_{0,2p-1}$) is well-defined and $\K_{\chi^{i}}\cong\K_{\chi^{j}}$ if and only if $i-j\equiv 0\mod 2p$.  We claim that any one-dimensional $\D$-module is isomorphic to $\K_{\chi^i}$ for some $i\in\I_{0,2p-1}$.
Let $\lambda\in \G(\D\As)$. Since $a^{2p}=1=g^{2p}$, it follows that $\lambda(a)$ and $\lambda(g)$ are both $2p$th roots of unity.
From $b^2=0$, $gb=-bg$ and $gx=-xg$, we have  $\lambda(b)=0=\lambda(x)$. Since $x^2=1-g^2$, it follows that $\lambda(g)^2=1$. From
the relation $bx-\xi xb=\theta\xi^{p+1}(a^{p+1}-ga)$, we have $\lambda(a)^p=\lambda(g)$. Thus $\lambda$ is completely determined by $\lambda(a)$. Let $\lambda(a)=\xi^i$ for some $i\in \I_{0,2p-1}$. Then $\lambda=\chi^i$.
\end{proof}
\begin{lem}$\label{twosimple}$
Set $\Lambda_p=\{(i,j)\in \I_{0,2p-1}\times \I_{0,2p-1}\mid pi-j\not\equiv 0\mod 2p\}$. For any $(i,j)\in\Lambda_p$, there exists a simple $\D$-module $V_{i,j}$ of dimension $2$ whose matrices defining the $\D$-action with respect to a fixed basis are given by
\begin{align*}
    [a]_{i,j}&=\left(\begin{array}{ccc}
                                   \xi^i & 0\\
                                   0 &    \xi^{i+1}
                                 \end{array}\right),\quad
    [b]_{i,j}=\left(\begin{array}{ccc}
                                   0 & 1\\
                                   0 & 0
                                 \end{array}\right),\quad
    [g]_{i,j}=\left(\begin{array}{ccc}
                                   \xi^j & 0\\
                                   0 & -\xi^j
                                 \end{array}\right),\\
    [x]_{i,j}&=\left(\begin{array}{ccc}
                                   0 & \theta^{-1}\xi^{p-1-i}((-1)^i+\xi^j)\\
                                   \theta\xi^{p+1+i}((-1)^i-\xi^j) & 0
                                 \end{array}\right).
\end{align*}
Moreover, any simple $\D$-module of dimension $2$ is isomorphic to $V_{i,j}$ for some $(i,j)\in\Lambda_p$;  $V_{i,j}\cong V_{r,s}$ if and only if $(i,j)=(r,s)$.
\end{lem}
\begin{proof}
It is straightforward to verify that $V_{i,j}$ is well-defined. Now we claim that any simple $\D$-module $V$ of dimension $2$ is isomorphic to $V_{i,j}$ for some $(i,j)\in\Lambda_p$.
Since $a^{2p}=g^{2p}=1$ and $ga=ag$, we can choose a basis of $V$ such that the matrices defining the $\D$-action on $V$ are of the form
\begin{align*}
    [a]&=\left(\begin{array}{ccc}
                                   a_1 & 0\\
                                    0  & a_2
                                 \end{array}\right),\quad
    [b]=\left(\begin{array}{ccc}
                                   b_1 & b_2\\
                                   b_3 & b_4
                                 \end{array}\right),\quad
    [g]=\left(\begin{array}{ccc}
                                   g_1 & 0\\
                                   0   & g_2
                                 \end{array}\right),\quad
    [x]=\left(\begin{array}{ccc}
                                   x_1 & x_2\\
                                   x_3 & x_4
                                 \end{array}\right),
\end{align*}
where $a_1^{2p}=1=a_2^{2p}$ and $g_1^{2p}=1=g_2^{2p}$. From the relation $gb=-bg$, we have
\begin{align*}
        \left(\begin{array}{ccc}
                                   g_1b_1 & g_1b_2\\
                                   g_2b_3 & g_2b_4
                                 \end{array}\right)=-
        \left(\begin{array}{ccc}
                                   b_1g_1 & b_2g_2\\
                                   b_3g_1  & b_4g_2
                                 \end{array}\right).
\end{align*}
Hence $2g_1b_1=0=2g_2b_4$ and $(g_1+g_2)b_3=0=(g_1+g_2)b_2$, which implies that $b_1=0=b_4$ since $g_1, g_2\neq 0$. Similarly, from the relation $gx=-xg$, we have $x_1=0=x_4$ and $(g_1+g_2)x_3=0=(g_1+g_2)x_2$. If $g_1+g_2\neq 0$, then $b_3=0=b_2,\,x_3=0=x_2$, that is, $[x], [b]$ are zero matrices, which implies $V$ is not simple $\D$-module, a contradiction. Hence $g_1=-g_2$.

From the relation $b^2=0$, we have $b_2b_3=0$. By permuting the elements of the basis, we may assume $b_3=0$. If $b_2=0$, it is clear that $V$ is simple if and only if $x_2x_3\neq 0$. From the relation $ax-\xi xa=\Lam^{-1}\theta\xi^{p+1}(ba^p-gb)$, we have
\begin{align*}
(a_2-\xi a_1)x_3=0,\quad (a_1-\xi a_2)x_2=\Lam^{-1}\theta\xi^{p+1}(a_2^p-g_1)b_2,
\end{align*}
which implies $a_2=0=a_1$, a contradiction. Therefore, $b_2\neq 0$. By rescaling the basis, we assume that $b_2=1$.

The relation $ba=\xi ab$ implies $a_2=\xi a_1$, and the relation $x^2=1-g^2$ implies $x_2x_3=1-g_1^2$. From the relation $bx-\xi xb=\theta\xi^{p+1}(a^p-g)a$, we have
\begin{align*}
        \left(\begin{array}{ccc}
                                   x_3 & 0\\
                                  0 & \xi^{p+1}x_3
                                 \end{array}\right)=\theta\xi^{p+1}
        \left(\begin{array}{ccc}
                                   a_1^{p+1}-g_1a_1 & 0\\
                                   0  & a_2^{p+1}-g_2a_2
                                 \end{array}\right),
\end{align*}
which implies $x_3=\theta\xi^{p+1}(a_1^p-g_1)a_1$. Indeed, $x_3=\theta(a_2^{p+1}-g_2a_2)\Leftrightarrow x_3=\theta\xi^{p+1}(a_1^p-g_1)a_1$ since $g_1=-g_2$ and $a_2=\xi a_1$. Thus we have $x_2=\theta^{-1}\xi^{p-1}a_1^{-1}(a_1^p+g_1)$ since $x_2x_3=1-g_1^2=(a_1^p+g_1)(a_1^p-g_1)$.

From the discussion above, the matrices defining the action on $V$ are of the form
\begin{align*}
    [a]_{i,j}&=\left(\begin{array}{ccc}
                                   \Lam_1 & 0\\
                                   0 &    \xi\Lam_1
                                 \end{array}\right),\quad
    [b]_{i,j}=\left(\begin{array}{ccc}
                                   0 & 1\\
                                   0 & 0
                                 \end{array}\right),\quad
    [g]_{i,j}=\left(\begin{array}{ccc}
                                   \Lam_2 & 0\\
                                   0 & -\Lam_2
                                 \end{array}\right),\\
    [x]_{i,j}&=\left(\begin{array}{ccc}
                                   0 & \theta^{-1}\xi^{p-1}\Lam_1^{-1}(\Lam_1^p+\Lam_2)\\
                                   \theta\xi^{p+1}\Lam_1(\Lam_1^p-\Lam_2) & 0
                                 \end{array}\right),
\end{align*}
with $\Lam_1^{2p}=1=\Lam_2^{2p}$. It is clear that $V$ is simple if and only if $\Lam_1^p-\Lam_2\neq 0$.
If we set $\Lam_1=\xi^i$ and $\Lam_2=\xi^j$ for some $i, j\in \I_{0,2p-1}$,
then $pi- j\not\equiv 0\mod 2p$, that is, $(i,j)\in\Lambda_p$. Hence the claim follows.

Now we claim that $V_{i,j}\cong V_{k,l}$ if and only if  $(i,j)=(k,l)$ in $C_{2p}\times C_{2p}$.
Suppose that $\Psi:V_{i,j}\longrightarrow V_{k,l}$ is a $\D$-module isomorphism. Denote by
$[\Psi]=(p_{i,j})_{i,j=1,2}$ the matrix of $\Psi$ in the given basis. As a module morphism, we have $[b][\Psi]=[\Psi][b]$ and $[a][\Psi]=[\Psi][a]$, which imply $p_{21}=0,\,p_{11}=p_{22}$ and $(\xi^k-\xi^i)p_{11}=0,\,(\xi^k-\xi^{i+1})p_{12}=0$. Thus we have $\xi^i=\xi^k$, which yields $p_{12}=0$ since $\Psi$ is an isomorphism.
Similarly, we have $\xi^j=\xi^l$, and hence the claim follows.
\end{proof}

\begin{rmk}\label{rmkDmoddual}
For a left $\D$-module $V$, there exists a left dual module $V\As$ with the module structure given by
$(h\rightharpoonup f)(v)=f(S(h)\cdot v)$ for all $h\in \D, v\in V, f\in V\As$. A direct computation shows that $V_{i,j}\As\cong V_{-i-1,-j-p}$ for all $(i,j)\in \Lambda_p$.
\end{rmk}

Recall that if $V$ is a simple $\D$-module, then $\Pp(V)$ is unique $($up to isomorphism$)$ indecomposable projective $\D$-module, which maps onto $V$. Denote by $\text{Irr}(\D)$  the set of isomorphism classes of simple $\D$-modules. Then
$\D\cong \bigoplus_{V\in\text{Irr}(D)}\Pp(V)^{\dim V}$. See e.g. \cite{ARS95} for details.

\begin{thm}\label{thmsimplemoduleD}
Up to isomorphism, there exist $4p^2$ simple left $\D$-modules, among which $2p$ one-dimensional modules are given in Lemma
$\ref{lemDonesimple}$ and $4p^2-2p$ two-dimensional simple modules are given in Lemma $\ref{twosimple}$.
\end{thm}
\begin{proof}
Assume that there exists at least one simple module of dimension $x$ bigger than $2$. Let $n$ be the amount of simple modules (up to isomorphism) of dimension $x$. By Lemmas \ref{lemDonesimple} and   $\ref{twosimple}$,
\begin{align*}
2p\times 1^2+(4p^2-2p)\times 2^2+nx^2=16p^2-6p+nx^2<\dim \D\As=16p^2.
\end{align*}
Then $nx^2<6p$. By \cite[Lemma 2.1]{AN01}, $4p=|\G(\D)|$ divides $nx^2$, a contradiction.
\end{proof}

\subsection{Category ${}_{\cH_{p,-1}}^{\cH_{p,-1}}\mathcal{YD}$}
We shall describe directly the braidings of the simple objects in ${}_{\cH_{p,-1}}^{\cH_{p,-1}}\mathcal{YD}$ by using the braided equivalence ${}_{\cH_{p,-1}}^{\cH_{p,-1}}\mathcal{YD}\cong {}_{\D}\mathcal{M}$ as braided tensor categories.
\begin{pro}\label{proYD-1}
If $\K_{\chi^i}=\K v \in{}_{\D}\mathcal{M}$ for $ i\in \I_{0,2p-1}$, then $\K_{\chi^i}\in{}_{\cH_{p,-1}}^{\cH_{p,-1}}\mathcal{YD}$ with the Yetter-Drinfeld module structure given by
\begin{align*}
a\cdot v=\xi^i v,\quad b\cdot v=0,\quad \delta(v)=a^{pi}\otimes v.
\end{align*}
\end{pro}

\begin{proof}
Since $\K_{\chi^i}$ is a one-dimensional $\D$-module with $i\in \I_{0,2p-1}$, the $\cH_{p,-1}$-action must be given by the restriction of the character of $\D$ given by Lemma \ref{lemDonesimple} and the coaction must be of the form $\delta(v)=h\otimes v$, where $h\in \G(\cH_{p,-1})=\{1,a^p\}$ such that $\langle g, h\rangle v=(-1)^iv$. It follows that the action is given by $a\cdot v=\xi^i v,\, b\cdot v=0$ and the coaction is given by $\delta(v)=a^{pi}\otimes v$.
\end{proof}

\begin{pro}\label{proYD-2}
If $V_{i,j}=\K v_1\bigoplus\K v_2 \in{}_{\D}\mathcal{M}$ for $(i,j)\in\Lambda_p$, then $V_{i,j}\in{}_{\cH_{p,-1}}^{\cH_{p,-1}}\mathcal{YD}$ with the Yetter-Drinfeld module structure  given by
\begin{gather*}
a\cdot v_1=\xi^iv_1,\quad b\cdot v_1=0,\quad a\cdot v_2=\xi^{i+1} v_2,\quad b\cdot v_2=v_1,\\
\delta(v_1)=a^{-j}\otimes v_1+x_2\theta^{-1}ba^{-1-j}\otimes v_2,\quad
\delta(v_2)=a^{p-j}\otimes v_2+x_1\theta^{-1}ba^{p-j-1}\otimes v_1,
\end{gather*}
where $x_1=\theta^{-1}\xi^{p-1-i}((-)^i+\xi^j)$ and $x_2=\theta\xi^{p+1+i}((-1)^i-\xi^j)$.
\end{pro}
\begin{proof}
The $\cH_{p,-1}$-action is given by the restriction of the $\D$-action given in Lemma \ref{twosimple} and the comodule structure is given by $\delta(v)=\sum_{i=1}^{4p}c_i\otimes c^i\cdot v$ for $v\in V_{i,j}$, where $\{c_i\}_{1\leq i\leq 4p}$ and $\{c^i\}_{1\leq i\leq 4p}$ are the dual bases of $\cH_{p,-1}$ and $\cH_{p,-1}\As$.
Note that by Remark $\ref{rmkDDDDD}$, we have
\begin{align*}
(g^i)\As=\frac{1}{2p}\left(\sum_{j=0}^{2p-1}\xi^{ij}a^j\right),\quad
(g^ix)\As=\frac{1}{2p\theta}\left(\sum_{j=0}^{2p-1}\xi^{i(j+1)}ba^j\right).
\end{align*}
We write $\lambda_1=\xi^i, \lambda_2=\xi^j$. Then the comodule structure is given as follows:
\begin{align*}
\delta(v_1)&=\sum_{k=0}^{2p-1}(g^k)\As\otimes g^k\cdot v_1+\sum_{k=0}^{2p-1}(g^kx)\As\otimes g^kx\cdot v_1
=a^{2p-j}\otimes v_1+x_2\theta^{-1}ba^{2p-1-j}\otimes v_2,\\
\delta(v_2)&=\sum_{k=0}^{2p-1}(g^k)\As\otimes g^k\cdot v_2+\sum_{k=0}^{2p-1}(g^kx)\As\otimes g^kx\cdot v_2
=a^{p-j}\otimes v_2+x_1\theta^{-1}ba^{p-j-1}\otimes v_1,
\end{align*}
where $x_1=\theta^{-1}\xi^{p-1}\Lam_1^{-1}(\Lam_1^p+\Lam_2)$ and $x_2=\theta\xi^{p+1}\Lam_1(\Lam_1^p-\Lam_2)$.
\end{proof}

Now using formula \eqref{equbraidingYDcat} in ${}_{\cH_{p,-1}}^{\cH_{p,-1}}\mathcal{YD}$,  we describe the braidings of simple objects  in ${}_{\cH_{p,-1}}^{\cH_{p,-1}}\mathcal{YD}$.
\begin{pro}\label{braidingone}
The braiding of the one-dimensional Yetter-Drinfeld module $\K_{\chi^i}=\K v $ is $c(v\otimes v)=(-1)^iv\otimes v$.
\end{pro}

\begin{pro}\label{probraidsimpletwo}
If $V_{i,j}=\K v_1\bigoplus\K v_2 \in{}_{\cH_{p,-1}}^{\cH_{p,-1}}\mathcal{YD}$ for $(i,j)\in\Lambda_p$, then the braiding of $V_{i,j}$ is given by

  \begin{align*}
   c(\left[\begin{array}{cc} v_1\\v_2\end{array}\right]{\otimes}\left[\, v_1~v_2\,\right])=
   \left[\begin{array}{cc}
   \xi^{-ij}v_1{\otimes} v_1    & \xi^{-j(i+1)}v_2{\otimes} v_1+[\xi^{-ij}{+}\xi^{(i+1)(p-j)}]v_1{\otimes} v_2\\
   (-1)^i\xi^{-ij}v_1{\otimes} v_2   & \xi^{(i+1)(p-j)}v_2{\otimes} v_2+\theta^{-2}\xi^{-(i+1)(2+j)}(1{+}\xi^{pi+j})v_1{\otimes} v_1
         \end{array}\right].
  \end{align*}
\end{pro}

\begin{rmk}
The braided vector spaces $V_{i,j}$ with $(i,j)\in\Lambda_p$ have already appeared in \cite{Hi93} $($also in \cite{AGi17}$)$. More precisely,
the braided vector spaces $V_{i,0}$ and $V_{i,p}$  belong to the case $\mathfrak{R}_{2,1}$ in \cite{Hi93,AGi17}, and the others belong to the case $\mathfrak{R}_{1,2}$.
\end{rmk}

\subsection{Category ${}_{\gr\A_{p,-1}}^{\gr\A_{p,-1}}\mathcal{YD}$} Let us describe the simple objects in ${}_{\gr\A_{p,-1}}^{\gr\A_{p,-1}}\mathcal{YD}$ via the braided equivalence ${}_{\cH_{p,-1}}^{\cH_{p,-1}}\mathcal{YD}\cong {}_{\gr\A_{p,-1}}^{\gr\A_{p,-1}}\mathcal{YD}$.
\begin{pro}\label{proVA1}
Suppose $\K_{\chi^i}=\K v \in{}_{\cH_{p,-1}}^{\cH_{p,-1}}\mathcal{YD}$. Then  $\K_{\chi^i}\in{}_{\gr\A_{p,-1}}^{\gr\A_{p,-1}}\mathcal{YD}$ with the Yetter-Drinfeld module structure given by
\begin{align*}
g\cdot v=(-1)^iv,\quad x\cdot v=0,\quad \delta(v)=g^i\otimes v.
\end{align*}
\end{pro}
\begin{proof}
Using formulae \eqref{eqVHD}, by Proposition \ref{proYD-1}, a direct computation shows that $\K_{\chi^i}\in{}_{\A_{p,-1}}^{\A_{p,-1}}\mathcal{YD}$ with the Yetter-Drinfeld module structure given by
\begin{align*}
g\cdot v=(-1)^iv,\quad x\cdot v=0,\quad \delta(v)=g^i\otimes v.
\end{align*}
Then the Proposition follows from a direct computation using formulae \eqref{formulaecocycle} and \eqref{eqSigama}.
\end{proof}
\begin{pro}\label{proVA2}
Suppose $V_{i,j}=\K v_1\bigoplus\K v_2 \in{}_{\cH_{p,-1}}^{\cH_{p,-1}}\mathcal{YD}$ for $(i,j)\in\Lambda_p$. Then
$V_{i,j}\in{}_{\gr\A_{p,-1}}^{\gr\A_{p,-1}}\mathcal{YD}$ with the Yetter-Drinfeld module structure given by
\begin{align*}
g\cdot v_1=\xi^{-j}v_1,\quad x\cdot v_1=x_2\xi^{p-j}v_2,\quad g\cdot v_2=\xi^{p-j}v_2,\quad x\cdot v_2=0;\\
\delta(v_1)=g^i\otimes v_1,\quad \delta(v_2)=g^{i+1}\otimes v_2+\theta^{-1}(-1)^{i+1}\xi^{-i-1}g^ix\otimes v_1,
\end{align*}
where $x_1=\theta^{-1}\xi^{p-1-i}((-1)^i+\xi^j)$ and $x_2=\theta\xi^{p+1+i}((-1)^i-\xi^j)$.
\end{pro}
\begin{proof}
Using formula \eqref{eqVHD}, by Proposition \ref{proYD-2}, a direct computation shows that $V_{i,j}\in{}_{\A_{p,-1}}^{\A_{p,-1}}\mathcal{YD}$ with the Yetter-Drinfeld module structure given by
\begin{align*}
g\cdot v_1=\xi^{-j}v_1,\quad x\cdot v_1=x_2\xi^{p-j}v_2,\quad g\cdot v_2=\xi^{p-j}v_2,\quad x\cdot v_2=x_1\xi^{-j}v_1;\\
\delta(v_1)=g^i\otimes v_1,\quad \delta(v_2)=g^{i+1}\otimes v_2+\theta^{-1}(-1)^{i+1}\xi^{-i-1}g^ix\otimes v_1.
\end{align*}
Then the Proposition follows from a direct computation using formulae \eqref{formulaecocycle} and \eqref{eqSigama}.
\end{proof}
\begin{rmk}
Suppose $V_{i,j}=\K v_1\bigoplus\K v_2 \in{}_{\gr\A_{p,-1}}^{\gr\A_{p,-1}}\mathcal{YD}$. Set  $e_1=v_1,\  e_2=x_2\xi^{p-j}v_2$. Since $(i,j)\in\Lambda$, it follows that $x_2\neq 0$ and hence $\{e_1,e_2\}$ is also a linear basis of $V_{i,j}$ in ${}_{\gr\A_{p,-1}}^{\gr\A_{p,-1}}\mathcal{YD}$ with the Yetter-Drinfeld module structure given by
\begin{align*}
g\cdot e_1=\xi^{-j}e_1,\quad x\cdot e_1=e_2,\quad g\cdot e_2=\xi^{p-j}e_2,\quad x\cdot e_2=0;\\
\delta(e_1)=g^i\otimes e_1,\quad \delta(e_2)=g^{i+1}\otimes e_2+((-1)^i-\xi^{-j})g^ix\otimes e_1.
\end{align*}
\end{rmk}

\begin{pro}
If $\K_{\chi^i}=\K v \in{}_{\gr\A_{p,-1}}^{\gr\A_{p,-1}}\mathcal{YD}$, then the braiding is $c(v\otimes v)=(-1)^iv\otimes v$.
\end{pro}

\begin{pro}\label{probraidsimpletwo-1}
If $V_{i,j}=\K v_1\bigoplus\K v_2 \in{}_{\gr\A_{p,-1}}^{\gr\A_{p,-1}}\mathcal{YD}$, then the braiding of $V_{i,j}$ is given by

  \begin{align*}
   c(\left[\begin{array}{ccc} v_1\\v_2\end{array}\right]\otimes\left[\,v_1~v_2\,\right])=
   \left[\begin{array}{ccc}
   \xi^{-ij}v_1\otimes v_1    & (-1)^i\xi^{-ij}v_2\otimes v_1\\
   \xi^{-j(i+1)}v_1\otimes v_2+[\xi^{-ij}+\xi^{(i+1)(p-j)}]v_2\otimes v_1   & \xi^{(i+1)(p-j)}v_2\otimes v_2
         \end{array}\right].
  \end{align*}
\end{pro}

\bigskip
\section{Nichols algebras in ${}_{\cH_{p,-1}}^{\cH_{p,-1}}\mathcal{YD}$}\label{secNicholsalg}
In this section, we shall determine all objects $V$ in ${}_{\cH_{p,-1}}^{\cH_{p,-1}}\mathcal{YD}$ such that  Nichols algebras $\BN(V)$ are finite-dimensional, under the assumption that $p$ is a prime number or $p=4$, which extends \cite[Theorem A]{X18} and partial results in \cite{BGGM}, \cite{X19}.

Let us begin with one-dimensional objects in ${}_{\cH_{p,-1}}^{\cH_{p,-1}}\mathcal{YD}$.
\begin{lem}\label{lemNicholbyone}
The Nichols algebra $\BN(\K_{\chi^k})$ over $\K_{\chi^k}=\K v $ for $k\in\I_{0,2n-1}$ is
\begin{align*}
\BN(\K_{\chi^k})=\begin{cases}
\K[v], &\text{~if~} k\text{~is even};\\
\bigwedge \K_{\chi^k}=\K_{\chi^k}, &\text{~if~} k\text{~is odd}.\\
\end{cases}
\end{align*}
Moreover, set $V=\bigoplus_{i\in I}V_i$, where $V_i\cong \K_{\chi^{k_i}}$, $k_i\in\I_{0,2n-1}$ is odd, and $I$ is a finite index set.
Then $\BN(V)=\bigwedge V\cong \bigotimes_{i\in I}\BN(V_i)$.
\end{lem}
\begin{proof}
The first claim follows immediately from Proposition $\ref{braidingone}$. Let $V_i=\K.v_{k_i}$ for $i\in I$ and odd number $k_i$, then by Proposition \ref{proYD-1}, $c(v_{k_i}\otimes v_{k_j})=(-1)^{k_ik_j}v_{k_j}\otimes v_{k_i}=-v_{k_j}\otimes v_{k_i}$, which implies that $\BN(V)=\bigwedge V$ and $c_{V^{\otimes 2}}^2=\id_{V^{\otimes2}}$. The last isomorphism follows from \cite[Theorem 2.2.]{G00}.
\end{proof}

From \cite[Theorem 1.2]{AA18b}, we have
\begin{pro}\label{proNicholsindecom}
 Suppose $V\in{}_{\cH_{p,-1}}^{\cH_{p,-1}}\mathcal{YD}$ with $\dim V<\infty$ is indecomposable but non-simple. Then $\dim\BN(V)=\infty$. Equivalently, if $\dim\BN(V)<\infty$, then $V$ is semisimple.
\end{pro}

\subsection{Nichols algebras over simple objects} We shall determine all simple objects $V$ in ${}_{\cH_{p,-1}}^{\cH_{p,-1}}\mathcal{YD}$ such that $\dim\BN(V)<\infty$ with    a prime $p$ or  $p=4$.

Recall that $\gr\A_{p,-1}=\BN(X)\sharp \K[\Gamma]$, where $\Gamma\cong C_{2p}$ with generator $g$ and $X:=\K.x\in{}_{\Gamma}^{\Gamma}\mathcal{YD}$ with the Yetter-Drinfeld module structure given by $g\cdot x=-x$ and $\delta(x)=g\otimes x$.

We begin this subsection by the following result.
\begin{lem}\cite[Theorem 2.7]{WZZ14}\label{lem-Nichols-finite-Zn-1}
Assume that $p$ is prime, $V\in{}_{C_{2p}}^{C_{2p}}\mathcal{YD}$  such that the associated generalized Dynkin diagram is connected. Assume that $\dim V=2$ and there exists nonzero $v\in V$ such that $c(v\otimes v)=-v\otimes v$. Then $\dim\BN(V)<\infty$, if and only if, the generalized Dynkin diagram of $V$ is equivalent to one of the following
\begin{enumerate}
  \item\cite[Table 1, Row 2]{H09} \xymatrix@C+15pt{\overset{-1 }{{\circ}}\ar@{-}[r]^{q^{-1}} & \overset{q}{{\circ}}};
        \xymatrix@C+15pt{\overset{-1 }{{\circ}}\ar@{-}[r]^{q} & \overset{-1}{{\circ}}};   where $q\in G_p$.
  \item\cite[Table 1, Row 4]{H09} \xymatrix@C+15pt{\overset{-1 }{{\circ}}\ar@{-}[r]^{q^{-2}} & \overset{q}{{\circ}}};
        \xymatrix@C+15pt{\overset{-1 }{{\circ}}\ar@{-}[r]^{q^2} & \overset{-q^{-1}}{{\circ}}};   where $q\in G_p$  and $p>2$.
  \item  \cite[Table 1, Row 6]{H09} \xymatrix@C+15pt{\overset{-1 }{{\circ}}\ar@{-}[r]^{-\varsigma} & \overset{\varsigma}{{\circ}}};
       \xymatrix@C+15pt{\overset{-1 }{{\circ}}\ar@{-}[r]^{-\varsigma^{-1}} & \overset{\varsigma^{-1}}{{\circ}}};   where $\varsigma\in G_3$.
  \item  \cite[Table 1, Row 13]{H09} \xymatrix@C+15pt{\overset{-1 }{{\circ}}\ar@{-}[r]^{\varsigma^2} & \overset{\varsigma}{{\circ}}};
       \xymatrix@C+15pt{\overset{-1 }{{\circ}}\ar@{-}[r]^{\varsigma^{-2}} & \overset{-\varsigma^{-2}}{{\circ}}};   where $\varsigma\in G_5$.
\end{enumerate}
\end{lem}

\begin{pro}\label{pro-Nichols-over-4p-1}
Assume that $p$ is prime, $V$ is a simple object in ${}_{\cH_{p,-1}}^{\cH_{p,-1}}\mathcal{YD}$. Then $\dim\BN(V)<\infty$, if and only if, $V$ is isomorphic to one of the objects:
\begin{itemize}
\item $\K_{\chi^{i}}$, for odd number $i\in\I_{0,2p-1}$;

\item \cite[Row 2(2), Table 1]{H09}\footnote{The generalized Dynkin diagram of $\BN(X\bigoplus X_{i,j})$} \ $V_{p,j}$, for odd number $j\neq p$;
\item \cite[Row 2(1), Table 1]{H09} \ $V_{p-1,j}$, for odd number $p-j\neq p$;
\item \cite[Row 4, Table 1]{H09} \ $V_{\frac{1}{2}(p-1),j}$, $p>2$. If $\frac{1}{2}(p-1)\equiv 0\mod 2$, then $j\neq 0$ is even, otherwise, $j\neq p$ is odd;
\item \cite[Row 4, Table 1]{H09} \ $V_{\frac{1}{2}(p-1)+p,j}$, $p>2$. If $\frac{1}{2}(p-1)\equiv 0\mod 2$, then $j\neq p$ is odd, otherwise, $j\neq 0$ is even;
\item \cite[Row 6, Table 1]{H09} \  $V_{1,j}$, for $j\in\{2,4\}$, $p=3$;

\item \cite[Row 6, Table 1]{H09} \  $V_{4,j}$, for $j\in\{1,5\}$, $p=3$;

\item \cite[Row 13(2), Table 1]{H09} \  $V_{1,j}$, for $j\in\{1,3,7,9\}$, $p=5$;

\item \cite[Row 13(1), Table 1]{H09} \ $V_{8,j}$, for $j\in\{2,4,6,8\}$, $p=5$.
\end{itemize}
\end{pro}
\begin{proof}
Assume that $V\cong \K_{\chi^i}=\K v \in{}_{\cH_{p,-1}}^{\cH_{p,-1}}\mathcal{YD}$ with $i\in\I_{0,2p-1}$.  Then $\K_{\chi^i}\in{}_{\gr\A_{p,-1}}^{\gr\A_{p,-1}}\mathcal{YD}$  with the structure given by Proposition \ref{proVA1}, so that  by \cite[Proposition 8.8]{HS13}, $\BN(\K_{\chi^i})\sharp\BN(X)\cong\BN(X\bigoplus Y_i )$ in ${}_{\Gamma}^{\Gamma}\mathcal{YD}$, where $ Y_i=\K v \in{}_{\Gamma}^{\Gamma}\mathcal{YD}$ with the Yetter-Drinfeld module structure given by
\begin{align*}
 g\cdot v=(-1)^iv,\quad\delta(v)=g^i\otimes v.
\end{align*}
It is clear that $\BN(X\bigoplus Y_{i})$ is of diagonal type with the generalized Dynkin diagram given by  \xymatrix@C+15pt{\overset{-1 }{{\circ}} & \overset{(-1)^i}{{\circ}}}. Hence by \cite[Theorem 1.1]{AA18b}, $\dim\BN(V)<\infty$, if and only if, $\dim\BN(X\bigoplus Y_i)<\infty$, if and only if, $i\in\I_{0,2p-1}$ is odd.

Assume that  $V\cong V_{i,j}=\K v_1\bigoplus\K v_2 \in{}_{\cH_{p,-1}}^{\cH_{p,-1}}\mathcal{YD}$ with $(i,j)\in\Lambda_p$.  Then  $V_{i,j}\in{}_{\gr\A_{p,-1}}^{\gr\A_{p,-1}}\mathcal{YD}$ with the structure given by Proposition \ref{proVA2}, so that  by \cite[Proposition 8.8]{HS13}, $\BN(V_{i,j})\sharp\BN(X)\cong\BN(X\bigoplus X_{i,j})$ in ${}_{\Gamma}^{\Gamma}\mathcal{YD}$, where $X_{i,j}=\K.v_1\in{}_{\Gamma}^{\Gamma}\mathcal{YD}$ with the Yetter-Drinfeld module structure given by
\begin{align*}
g\cdot v_1=\xi^{-j}v_1,\quad\delta(v_1)=g^i\otimes v_1.
\end{align*}

It is clear that $\BN(X\bigoplus X_{i,j})$ is of diagonal type with the generalized Dynkin diagram   given by
\begin{equation}\label{diagram:2}
\xymatrix@C+15pt{\overset{-1 }{{\circ}}\ar
@{-}[r]^{\xi^{pi-j}} & \overset{\xi^{-ij}}{{\circ}}}.
\end{equation}
Moreover, the Dynkin diagram is connected, since $(i,j)\in\Lambda_p$. Then $\dim\BN(V_{i,j})<\infty$, if and only if, $\dim\BN(X\bigoplus X_{i,j})<\infty$. Hence the Proposition follows from a case by case computation using Lemma \ref{lem-Nichols-finite-Zn-1}.

We claim that $j\not\in\{0,p\}$. Indeed, if $j=0$ or $p$, then $\xi^{-ij}=1$ or $(-)^{i+1}\xi^{-(i+1)j}=1$ and hence the diagram \eqref{diagram:2} has infinite root system, which implies $\dim\BN(V_{i,j})=\infty$.

If the diagram \eqref{diagram:2} belongs to Lemma \ref{lem-Nichols-finite-Zn-1} $(1)$, then $\xi^{-ij}=-1$ or $(-1)^i\xi^{-(i+1)j}=1$, that is,
\begin{gather}
ij\equiv p\mod 2p,\quad\text{ or }\label{eq:Diagram-2-1}\\
pi-(i+1)j\equiv 0\mod 2p.\label{eq:Diagram-2-2}
\end{gather}
Recall that $j\not\in\{0,p\}$. Since $p$ is prime, it is easy to see that the solution of equation \eqref{eq:Diagram-2-1}  is
\begin{gather*}
i=p,\quad j\equiv 1\mod 2,\quad j\neq p.
\end{gather*}
Equation \eqref{eq:Diagram-2-2} has a solution, only if $(i+1)j\equiv 0\mod p$, which implies that $i+1\equiv 0\mod p$, that is, $i=p-1$ or $2p-1$. Hence the solution of equation \eqref{eq:Diagram-2-2} is
\begin{gather*}
i=p-1,\quad p-j\equiv 1\mod 2,\quad j\neq 0.
\end{gather*}

If the diagram \eqref{diagram:2} belongs to Lemma \ref{lem-Nichols-finite-Zn-1} $(2)$, then $(-1)^i\xi^{-j}\xi^{-2ij}=1$, that is,
\begin{align}
pi-(2i+1)j\equiv 0\mod 2p.\label{eq:Diagram-2-3}
\end{align}
Since $j\not\in\{0,p\}$ and $p$ is prime, it has a solution, only if $(2i+1)j\equiv 0\mod p$, which implies that $2i+1\equiv 0\mod p$, that is, $i=\frac{p-1}{2}$ or $p+\frac{p-1}{2}$. Hence the solutions of equation \eqref{eq:Diagram-2-3} are
\begin{gather*}
i=\frac{p-1}{2},\quad j-\frac{p-1}{2}\equiv 0\mod 2\quad \text{and}\\
i=\frac{p-1}{2}+p,\quad j-\frac{p-1}{2}-p\equiv 0\mod 2.
\end{gather*}

If the diagram \eqref{diagram:2} belongs to Lemma \ref{lem-Nichols-finite-Zn-1} $(3)$, then $\xi^{-3ij}=1$ and $(-1)^{i+1}\xi^{-j}=\xi^{-ij}$, that is,
\begin{gather}
3ij\equiv 0\mod 2p,\label{eq:Diagram-2-4}\\
p(i+1)+(i-1)j\equiv 0\mod 2p.\label{eq:Diagram-2-5}
\end{gather}
Since $j\not\in\{0,p\}$ and $p$ is prime, the equation \eqref{eq:Diagram-2-5} has a solution, only if $(i-1)j\equiv 0\mod p$, which implies that $i-1\equiv 0\mod p$, that is, $i=1$ or $i=p+1$, then the solution of \eqref{eq:Diagram-2-5} is
\begin{align*}
i=1,\quad j\neq 0,p, \quad\text{or}\quad i=p+1,\quad p-j\equiv 0\mod 2,\quad j\neq 0,p.
\end{align*}
Since $j\not\equiv 0\mod p$, equation \eqref{eq:Diagram-2-4} has a solution, only if $3i\equiv 0\mod p$. Since $i=1$ or $p+1$, it follows that $p=3$. Hence the solutions of \eqref{eq:Diagram-2-4} and \eqref{eq:Diagram-2-5} are
\begin{gather*}
i=1,\quad j=2,4,\quad \text{and}\quad i=4,\quad j=1,5.
\end{gather*}

If the diagram \eqref{diagram:2} belongs to Lemma \ref{lem-Nichols-finite-Zn-1} $(4)$, then $\xi^{-5ij}=1, (-1)^i\xi^{-j}=\xi^{-2ij}$ or $\xi^{-5ij}=-1,(-1)^{i+1}\xi^{-j}=\xi^{-ij}$, that is,
\begin{gather}
5ij\equiv 0\mod 2p,\quad pi+(2i-1)j\equiv 0\mod 2p,\quad \text{or}\label{eq:Diagram-2-6}\\
5ij\equiv p\mod 2p,\quad p(i+1)+(i-1)j\equiv 0\mod 2p.\label{eq:Diagram-2-7}
\end{gather}
Similar to the preceding proof, we have $p=5$, the solution of \eqref{eq:Diagram-2-6} or \eqref{eq:Diagram-2-7} is
\begin{align*}
i=8,\quad j=2,4,6,8,\quad \text{or}\quad i=1,\quad j=1,3,7,9,\quad\text{respectively.}
\end{align*}

This completes the proof.
\end{proof}

\begin{pro}\label{pro:Nichols-algebra-simple-p=4-1}
Assume $p=4$, $V$ is a simple object in ${}_{\cH_{p,-1}}^{\cH_{p,-1}}\mathcal{YD}$. Then $\dim\BN(V)<\infty$, if and only if, $V$ is isomorphic to
\begin{itemize}
\item $\K_{\chi^{i_k}}$ with $i_k\in\{1,3,5,7\}$,
\end{itemize}
or $V_{i,j}$, where $(i,j)$ is equal to one of the following:
\begin{itemize}
\item \cite[Row 2(2), Table 1]{H09} \quad $(i,j)\in\{(2,2),(2,6),(6,2),(6,6)\}$;
\item \cite[Row 2(2), Table 1]{H09} \quad $(i,j)\in\{(4,1),(4,3),(4,5),(4,7)\}$;
\item \cite[Row 2(1), Table 1]{H09} \quad $(i,j)\in\{(5,2),(5,6), (1,2),(1,6)\}$;
\item \cite[Row 2(1), Table 1]{H09} \quad $(i,j)\in\{  (3,3),(3,1),(3,7),(3,5) \}$;
\item \cite[Row 11(3), Table 1]{H09} \quad $(i,j)\in\{(1,1),(1,3),(1,5),(1,7)\}$;
\item \cite[Row 11(2), Table 1]{H09} \quad $(i,j)\in\{(6,1),(6,3),(6,5),(6,7)\}$.
\end{itemize}
\end{pro}
\begin{proof}
Let $L_1:=\{(i,j)\in\Lambda_4\mid \xi^{ij}=1~\text{or}~\xi^{(i+1)(4-j)}=1\}$. It follows from a direct computation that $|L_1|=24$. By Remarks \ref{rmkN-infity}, \ref{rmkDmoddual} and Proposition \ref{proNicholsdual}, $\dim\BN(V_{i,j})=\infty$ for any $(i,j)\in L_1$. Let $L_2:=\{(i,j)\in\Lambda_4\mid i\in\{2,5\}, j\in\{1,3,5,7\}\}$. Since the generalized Dynkin diagram of $\BN(V_{i,j})$ for $(i,j)\in\L_2$ is \xymatrix@C+15pt{\overset{-1 }{{\circ}}\ar
@{-}[r]^{\xi^{-j}} & \overset{\xi^{-2j}}{{\circ}}} or \xymatrix@C+15pt{\overset{-1 }{{\circ}}\ar
@{-}[r]^{-\xi^{-j}} & \overset{\xi^{-5j}}{{\circ}}}. They do not appear in \cite[ Table 1]{H09} and hence have infinite root systems, which implies that $\dim\BN(X\bigoplus X_{i,j})=\infty$. From the proof of Proposition \ref{pro-Nichols-over-4p-1}, we have $\dim\BN(V_{i,j})=\infty$ for $(i,j)\in L_2$. Similarly, for any $(i,j)\in\Lambda_4-L_1-L_2$, $\dim\BN(X\bigoplus X_{i,j})<\infty$ and hence $\dim\BN(V_{i,j})<\infty$. Indeed, if $V_{i,j}$ belongs to the last two rows in the Proposition, the generalized Dynkin diagram of $\BN(X\bigoplus X_{i,j})$ is \xymatrix@C+15pt{\overset{-1 }{{\circ}}\ar@{-}[r]^{\xi^{-j}} & \overset{\xi^{2j}}{{\circ}}}, which appeared in \cite[Row 11, Table 1]{H09}; if $V_{i,j}$ belongs to the first four rows in the Proposition, then  the generalized Dynkin diagram is \xymatrix@C+15pt{\overset{-1 }{{\circ}}\ar
@{-}[r]^{\xi^{-j}} & \overset{-1}{{\circ}}} or \xymatrix@C+15pt{\overset{-1 }{{\circ}}\ar
@{-}[r]^{-\xi^{-j}} & \overset{-\xi^j}{{\circ}}}, which appeared in \cite[Row 2, Table 1]{H09}.
\end{proof}

\subsection{Nichols algebras over semisimple objects} We shall determine all semisimple objects $V$ in ${}_{\cH_{p,-1}}^{\cH_{p,-1}}\mathcal{YD}$ such that $\dim\BN(V)<\infty$ with a prime $p$ or $p=4$.
\begin{lem}\label{lem-Nichols-finite-Zn-2}
Assume $V\in{}_{C_{2p}}^{C_{2p}}\mathcal{YD}$ with $\dim V=3$ such that the associated generalized Dynkin diagram is connected. If there exists nonzero $v\in V$ such that $c(v\otimes v)=-v\otimes v$. Then $\dim\BN(V)<\infty$, only if the Dynkin diagram of $V$ is equivalent to one of the following
\begin{enumerate}
  \item \cite[Table 2, Row 8]{H09}
  \xymatrix@C+9pt{\overset{q}{{\circ}}\ar  @ {-}[r]^{q^{-1}}  & \overset{-1}{{\circ}}\ar  @{-}[r]^{q}& \overset{q^{-1}}{{\circ}}};
  \xymatrix@C+9pt{\overset{-1}{{\circ}}\ar  @ {-}[r]^{q}  & \overset{-1}{{\circ}}\ar  @{-}[r]^{q^{-1}}& \overset{-1}{{\circ}}};\\
  \xymatrix@C+9pt{\overset{-1}{{\circ}}\ar  @ {-}[r]^{q^{-1}}  & \overset{q}{{\circ}}\ar  @{-}[r]^{q^{-1}}& \overset{-1}{{\circ}}};
  \xymatrix@C+9pt{\overset{-1}{{\circ}}\ar  @ {-}[r]^{q}  & \overset{q^{-1}}{{\circ}}\ar  @{-}[r]^{q}& \overset{-1}{{\circ}}}; where $q\in G_p$.
  \item \cite[Table 2, Row 15]{H09}
  \xymatrix@C+9pt{\overset{-1}{{\circ}}\ar  @ {-}[r]^{\varsigma^{-1}}  & \overset{\varsigma}{{\circ}}\ar  @{-}[r]^{\varsigma}& \overset{-1}{{\circ}}};
  \xymatrix@C+9pt{\overset{-1}{{\circ}}\ar  @ {-}[r]^{\varsigma}  & \overset{-1}{{\circ}}\ar  @{-}[r]^{\varsigma}& \overset{-1}{{\circ}}};\\
  \xymatrix@C+9pt{\overset{-1}{{\circ}}\ar  @ {-}[r]^{\varsigma^{-1}}  & \overset{-\varsigma^{-1}}{{\circ}}\ar  @{-}[r]^{\varsigma^{-1}}& \overset{-1}{{\circ}}}; 
  \xymatrix@R-12pt{  &    \overset{-1}{{\circ}} \ar  @{-}[dl]_{\varsigma^{-1}}\ar  @{-}[dr]^{\varsigma^{-1}} & \\
\overset{\varsigma}{{\circ}} \ar  @{-}[rr]^{\varsigma^{-1}}  &  &\overset{\varsigma}{{\circ}} };

  where $\varsigma\in G_3$.

  \item \cite[Table 2, Row 9]{H09}
  \xymatrix@C+9pt{\overset{q}{{\circ}}\ar  @ {-}[r]^{q^{-1}}  & \overset{-1}{{\circ}}\ar  @{-}[r]^{r^{-1}}& \overset{r}{{\circ}}};
  \xymatrix@C+9pt{\overset{q}{{\circ}}\ar  @ {-}[r]^{q^{-1}}  & \overset{-1}{{\circ}}\ar  @{-}[r]^{s^{-1}}& \overset{s}{{\circ}}};\\
  \xymatrix@C+9pt{\overset{r}{{\circ}}\ar  @ {-}[r]^{r^{-1}}  & \overset{-1}{{\circ}}\ar  @{-}[r]^{s^{-1}}& \overset{s}{{\circ}}};
  \xymatrix@R-12pt{  &    \overset{-1}{{\circ}} \ar  @{-}[dl]_{q}\ar  @{-}[dr]^{r} & \\
\overset{-1}{{\circ}} \ar  @{-}[rr]^{s}  &  &\overset{-1}{{\circ}} };
   \\where $q,r,s\in\K-\{1\}$, $qrs=1$, $q\neq r$, $q\neq s$, $r\neq s$.
\end{enumerate}
\end{lem}
\begin{proof}
It follows directly from \cite[Theorem 3.7]{WZZ14}.
\end{proof}

\begin{pro}\label{pro-Nichols-over-4p-2}
Assume that $p$ is prime. Let $V_{i,j}\bigoplus\K_{\chi^k}$ be a semisimple object in ${}_{\cH_{p,-1}}^{\cH_{p,-1}}\mathcal{YD}$ for $(i,j)\in\Lambda_p,k\in\I_{0,2p-1}$. Then $\dim\BN(V_{i,j}\bigoplus\K_{\chi^k})<\infty$, if and only if, $V_{i,j}\bigoplus\K_{\chi^k}$ is isomorphic to one of the following objects
\begin{itemize}

\item \cite[Row 15 (2), Table 2]{H09} \ $V_{3,j}\bigoplus\K_{\chi}$, for $j\in\{1,5\}$, $p=3$;
\item \cite[Row 15 (1), Table 2]{H09} \ $V_{2,j}\bigoplus\K_{\chi^5}$, for $j\in\{2,4\}$, $p=3$;
\item \cite[Row 15 (3), Table 2]{H09} \ $V_{1,j}\bigoplus\K_{\chi}$, for $j\in\{1,5\}$, $p=3$;
\item \cite[Row 15 (1), Table 2]{H09} \ $V_{4,j}\bigoplus\K_{\chi^5}$, for $j\in\{2,4\}$, $p=3$;
\item \cite[Row 8 (2), Table 2]{H09} \ $V_{p,j}\otimes\K_{\chi^{-1}}$, for odd number $j\neq p$;
\item \cite[Row 8 (3), Table 2]{H09} \ $V_{p-1,j}\otimes\K_{\chi^{}}$, for odd number $p-j\neq p$;

\item   $V_{1,j}\bigoplus\K_{\chi^p}$, for $j\in\{1,3,7,9\}$, $p=5$;
\item   $V_{8,j}\bigoplus\K_{\chi^p}$, for $j\in\{2,4,6,8\}$, $p=5$;
\item  $V_{p,j}\bigoplus\K_{\chi^p}$, for odd number $j\neq p$, $p>2$;
\item  $V_{p-1,j}\bigoplus\K_{\chi^p}$, for even number $j\neq 0$, $p>2$;
\item  $V_{\frac{1}{2}(p-1),j}\bigoplus\K_{\chi^p}$, $p>2$. If $\frac{1}{2}(p-1)\equiv 0\mod 2$, then $j\neq 0$ is even, otherwise, $j\neq p$ is odd;
\item  $V_{\frac{1}{2}(p-1)+p,j}\bigoplus\K_{\chi^p}$, $p>2$. If $\frac{1}{2}(p-1)\equiv 0\mod 2$, then $j\neq p$ is odd, otherwise, $j\neq 0$ is even.
\end{itemize}
\end{pro}
\begin{proof}
Let $V_{i,j}=\K v_1\bigoplus\K v_2 \in{}_{\cH_{p,-1}}^{\cH_{p,-1}}\mathcal{YD}$, for $(i,j)\in\Lambda_p$ and $\K_{\chi^k}=\K.v_3\in{}_{\cH_{p,-1}}^{\cH_{p,-1}}\mathcal{YD}$, for $k\in\I_{0,2p-1}$.
Then $V_{i,j}\bigoplus\K_{\chi^k}\in{}_{\gr\A_{p,-1}}^{\gr\A_{p,-1}}\mathcal{YD}$ with the structure given by Propositions \ref{proVA1} \& \ref{proVA2}, so that by \cite[Proposition 8.8]{HS13}, $\BN(V_{i,j}\bigoplus\K_{\chi^k})\sharp\BN(X)\cong\BN(X\bigoplus X_{i,j}\bigoplus Y_k)$ in ${}_{\Gamma}^{\Gamma}\mathcal{YD}$, where $X_{i,j}\bigoplus Y_k=\K.v_1\bigoplus\K.v_3\in{}_{\Gamma}^{\Gamma}\mathcal{YD}$ with the Yetter-Drinfeld module structure given by
\begin{align*}
 g\cdot v_1=\xi^{-j}v_1,\quad g\cdot v_3=(-1)^kv_3,\quad\delta(v_1)=g^i\otimes v_1,\quad \delta(v_3)=g^k\otimes v_3.
\end{align*}
It is clear that $\BN(X\bigoplus X_{i,j}\bigoplus Y_k)$ is of diagonal type with the generalized  Dynkin diagram given by
\begin{align}\label{diagram:21}
\xymatrix@C+9pt{
\overset{-1}{\underset{x}{\circ}}\ar  @ {-}[r]^{\xi^{pi-j}}  & \overset{\xi^{-ij}}{\underset{v_1
}{\circ}}\ar  @{-}[r]^{\xi^{k(pi-j)}}
& \overset{(-1)^k}{\underset{v_3}{\circ}} }.
\end{align}
Then $\dim\BN(V_{i,j}\bigoplus\K_{\chi^k})<\infty$, if and only if, $\dim\BN(X\bigoplus X_{i,j}\bigoplus Y_k)<\infty$.

Assume that $\xi^{k(pi-j)}=1$, that is, $k(pi-j)\equiv 0\mod 2p$. Since $(i,j)\in\Lambda_p$ and $j\not\equiv 0\mod p$, it follows that $pi-j\not\equiv 0\mod 2p$ and $pi-j\not\equiv 0\mod p$. Therefore,  $k=p$. We claim that $p\neq 2$. Indeed, if $p=2$, then $j\equiv 0\mod 2$,  a contradiction. Furthermore, $pi-j\equiv 0\mod 2$. Then by Proposition \ref{pro-Nichols-over-4p-1}, $V_{i,j}\bigoplus \K_{\chi^k}$ must be isomorphic to one of the last six items of the Proposition. In this case, we have $c_{V_{i,j},\K_{\chi^p}}^2=\id$, which implies that $\BN(V_{i,j}\bigoplus\K_{\chi^p})\cong \BN(V_{i,j})\otimes\BN(\K_{\chi^p})$.

Assume that $\xi^{k(pi-j)}\neq 1$, that is, \eqref{diagram:21} is connected. The Proposition follows from a case by case computation using Lemma \ref{lem-Nichols-finite-Zn-2}.

If \eqref{diagram:21} belongs to Lemma \ref{lem-Nichols-finite-Zn-2} $(1)$, then $\xi^{-ij}=-1, \xi^{pi-j}\xi^{k(pi-j)}=1, (-1)^k=-1$ or $\xi^{pi-j}\xi^{-ij}=1=\xi^{k(pi-j)}\xi^{-ij},(-1)^k=-1$, that is,
\begin{gather}
ij\equiv 0\mod 2p,\quad (k+1)(pi-j)\equiv 0\mod 2p,\quad k\equiv 1\mod 2\quad\text{or} \label{eq:diagram21-1}  \\
pi-(i+1)j\equiv 0\mod 2p,\quad kpi-(k+i)j\equiv 0\mod 2p,\quad k\equiv 1\mod 2. \label{eq:diagram21-2}
\end{gather}
 Similar to the proof of Proposition \ref{pro-Nichols-over-4p-1}, the solution of  \eqref{eq:diagram21-1} or \eqref{eq:diagram21-2} is
\begin{gather*}
i=p,\quad j\equiv 1\mod 2,\quad j\neq p,\quad k=2p-1,\quad\text{or}\\
i=p-1,\quad p-j\equiv 1\mod 2,\quad j\neq 0,\quad k=1.
\end{gather*}

If \eqref{diagram:21} belongs to Lemma \ref{lem-Nichols-finite-Zn-2} $(2)$, then $\xi^{-3ij}=1=\xi^{pi-j}\xi^{-ij}=\xi^{(k+1)(pi-j)}=(-1)^{k+1}$, $\xi^{-3ij}=1=\xi^{k(pi-j)}\xi^{-ij} =\xi^{(k+1)(pi-j)}=(-1)^{k+1}$, $\xi^{-ij}=-1,\xi^{(3(pi-j)}=1=\xi^{(k-1)(pi-j)}=(-1)^{k+1}$, or $-\xi^{-3ij}=1=-\xi^{pi-j+ij}=\xi^{(k-1)(pi-j)}=(-1)^{k+1}$, that is,
\begin{align}
\begin{split}\label{eq:diagram21-3}
3ij\equiv 0\mod 2p,\quad pi-(i+1)j\equiv 0\mod 2p,\\
(k+1)(pi-j)\equiv 0\mod 2p,\quad k\equiv 1\mod 2;
\end{split}\\
\begin{split}\label{eq:diagram21-4}
3ij\equiv 0\mod 2p,\quad pki-(i+k)j\equiv 0\mod 2p,\\
(k+1)(pi-j)\equiv 0\mod 2p,\quad k\equiv 1\mod 2;
\end{split}\\
\begin{split}\label{eq:diagram21-5}
ij\equiv p\mod 2p,\quad 3(pi-j)\equiv 0\mod 2p,\\
(k-1)(pi-j)\equiv 0\mod 2p,\quad k\equiv 1\mod 2;\\
\end{split}\quad\text{or}\\
\begin{split}\label{eq:diagram21-6}
3ij\equiv p\mod 2p,\quad p(i+1)+(i-1)j\equiv 0\mod 2p,\\
(k-1)(pi-j)\equiv 0\mod 2p,\quad k\equiv 1\mod 2.
\end{split}
\end{align}
Similar to the proof of Proposition \ref{pro-Nichols-over-4p-1}, the solution of \eqref{eq:diagram21-3}, \eqref{eq:diagram21-4}, \eqref{eq:diagram21-5} or \eqref{eq:diagram21-6} is
\begin{gather*}
p=3,\quad i=p-1,\quad j\neq 0,j\equiv 0\mod 2,\quad k=2p-1;\\
p=3,\quad i=p+1,\quad j\neq 0,j\equiv 0\mod 2,\quad k=2p-1;\\
p=3,\quad i=p,\quad j\neq p, j\equiv 1\mod 2,\quad k=1;\quad\text{or}\\
p=3,\quad i=1, \quad j\neq p,j\equiv 1\mod 2,\quad  k=1.
\end{gather*}

If \eqref{diagram:21} belongs to Lemma \ref{lem-Nichols-finite-Zn-2} $(3)$, then it is clear that it is impossible.
\end{proof}

\begin{pro}\label{pro:Nichols-algebra-semisimple-12-p=4-1}
Suppose $V_{i,j}\bigoplus\K_{\chi^k}\in{}_{\cH_{4,-1}}^{\cH_{4,-1}}\mathcal{YD}$ is semisimple, for $(i,j)\in\Lambda_4,k\in\I_{0,7}$. Then $\dim\BN(V_{i,j}\bigoplus\K_{\chi^k})<\infty$, if and only if, $V_{i,j}\bigoplus\K_{\chi^k}$ is isomorphic to one of the following objects:
  \begin{itemize}
\item \cite[Row 8(2), Table 2]{H09} \ $V_{i,,j} \ \bigoplus\K_{\chi^3}$, $V_{i,,j}\bigoplus\K_{\chi^7}$, for  $(i,j)\in\{(2,2),(2,6),(6,2),(6,6)\}$;
\item \cite[Row 8(2), Table 2]{H09} \ $V_{i,,j}\bigoplus\K_{\chi^7}$, $(i,j)\in\{(4,1),(4,3),(4,5),(4,7)\}$;
\item \cite[Row 8(3), Table 2]{H09} \ $V_{i,,j}\bigoplus\K_{\chi}$, $V_{i,,j}\bigoplus\K_{\chi^5}$, for $(i,j)\in\{(5,2),(5,6), (1,2),(1,6)\}$;
\item \cite[Row 8(3), Table 2]{H09} \ $V_{i,,j}\bigoplus\K_{\chi}$, for $(i,j)\in\{  (3,3),(3,1),(3,7),(3,5) \}$.
\end{itemize}
\end{pro}
\begin{proof}
The proof follows from the same lines of Proposition \ref{pro-Nichols-over-4p-2}.
\end{proof}
\begin{pro}\label{pro-Nichols-over-4p-3}
Assume that $p$ is prime, and $V_{i,j}, V_{k,\ell}\in{}_{\cH_{p,-1}}^{\cH_{p,-1}}\mathcal{YD}$, for $(i,j),(k,\ell)\in\Lambda_p$.  Then $\dim\BN(V_{i,j}\bigoplus V_{k,\ell})<\infty$, if and only if, $V_{i,j}\bigoplus V_{k,\ell}$ is isomorphic to one of the following objects
\begin{itemize}
\item \cite[Row 8 (2), Table 2]{H09} $V_{p,j}\bigoplus V_{p,-j}$ for odd number $j\neq p$;
\item \cite[Row 8 (1), Table 2]{H09} $V_{p-1,j}\bigoplus V_{p-1,-j}$ for odd number $p-j\neq p$;
\item \cite[Row 15 (2), Table 2]{H09} $V_{3,j}\bigoplus V_{3,j}$ for $j\in\{1,5\}$, $p=3$;
\item \cite[Row 15 (4), Table 2]{H09} $V_{2,j}\bigoplus V_{2,j}$ for $j\in\{2,4\}$, $p=3$.
\end{itemize}
\end{pro}
\begin{proof}
Let $V_{i,j}=\K. v_1\bigoplus\K. v_2 , V_{k,\ell}=\K.w_1\bigoplus\K.w_2\in{}_{\cH_{p,-1}}^{\cH_{p,-1}}\mathcal{YD}$, for $(i,j),(k,\ell)\in\Lambda_p$. 
Then by Proposition \ref{proVA2}, $V_{i,j}\bigoplus V_{k,\ell}\in{}_{\gr\A_{p,-1}}^{\gr\A_{p,-1}}\mathcal{YD}$
so that by \cite[Proposition 8.8]{HS13}, $\BN(V_{i,j}\bigoplus V_{k,\ell})\sharp\BN(X)\cong\BN(X\bigoplus X_{i,j}\bigoplus X_{k,\ell})$ in ${}_{\Gamma}^{\Gamma}\mathcal{YD}$, where $ X_{i,j}\bigoplus X_{k,\ell}=\K.x\bigoplus\K.v_1\bigoplus \K.w_1\in{}_{\Gamma}^{\Gamma}\mathcal{YD}$ with the Yetter-Drinfeld module structure given by
\begin{align*}
 g\cdot v_1=\xi^{-j}v_1,\quad  g\cdot w_1=\xi^{-\ell}w_1,\quad\delta(v_1)=g^i\otimes v_1,\quad \delta(w_1)=g^k\otimes w_1.
\end{align*}
It is clear that $\BN(X\bigoplus X_{i,j}\bigoplus X_{k,\ell})$ is of diagonal type with the generalized Dynkin diagram given by
\begin{align}\label{diagram:22}
 \xymatrix@R-12pt{  &    \overset{-1}{\underset{x}{\circ}} \ar  @{-}[dl]_{\xi^{pk-\ell}}\ar  @{-}[dr]^{\xi^{pi-j}} & \\
\overset{\xi^{-k\ell}}{\underset{w_1}{\circ}} \ar  @{-}[rr]^{\xi^{-kj-i\ell}}  &  &\overset{\xi^{-ij}}{\underset{v_1}{\circ}} }
\end{align}
Then $\dim\BN(V_{i,j}\bigoplus V_{k,\ell})<\infty$, if and only if, $\dim\BN(X\bigoplus X_{i,j}\bigoplus X_{k,\ell})<\infty$. Since $(i,j),(k,\ell)\in\Lambda_p$, it follows that $(-1)^i\xi^{-j}\neq 1$ and $(-1)^k\xi^{-\ell}\neq 1$, which implies that \eqref{diagram:22} is connected. Hence the Proposition follows from a case by case computation using Lemma  \ref{lem-Nichols-finite-Zn-2}.

Observe that $k,\ell\not\in\{0,p\}$; otherwise, $\dim\BN(V_{i,j}\bigoplus V_{k,\ell})=\infty$.

If \eqref{diagram:22} belongs to Lemma \ref{lem-Nichols-finite-Zn-2} $(1)$, then $\xi^{-kj-i\ell}=1=-\xi^{-ij}=-\xi^{k\ell}=\xi^{p(i+k)-(j+\ell)}$ or $\xi^{-kj-i\ell}=1=\xi^{pk-(k+1)\ell}=\xi^{pi-(i+1)j}=\xi^{-k\ell-ij}$, that is,
\begin{gather}
\begin{split}\label{eq:diagram22-1}
kj+i\ell\equiv 0\mod 2p,\quad k\ell\equiv p\mod 2p,\\ ij\equiv p\mod 2p,\quad p(i+k)-(j+\ell)\equiv 0\mod 2p;
\end{split}\quad\text{or}\\
\begin{split}\label{eq:diagram22-2}
kj+i\ell\equiv 0\mod 2p,\quad pk-(k+1)\ell\equiv 0\mod 2p,\\ pi-(i+1)\ell\equiv 0\mod 2p,\quad k\ell+ij\equiv 0\mod 2p.
\end{split}
\end{gather}
Similar to the proof of Proposition \ref{pro-Nichols-over-4p-1}, the solution of \eqref{eq:diagram22-1} or \eqref{eq:diagram22-2} is
\begin{gather*}
i=k=p, \quad j+\ell\equiv 0\mod 2p,\quad j\equiv 1\mod 2,\quad j\neq p;\quad \text{or}\\
i=k=p-1,\quad j+\ell\equiv 0\mod 2p,\quad p-j\equiv 1\mod 2,\quad j-p\neq p.
\end{gather*}

If \eqref{diagram:22} belongs to Lemma \ref{lem-Nichols-finite-Zn-2} $(2)$, then $-\xi^{-k\ell}=1=-\xi^{-ij}=\xi^{-kj-i\ell}=\xi^{pk-\ell}\xi^{j-pi}=(-1)^i\xi^{-3j}$ or $\xi^{-3ij}=1=\xi^{ij-k\ell}=(-1)^i\xi^{-j}\xi^{-ij}=(-1)^k\xi^{-\ell}\xi^{-k\ell}=\xi^{kj+i\ell}\xi^{pi-j}$, that is,
\begin{align}
\begin{split}\label{eq:diagram22-3}
ij\equiv p\mod 2p,\quad k\ell\equiv p\mod 2p,\quad kj+i\ell\equiv 0\mod 2p,\\
p(k-i)-(\ell-j)\equiv 0\mod 2p,\quad pi-3j\equiv 0\mod 2p;
\end{split}\\
\begin{split}\label{eq:diagram22-4}
3ij\equiv 0\mod 2p,\quad ij-k\ell\equiv 0\mod 2p,\quad pi-(i+1)j\equiv 0\mod 2p,\\
pk-(k+1)\ell\equiv 0\mod 2p,\quad kj+i\ell+pi-j\equiv 0\mod 2p.
\end{split}
\end{align}
Similar to the proof of Proposition \ref{pro-Nichols-over-4p-1}, the solution of \eqref{eq:diagram22-3} or \eqref{eq:diagram22-4} is
\begin{gather*}
p=3,\quad i=p=k,\quad \ell-j\equiv 0\mod 2p,\quad j\neq p,\quad j\equiv 1\mod 2;\quad \text{or}\\
p=3,\quad i=p-1=k,\quad \ell-j\equiv 0\mod 2p,\quad j\neq 0,\quad j\equiv 0\mod 2.
\end{gather*}

If \eqref{diagram:22} belongs to Lemma \ref{lem-Nichols-finite-Zn-2} $(3)$, then $(-1)^i\xi^{-j}\xi^{-ij}=1=(-1)^k\xi^{-\ell}\xi^{-k\ell}=\xi^{-kj-i\ell},\  \xi^{ij-k\ell}\neq 1,\  \xi^{-ij-k\ell}\neq 1$ or $-\xi^{-ij}=1=-\xi^{-k\ell}=(-1)^{i+k}\xi^{-j-\ell}\xi^{-i\ell-kj},\  (-1)^{i-k}\xi^{\ell-j}\neq 1,\  (-1)^i\xi^{-j}\neq \xi^{-i\ell-kj},\ (-1)^k\xi^{-\ell}\neq \xi^{-i\ell-kj},\ \xi^{-i\ell-kj}\neq 1$. There do not exist $(i,j),(k,\ell)\in\Lambda_p$ satisfying the above conditions. Indeed, similar to the proof of Proposition \ref{pro-Nichols-over-4p-1}, from $(-1)^i\xi^{-j}\xi^{-ij}=1=(-1)^k\xi^{-\ell}\xi^{-k\ell}$, we have $i=p-1=k$, then
\begin{align*}
\xi^{-ij-k\ell}&\neq 1 \quad \text{implies that}\quad \xi^{(p-1)(j-\ell)}\neq 1;\\
\xi^{-kj-i\ell}&=1,\quad \text{implies that}\quad \xi^{(p-1)(j-\ell)}= 1, \quad \text{a contradiction}.
\end{align*}
From $-\xi^{-ij}=1=-\xi^{-k\ell}$, we have $i=p=k$ and $j,\ell$ are both odd, which contradict $\xi^{-i\ell-jk}\neq 1$.
\end{proof}

\begin{cor}\label{cor-Nichols-4p-1}
$\dim\BN(V)=\infty$, if $V$ is isomorphic to $V_{i,j}\bigoplus V_{k,\ell}\bigoplus V_{m,n}$, for any $(i,j),(k,\ell)$, $(m,n)\in\Lambda_p$.
\end{cor}
\begin{proof}
From the proof of Proposition \ref{pro-Nichols-over-4p-3}, the generalized Dynkin diagram of $\BN(X\bigoplus X_{i,j}\bigoplus$ $ X_{k,\ell}\bigoplus X_{m,n})$ is connected. Observe that $\dim X\bigoplus X_{i,j}\bigoplus X_{k,\ell}\bigoplus X_{m,n}=4$. Then by \cite[Theorem 4.1]{WZZ14}, $\dim\BN(X\bigoplus X_{i,j}\bigoplus X_{k,\ell}\bigoplus X_{m,n})=\infty$ and hence by \cite[Theorem 1.1]{AA18b}, $\dim\BN(V)=\infty$.
\end{proof}

\begin{pro}\label{pro:Nichols-algebra-semisimple-22-p=4-1}
Suppose $V_{i,j}, V_{k,\ell}\in{}_{\cH_{4,-1}}^{\cH_{4,-1}}\mathcal{YD}$, for $(i,j),(k,\ell)\in\Lambda_4$.  Then $\dim\BN(V_{i,j}\bigoplus V_{k,\ell})<\infty$, if and only if, $V_{i,j}\bigoplus V_{k,\ell}$ is isomorphic to one of the following objects
\begin{itemize}
\item \cite[Row 8 (2), Table 2]{H09} \ $V_{2,2}\bigoplus V_{i,6}$, $V_{6,2}\bigoplus V_{i,6}$, for $i\in\{2,6\}$;
\item \cite[Row 8 (2), Table 2]{H09} \ $V_{4,1}\bigoplus V_{4,7}$, $V_{4,3}\bigoplus V_{4,5}$;
\item \cite[Row 8 (1), Table 2]{H09} \ $V_{5,2}\bigoplus V_{i,6}$, $V_{1,2}\bigoplus V_{i,6}$, for $i\in\{1,5\}$;
\item \cite[Row 8 (1), Table 2]{H09} \ $V_{3,3}\bigoplus V_{3,5}$, $V_{3,1}\bigoplus V_{3,7}$.
\end{itemize}
\end{pro}
\begin{proof}
The proof follows from the same lines of Proposition \ref{pro-Nichols-over-4p-3}.
\end{proof}
\begin{lem}\label{lem-Nichols-4p-1}
$\dim\BN(V)=\infty$, if $V$ is isomorphic to one of the objects: $V_{i,j}\bigoplus V_{k,\ell}\bigoplus\K_{\chi^r}$, $V_{i,j}\bigoplus\K_{\chi^s}\bigoplus\K_{\chi^t}$, for any $(i,j),(k,\ell)\in\Lambda_p$, $r,s,t\in\I_{0,2p-1}-\{p\}$, or $V_{i,j}\bigoplus V_{k,\ell}\bigoplus V_{m,n}$, for any $(i,j),(k,\ell),(m,n)\in\Lambda_p$.
\end{lem}
 \begin{proof}
Assume that
$V\cong V_{i,j}\bigoplus V_{k,\ell}\bigoplus\K_{\chi^r}=\K. v_1\bigoplus\K. v_2 \bigoplus\K.e_1\bigoplus\K.e_2\bigoplus\K. v \in{}_{\cK_{p,-1}}^{\cK_{p,-1}}\mathcal{YD}$, for any $(i,j),(k,\ell)\in\Lambda_p$, $r\in\I_{0,2p-1}-\{p\}$. Then $V\in{}_{\gr\A_{p,-1}}^{\gr\A_{p,-1}}\mathcal{YD}$, so that by \cite[Proposition 8.8]{HS13}, $\BN(V)\sharp\BN(X)\cong\BN(X\bigoplus X_{i,j}\bigoplus X_{k,\ell}\bigoplus Y_s)$ in ${}_{\Gamma}^{\Gamma}\mathcal{YD}$, where $X_{i,j}\bigoplus X_{k,\ell}\bigoplus Y_s=\K.v_1\bigoplus\K.e_1\bigoplus\K. v \in{}_{\Gamma}^{\Gamma}\mathcal{YD}$ with the Yetter-Drinfeld module structure given by
\begin{align*}
 g\cdot v_1=\xi^{-j}v_1,\quad  g\cdot w_1=\xi^{-\ell}w_1,\quad g\cdot v=(-1)^rv,\\
 \delta(v_1)=g^i\otimes v_1,\quad \delta(w_1)=g^k\otimes w_1,\quad \delta(v)=g^r\otimes v.
\end{align*}
It is clear that $\BN(X_{i,j}\bigoplus X_{k,\ell}\bigoplus Y_s)$ is of diagonal type with the generalized Dynkin diagram given by
$$
\xymatrix{ & \overset{(-1)^r}{\underset{v}{\circ}}   \ar  @{-}[dl]_{\xi^{r(pi-j)}}  \ar  @{-}[dr]^{\xi^{r(pk-\ell)}}  &   \\
\overset{\xi^{-ij}}{\underset{v_1}\circ}\ar  @{-}[r]^{\xi^{pi-j}}  & \overset{-1}{\underset{x}\circ} \ar  @{-}[r]^{\xi^{pk-\ell}}  & \overset{\xi^{-k\ell}}{\underset{e_1}\circ}}
$$

Assume that  $V\cong V_{i,j}\bigoplus\K_{\chi^s}\bigoplus\K_{\chi^t}=\K. v_1\bigoplus\K. v_2 \bigoplus\K.v_3\bigoplus\K.v_4\in{}_{\cK_{p,-1}}^{\cK_{p,-1}}\mathcal{YD}$ for any $(i,j)\in\Lambda_p$, $s,t\in\I_{0,2p-1}-\{p\}$. Then $V_{i,j}\bigoplus\K_{\chi^{s}}\bigoplus\K_{\chi^{t}}\in{}_{\gr\A_{p,-1}}^{\gr\A_{p,-1}}\mathcal{YD}$, so that by \cite[Proposition 8.8]{HS13}, $\BN(V)\sharp\BN(X)\cong\BN(X\bigoplus X_{i,j}\bigoplus Y_s\bigoplus Y_t)$ in ${}_{\Gamma}^{\Gamma}\mathcal{YD}$, where $X_{i,j}\bigoplus Y_{s}\bigoplus Y_t=\K.v_1\bigoplus\K.v_3\bigoplus\K.v_4\in{}_{\Gamma}^{\Gamma}\mathcal{YD}$ with the Yetter-Drinfeld module structure given by
 \begin{align*}
 g\cdot v_1=\xi^{-j}v_1,\quad g\cdot v_3=(-1)^sv_3\quad g\cdot v_4=(-1)^tv_4\\
 \delta(v_1)=g^i\otimes v_1,\quad\delta(v_3)=g^{s}\otimes v_3,\quad \delta(v_4)=g^{t}\otimes v_4.
\end{align*}
It is clear that $\BN(X_{i,j}\bigoplus Y_s\bigoplus Y_t)$ is of diagonal type with the generalized Dynkin diagram given by
$$
\xymatrix{ & \overset{(-1)^t}{\underset{v_4}{\circ}} \ar  @{-}[d]^{\xi^{t(pi-j)}} &   \\
	\overset{-1}{\underset{x}{\circ}}\ar  @{-}[r]^{\xi^{pi-j}}  & \overset{\xi^{-ij}}{\underset{v_1}{\circ}} \ar  @{-}[r]^{\xi^{s(pi-j)}}  & \overset{(-1)^s}{\underset{v_3}{\circ}}}
$$

Assume that $V\cong V_{i,j}\bigoplus V_{k,\ell}\bigoplus V_{m,n}=\K. v_1\bigoplus\K. v_2 \bigoplus \K.w_1\bigoplus \K.w_2\bigoplus \K.e_1\bigoplus \K.e_2$ for any $(i,j),(k,\ell),(m,n)\in\Lambda_p$, the proof follows from the same lines. Indeed, since $(i,j),(k,\ell),(m,n)\in\Lambda_p$, it follows that $(-1)^i\xi^j\neq 1$, $(-1)^k\xi^{\ell}\neq 1$ and $(-1)^m\xi^{-n}\neq 1$, which implies that   the vertices $v_1, w_1, e_1$ and $x$ of the Dynkin diagram are connected.
\end{proof}

Now we give all objects $V$ in ${}_{\cH_{p,-1}}^{\cH_{p,-1}}\mathcal{YD}$ such that $\dim\BN(V)<\infty$ with a prime $p$ or  $p=4$.
\begin{thm}\label{thm-finite-braided-vector-spaces-H}
Assume that $p$ is prime, $V\in{}_{\cH_{p,-1}}^{\cH_{p,-1}}\mathcal{YD}$ such that $\dim\BN(V)<\infty$, then $V$ is isomorphic to
one of the objects:
\begin{itemize}
  \item $\bigoplus_{k=1}^n\K_{\chi^{i_k}}$ with $n\in\N$ and odd number $i_k\in\I_{0,2p-1}$;
\end{itemize}

(1) If $p=2$,
\begin{itemize}
\item $\bigoplus_{k=1}^n\K_{\chi^{i_k}}$ with $i_k\in\{1,3\}$,
\item \cite[Table 1, Row 2(1))]{H09} \ $V_{1,j}$,  $j\in\{1,3\}$,
\item \cite[Table 1, Row 2(2)]{H09} \ $V_{2,j}$,  $j\in\{1,3\}$,
\item \cite[Table 2, Row 8(3)]{H09} \ $V_{1,j}\bigoplus \K_{\chi}$,   $j\in\{1,3\}$,
\item \cite[Table 2, Row 8(2)]{H09} \ $V_{2,j}\bigoplus\K_{\chi^3}$,  $j\in\{1,3\}$,
\item \cite[Table 2, Row 8(1)]{H09} \ $V_{1,1}\bigoplus V_{1,3}$,
\item \cite[Table 2, Row 8(2)]{H09} \ $V_{2,1}\bigoplus V_{2,3}$;
\end{itemize}

(2) If $p>2$,
 $(\bigoplus_{k=0}^n\K_{\chi^p})\bigoplus$
\begin{itemize}
\item \cite[Row 2(2), Table 1]{H09} \ $V_{p,j}$, for odd number $j\neq p$, 
\item \cite[Row 2(1), Table 1]{H09} \ $V_{-p-1,j}$, for even number $j\neq 0$, 
\item \cite[Row 4, Table 1]{H09} \ $V_{\frac{1}{2}(p-1),j}$. If $\frac{1}{2}(p-1)\equiv 0\mod 2$, then $j\neq 0$ is even, otherwise, $j\neq p$ is odd,
\item \cite[Row 4, Table 1]{H09} \ $V_{\frac{1}{2}(p-1)+p,j}$. If $\frac{1}{2}(p-1)\equiv 0\mod 2$, then $j\neq p$ is odd, otherwise, $j\neq 0$ is even,

\item \cite[Row 8 (2), Table 2]{H09} \ $V_{p,j}\otimes\K_{\chi^{-1}}$, for odd number $j\neq p$,

\item \cite[Row 8 (3), Table 2]{H09} \ $V_{-p-1,j}\otimes\K_{\chi^{}}$, for even number $j\neq 0$,

\item \cite[Row 8 (2), Table 2]{H09} \ $V_{p,j}\bigoplus V_{p,-j}$, for odd number $j\neq p$, 

\item \cite[Row 8 (1), Table 2]{H09} \ $V_{-p-1,j}\bigoplus V_{-p-1,-j}$, for even number $j\neq 0$; 
\end{itemize}

(3) If $p=3$,
\begin{itemize}
\item \cite[Row 6, Table 1]{H09} \ $V_{1,j}$, for $j\in\{2,4\}$, 

\item \cite[Row 6, Table 1]{H09} \ $V_{4,j}$, for $j\in\{1,5\}$, 

\item $(\bigoplus_{k=0}^n\K_{\chi^3})\bigoplus$
\begin{itemize}

\item \cite[Row 15 (3), Table 2]{H09} \ $V_{3,j}\bigoplus\K_{\chi}$, for $j\in\{1,5\}$, 

\item \cite[Row 15 (1), Table 2]{H09} \ $V_{2,j}\bigoplus\K_{\chi^5}$, for $j\in\{2,4\}$, 

\item \cite[Row 15 (3), Table 2]{H09} \ $V_{1,j}\bigoplus\K_{\chi}$, for $j\in\{1,5\}$, 

\item \cite[Row 15 (1), Table 2]{H09} \ $V_{4,j}\bigoplus\K_{\chi^5}$, for $j\in\{2,4\}$, 

\item \cite[Row 15 (2), Table 2]{H09} \ $V_{3,j}\bigoplus V_{3,j}$, for $j\in\{1,5\}$, 

\item \cite[Row 15 (4), Table 2]{H09} \ $V_{2,j}\bigoplus V_{2,j}$,
for $j\in\{2,4\}$; 
\end{itemize}
\end{itemize}

(4) If $p=5$,
$(\bigoplus_{k=0}^n\K_{\chi^p})\bigoplus$
\begin{itemize}
\item \cite[Row 13(2), Table 1]{H09} \ $V_{1,j}$, for $j\in\{1,3,7,9\}$, 

\item \cite[Row 13(1), Table 1]{H09} \ $V_{8,j}$, for $j\in\{2,4,6,8\}$.
\end{itemize}
\end{thm}
\begin{proof}
The theorem follows from Propositions \ref{pro-Nichols-over-4p-1}, \ref{pro-Nichols-over-4p-2}, \ref{pro-Nichols-over-4p-3} and Lemma \ref{lem-Nichols-4p-1}.
\end{proof}

\begin{thm}\label{thm-finite-braided-vector-spaces-H-p=4}
Assume $p=4$ and $V\in {}_{\cH_{p,-1}}^{\cH_{p,-1}}\mathcal{YD}$. Then $\dim\BN(V)<\infty$, if and only if, $V$ is isomorphic to one of the following
\begin{itemize}
\item $\K_{\chi^{i_k}}$ with $i_k\in\{1,3,5,7\}$,
\item \cite[Row 2(2), Table 1]{H09} \ $V_{i,j}$, for $(i,j)\in\{(2,2),(2,6),(6,2),(6,6)\}$;
\item \cite[Row 2(2), Table 1]{H09} \ $V_{i,j}$, for $(i,j)\in\{(4,1),(4,3),(4,5),(4,7)\}$;
\item \cite[Row 2(1), Table 1]{H09} \ $V_{i,j}$, for $(i,j)\in\{(5,2),(5,6), (1,2),(1,6)\}$;
\item \cite[Row 2(1), Table 1]{H09} \ $V_{i,j}$, for $(i,j)\in\{  (3,3),(3,1),(3,7),(3,5) \}$;
\item \cite[Row 11(3), Table 1]{H09} \ $V_{i,j}$, for $(i,j)\in\{(1,1),(1,3),(1,5),(1,7)\}$;
\item \cite[Row 11(2), Table 1]{H09} \ $V_{i,j}$, for $(i,j)\in\{(6,1),(6,3),(6,5),(6,7)\}$;
\item \cite[Row 8(2), Table 2]{H09} \ $V_{i,,j}\bigoplus\K_{\chi^3}$, $V_{i,,j}\bigoplus\K_{\chi^7}$, for  $(i,j)\in\{(2,2),(2,6),(6,2),(6,6)\}$;
\item \cite[Row 8(2), Table 2]{H09} \ $V_{i,,j}\bigoplus\K_{\chi^7}$, for $(i,j)\in\{(4,1),(4,3),(4,5),(4,7)\}$;
\item \cite[Row 8(3), Table 2]{H09} \ $V_{i,,j}\bigoplus\K_{\chi}$, $V_{i,,j}\bigoplus\K_{\chi^5}$, for $(i,j)\in\{(5,2),(5,6), (1,2),(1,6)\}$;
\item \cite[Row 8(3), Table 2]{H09} \ $V_{i,,j}\bigoplus\K_{\chi}$, for $(i,j)\in\{  (3,3),(3,1),(3,7),(3,5) \}$;
\item \cite[Row 8 (2), Table 2]{H09} \ $V_{2,2}\bigoplus V_{i,6}$, $V_{6,2}\bigoplus V_{i,6}$, for $i\in\{2,6\}$;
\item \cite[Row 8 (2), Table 2]{H09} \ $V_{4,1}\bigoplus V_{4,7}$, $V_{4,3}\bigoplus V_{4,5}$;
\item \cite[Row 8 (1), Table 2]{H09} \ $V_{5,2}\bigoplus V_{i,6}$, $V_{1,2}\bigoplus V_{i,6}$, for $i\in\{1,5\}$;
\item \cite[Row 8 (1), Table 2]{H09} \ $V_{3,3}\bigoplus V_{3,5}$, $V_{3,1}\bigoplus V_{3,7}$.
\end{itemize}
\end{thm}
\begin{proof}
It follows from Propositions \ref{pro:Nichols-algebra-simple-p=4-1}, \ref{pro:Nichols-algebra-semisimple-12-p=4-1}, \ref{pro:Nichols-algebra-semisimple-22-p=4-1} and Lemma \ref{lem-Nichols-4p-1}.
\end{proof}

\bigskip
\section{On Hopf algebras over $\cH_{p,-1}$}\label{secHopfalgebra}
In this section, we shall study finite-dimensional Hopf algebras over $\cH_{p,-1}$ for a natural number $p$. By computing the liftings of Nichols algebras, we construct families of finite-dimensional Hopf algebras without the dual Chevalley property and pointed duals, which constitute new examples of finite-dimensional Hopf algebras. In particular, we give a complete classification of finite-dimensional Hopf algebras over $\cH_{p,-1}$ when $p>5$ is a prime number.
\subsection{Defining relations of Nichols algebras in ${}_{\cH_{p,-1}}^{\cH_{p,-1}}\mathcal{YD}$} It is a hard question to present Nichols algebras in a suitable category of Yetter-Drinfeld modules by generators and relations, for lack of standard  methods and effective tools. After tedious computations, we solve the question for some Nichols algebras in ${}_{\cH_{p,-1}}^{\cH_{p,-1}}\mathcal{YD}$ for $p\in\N$, which are necessary to obtain our classification results. We leave the remaining cases to future work.
\begin{pro}\label{pro:relations-Nichols-algebra-2-dim-simple-quad}
Let $V_{i,j}=\K v_1\bigoplus\K v_2 \in{}_{\cH_{p,-1}}^{\cH_{p,-1}}\mathcal{YD}$. Then $\BN(V_{i,j})$ admits quadratic relations if and only if $\xi^{-ij}=-1$ or $\xi^{(i+1)(p-j)}= -1$. Moreover,
if $\xi^{(i+1)(p-j)}= -1$, then $\BN(V_{i,j})$  is generated as an algebra by $v_1, v_2$ satisfying the relations
\begin{gather}
v_1v_2+(-1)^{i+1}v_2v_1=0,\quad v_2^2+(-1)^{i+1}(\theta\xi^{i+1})^{-2}v_1^2=0,\label{eq:H4p-relations-Nichols-algebra-quad-3}\\
v_1^N=0,\quad  \text{~where~} N=\ord(\xi^{-ij});\label{eq:H4p-relations-Nichols-algebra-quad-4}
\end{gather}
if $\xi^{-ij}=-1$, then $\BN(V_{i,j})$ is generated as an algebra by $v_1, v_2$, subject to the relations
\begin{gather}
v_1^2=0,\quad v_1v_2+\xi^{-j}v_2v_1=0,\label{eq:H4p-relations-Nichols-algebra-quad-1}\\
v_2^N=0, \text{~where~} N=\ord ((-1)^i\xi^{-j}).\label{eq:H4p-relations-Nichols-algebra-quad-2}
\end{gather}
\end{pro}
\begin{proof}
 Put $x=\alpha_1v_1^2+\alpha_2v_1v_2+\alpha_3v_2v_1+\alpha_4v_2^2$. Using the braidings in Proposition \ref{probraidsimpletwo}, we have
\begin{align*}
\Delta(x)&=x\otimes 1+1\otimes x+[\alpha_1(1+\xi^{-ij})+\alpha_4\theta^{-2}\xi^{-(i+1)(j+2)}(1+(-1)^i\xi^j)]v_1\otimes v_1\\&\quad
+\alpha_4[1+\xi^{(p-j)(i+1)}]v_2\otimes v_2+[\alpha_3+\alpha_2\xi^{-j(i+1)}]v_2\otimes v_1\\&\quad
+[\alpha_2(1+\xi^{-ij}+\xi^{(p-j)(i+1)})+\alpha_3\xi^{i(p-j)}]v_1\otimes v_2.
\end{align*}
Then $\BN(V_{i,j})$ admits quadratic relations, if and only if,
\begin{gather*}
\alpha_1(1+\xi^{-ij})+\alpha_4\theta^{-2}\xi^{-(i+1)(j+2)}(1+(-1)^i\xi^j)=0,\\
\alpha_4[1+\xi^{(p-j)(i+1)}]=0,\quad
\alpha_3+\alpha_2\xi^{-j(i+1)}=0,\\
\alpha_2(1+\xi^{-ij}+\xi^{(p-j)(i+1)})+\alpha_3\xi^{i(p-j)}=0;
\end{gather*}
if and only if
\begin{align*}
1+\xi^{-ij}&=0,\quad \alpha_4=0,\quad \alpha_3=\alpha_2\xi^{-j},\quad\text{or}\\1+\xi^{(p-j)(i+1)}=0,\quad \alpha_3+\alpha_2(-1)^i&=0,\quad \alpha_1+\alpha_4(-1)^i\theta^{-2}\xi^{-2(i+1)}=0;
\end{align*}
if and only if, $1+\xi^{-ij}= 0$ or $1+\xi^{(p-j)(i+1)}= 0$.

If $1+\xi^{(p-j)(i+1)}= 0$, then from the preceding proof, relations \eqref{eq:H4p-relations-Nichols-algebra-quad-3} must hold in $\BN(V_{i,j})$. Since $c(v_1\otimes v_1)=\xi^{-ij}v_1\otimes v_1$, it follows from a direct computation that $\Delta(v_1^N)=v_1^N\otimes 1+1\otimes v_1^N$ and hence relation \eqref{eq:H4p-relations-Nichols-algebra-quad-4} must hold in $\BN(V_{i,j})$. Therefore, the quotient $\BN$ of $T(V_{i,j})$ by relations \eqref{eq:H4p-relations-Nichols-algebra-quad-3}--\eqref{eq:H4p-relations-Nichols-algebra-quad-4} projects onto $\BN(V_{i,j})$.

Let $B:=\{v_2^{n_2}v_1^{n_1}\mid n_1\in\I_{0,N-1},n_2\in\I_{0,1} \}$. We claim that  the subspace $I$ linearly spanned by $B$ is a left ideal of $\BN$.  Then $B$ linearly  generates $\BN$ since $1\in I$. Indeed,  from the defining relations of $\BN$, it is clear that  $gI\subset I$ for $g\in\{v_1,v_2\}$.

We claim that $\dim\BN(V_{i,j})=|B|$. Indeed, a direct computation shows that the generalized Dynkin diagram of $\BN(X\bigoplus X_{i,j})$ is \xymatrix@C+15pt{\overset{-1 }{{\circ}}\ar
@{-}[r]^{\xi^{ij}} & \overset{\xi^{-ij}}{{\circ}}}, which belongs to \cite[Table 1, row 2(1)]{H09} and is of standard type $A_2$. Hence by \cite{An13,An15}, $\dim\BN(X\bigoplus X_{i,j})=2|B|$. It follows that $\dim\BN(V_{i,j})\geq\frac{1}{2}\dim \BN(X\bigoplus X_{i,j})=|B|$. Consequently, $\BN\cong\BN(V_{i,j})$.

If $1+\xi^{-ij}= 0$, then from the preceding proof, relations \eqref{eq:H4p-relations-Nichols-algebra-quad-1} must hold in $\BN(V_{i,j})$. Consider the quotient  $\mathfrak{B}$ of $T(V_{i,j})$ by relations \eqref{eq:H4p-relations-Nichols-algebra-quad-1}. Observe that $\xi^{-2j}\neq 1$. Let $A:=\theta^{-2}\xi^{-(i+1)(2+j)}(1+\xi^{pi+j})$ for short. Then for $n\in\N$, using the braiding of Proposition \ref{probraidsimpletwo}, by induction on $n>1$, we have
\begin{align*}
\partial_2(v_2^n)&=(n)_{(-1)^i\xi^{-j}}v_2^{n-1};\\
\partial_1(v_2^n)&=A\sum_{k=1}^{n-1}(-1)^{(i+1)(k-1)} (n-k)_{(-1)^i\xi^{-j}}v_2^{k-1}v_1v_2^{n-k-1}\\
&=A\sum_{k=1}^{n-1}(-1)^{(i+1)(k-1)} (n-k)_{(-1)^i\xi^{-j}}(-1)^{n-k-1}\xi^{-j(n-k-1)}v_2^{n-2}v_1\\
&=A(-1)^{(i+1)n}\frac{1-((-1)^{i}\xi^{-j})^{n-1}-((-1)^{i}\xi^{-j})^{n}+((-1)^{i}\xi^{-j})^{2n-1}}{(1-(-1)^i\xi^{-j})^2(1+(-1)^i\xi^{-j}}v_2^{n-2}v_1.
\end{align*}
Therefore,  $\partial_1(v_2^N)=0=\partial_2(v_2^N)$ and hence relation \eqref{eq:H4p-relations-Nichols-algebra-quad-2} holds in $\BN(V_{i,j})$. Consequently, the quotient $\BN$ of $T(V_{i,j})$ by relations \eqref{eq:H4p-relations-Nichols-algebra-quad-1}--\eqref{eq:H4p-relations-Nichols-algebra-quad-2} projects onto $\BN(V_{i,j})$.

Let $B:=\{v_2^{n_2}v_1^{n_1}\mid n_1\in\I_{0,1},n_2\in\I_{0,N-1} \}$. We claim that the subspace $I$ linearly spanned by $B$ is a left ideal of $\BN$.  Then $B$ linearly  generates $\mathfrak{B}$ since $1\in I$. Indeed, from the defining relations of $\BN$, it is clear that  $v_1I,v_2I\subset I$.

We claim that $\dim\BN(V_{i,j})=|B|$. Indeed, a direct computation shows that the generalized Dynkin diagram of $\BN(X\bigoplus X_{i,j})$ is \xymatrix@C+15pt{\overset{-1 }{{\circ}}\ar
@{-}[r]^{(-1)^i\xi^{-j}} & \overset{-1}{{\circ}}}, which belongs to \cite[Table 1, row 2(2)]{H09} and is of standard type $A_2$. Hence by \cite{An13,An15}, $\dim\BN(X\bigoplus X_{i,j})=2|B|$. It follows that $\dim\BN(V_{i,j})\geq\frac{1}{2}\dim \BN(X\bigoplus X_{i,j})=|B|$. Consequently, $\BN\cong\BN(V_{i,j})$.
\end{proof}
\begin{rmk}
Nichols algebras in Proposition \ref{pro:relations-Nichols-algebra-2-dim-simple-quad} have already appeared in \cite[Proposition 3.10 and 3.11]{AGi17} as examples of Nichols algebras of non-diagonal type.
\end{rmk}

To simplify the exposition, we consider the following subsets of $\Lambda_p$:
$$\Lambda^1_p=\{(i,j)\in\Lambda\mid 1+\xi^{-ij}=0\},$$
$$\Lambda^2_p=\{(i,j)\in\Lambda\mid 1+\xi^{(p-j)(i+1)}=0\},$$
$$\Lambda^3_p=\{(i,j)\in\Lambda^2_p\mid \xi^{2(i+1)}=1,~1+\xi^{2j}=0\}.$$
  Note that $\Lambda^3_p\neq\emptyset$ if and only if $p$ is even and $\frac{p}{2}$ is odd and $\Lambda^3_p=\{(p-1,\frac{p}{2}),(p-1,\frac{3p}{2})\}$. In particular, if $(i,j)\in\Lambda^3_p$, then $(-1)^i=-1$, $\xi^{-ij}=-\xi^j$ and $\ord (\xi^{-ij})=4$.
\begin{rmk}\label{rmk:relations-Nichols-algebra-2-dim-simple-quad}
By Proposition \ref{pro:relations-Nichols-algebra-2-dim-simple-quad}, $\BN(V_{i,j})$ admits quadratic relations if and only if $(i,j)\in\Lambda_p^1\cup\Lambda_p^2$.
\end{rmk}
\begin{pro}
\begin{enumerate}

\item Suppose $V_{i,j}=\K v_1\bigoplus\K v_2 , V_{k,\ell}=\K w_1\bigoplus\K w_2\in{}_{\cH_{p,-1}}^{\cH_{p,-1}}\mathcal{YD}$, where $(i,j)$, $(k,\ell)\in\Lambda^1_p$ satisfy $kj+i\ell\equiv 0\mod 2p$ and $p(i+k)+j+\ell\equiv 0\mod 2p$. Then $\BN(V_{i,j}\bigoplus V_{k,\ell})$ is generated by $v_1, v_2, w_1, w_2$, subject to the relations
      \begin{gather}
      v_1^2=0,\quad v_1v_2+\xi^{-j}v_2v_1=0,\quad v_2^N=0, \text{~where~} N=\ord ((-1)^i\xi^{-j}),\\
      w_1v_1-\xi^{-i\ell}v_1w_1=0, \\
       w_2v_1-\xi^{(k+1)j}v_1w_2-(-1)^i\xi^{i-k}(w_1v_2-\xi^{-(i+1)\ell}v_2w_1)=0,\\
      w_2v_2-\xi^{(i+1)(p-\ell)}v_2w_2-\theta^{-2}\xi^{(i+1)(p-1-\ell)+(p-1-k)}((-1)^k+\xi^{\ell})v_1w_1=0.
      \end{gather}
  \item Suppose $V_{i,j}=\K v_1\bigoplus\K v_2 , V_{k,\ell}=\K w_1\bigoplus \K w_2\in{}_{\cH_{p,-1}}^{\cH_{p,-1}}\mathcal{YD}$, where $(i,j),(k,\ell)\in\Lambda^2_p$ satisfy $kj+i\ell\equiv 0\mod 2p$ and $p(i+k)+j+\ell\equiv 0\mod 2p$.
      Then $\BN(V_{i,j}\bigoplus V_{k,\ell})$ is generated by $v_1, v_2, w_1, w_2$, subject to the relations
      \begin{gather}
      v_1v_2+(-1)^{i+1}v_2v_1=0,\  v_2^2+(-1)^{i+1}(\theta\xi^{i+1})^{-2}v_1^2=0,\
v_1^N=0,  \\
      w_1v_1-\xi^{-i\ell}v_1w_1=0, \\
       w_2v_1-\xi^{(k+1)j}v_1w_2-(-1)^i\xi^{i-k}(w_1v_2-\xi^{-(i+1)\ell}v_2w_1)=0,\\
      w_2v_2-\xi^{(i+1)(p-\ell)}v_2w_2-\theta^{-2}\xi^{(i+1)(p-1-\ell)+(p-1-k)}((-1)^k+\xi^{\ell})v_1w_1=0,
      \end{gather}
      where $N=\ord(\xi^{-ij})$.
\end{enumerate}
\end{pro}
\begin{proof}
We prove the statement only for $(i,j),(k,\ell)\in\Lambda_p^1$, being the proof for $(i,j),(k,\ell)\in\Lambda_p^2$ completely analogous.
Let $x_1=\theta^{-1}\xi^{p-1-i}((-1)^i+\xi^j)$, $x_2=\theta\xi^{p+1+i}((-1)^i-\xi^j)$, $y_1=\theta^{-1}\xi^{p-1-k}((-1)^k+\xi^\ell)$ and $y_2=\theta\xi^{p+1+k}((-1)^k-\xi^\ell)$ for short. Since
\begin{align*}
c(v_1\otimes v_1)=\xi^{-ij}v_1\otimes v_1,\quad  c(v_1\otimes w_1)=\xi^{-jk}w_1\otimes v_1,\\
c(v_2\otimes w_1)=(-1)^{k}\xi^{-kj}w_1\otimes v_2,\quad  c(w_1\otimes w_1)=\xi^{-k\ell}w_1\otimes w_1;
\end{align*}
it follows that $\K v_1\bigoplus \K w_1\in{}_{\cH_{p,-1}}^{\cH_{p,-1}}\mathcal{YD}$ is a quantum linear space and hence $w_1v_1-\xi^{-i\ell}v_1w_1=0$ in $\BN(V_{i,j}\bigoplus V_{k,\ell})$. Let $r_1:=w_2v_1-\xi^{(k+1)j}v_1w_2-(-1)^i\xi^{i-k}(w_1v_2-\xi^{-(i+1)\ell}v_2w_1)$ for short. Observe that $(-1)^k\xi^{-\ell}=(-1)^i\xi^j$ and $\xi^{-i\ell}=\xi^{kj}$. We have
\begin{align*}
(-1)^i\xi^{-i\ell}-\xi^{(k+1)j}=(-1)^i\xi^{i-k}\xi^{p+1+k}\xi^{-(i+1)(\ell+1)}((-1)^k-\xi^{\ell}).
\end{align*}

Since
\begin{align*}
c(v_1\otimes w_2)=\xi^{-(k+1)j}w_2\otimes v_1+x_2\theta^{-1}\xi^{-(k+1)(1+j)}w_1\otimes v_2,\quad
c(w_2\otimes v_1)=(-1)^i\xi^{-i\ell}v_1\otimes w_2,\\
c(w_1\otimes v_2)=\xi^{-(i+1)\ell}v_2\otimes w_1+y_2\theta^{-1}\xi^{-(i+1)(\ell+1)}v_1\otimes w_2,\quad c(v_2\otimes w_1)=(-1)^{k}\xi^{-kj}w_1\otimes v_2,
\end{align*}
it follows from a direct computation that
\begin{align*}
\partial_{v_1}(w_2v_1)&=(-1)^i\xi^{-i\ell}w_2,\quad \partial_{v_1}(v_1w_2)=w_2,\\
\partial_{v_1}(v_2w_1)&=0,\quad \partial_{v_1}(w_1v_2)=y_2\theta^{-1}\xi^{-(i+1)(\ell+1)}w_2,\\
\partial_{v_2}(w_2v_1)&=0,\quad \partial_{v_2}(v_1w_2)=0,\quad \partial_{v_2}(v_2w_1)=w_1,\quad \partial_{v_2}(w_1v_2)=\xi^{-(1+i)\ell}w_1,\\
\partial_{w_1}(w_2v_1)&=0,\quad \partial_{w_1}(v_1w_2)=x_2\theta^{-1}\xi^{-(k+1)(j+1)}v_2,\\
\partial_{w_1}(v_2w_1)&=(-1)^k\xi^{-kj}v_2,\quad \partial_{w_1}(w_1v_2)=v_2,\\
\partial_{w_2}(w_2v_1)&=v_1,\quad \partial_{w_2}(v_1w_2)=\xi^{-(k+1)j}v_1,\quad \partial_{w_2}(v_2w_1)=0,\quad \partial_{w_2}(w_1v_2)=0.
\end{align*}
Therefore, $\partial_{t}(r_1)=0$ for any $t\in\{v_1,\,v_2,\,w_1,\,w_2\}$ and hence $r_1=0$ in $\BN(V_{i,j}\bigoplus V_{k,\ell})$.

Let $r_2:=w_2v_2-\xi^{(i+1)(p-\ell)}v_2w_2-\theta^{-2}\xi^{(i+1)(p-1-\ell)+(p-1-k)}((-1)^k+\xi^{\ell})v_1w_1$ for short. Since
\begin{align*}
c(v_2\otimes w_2)&=\xi^{(p-j)(k+1)}w_2\otimes v_2+x_1\theta^{-1}\xi^{(p-1-j)(k+1)}w_1\otimes v_1,\\
c(w_2\otimes v_2)&=\xi^{(i+1)(p-\ell)}v_2\otimes w_2+y_1\theta^{-1}\xi^{(i+1)(p-1-\ell)}v_1\otimes w_1,
\end{align*}
 it follows from a direct computation that
\begin{align*}
\partial_{v_1}(w_2v_2)&=y_1\theta^{-1}\xi^{(i+1)(p-1-\ell)}w_1,\quad \partial_{v_1}(v_2w_2)=0,\quad \partial_{v_1}(v_1w_1)=w_1,\\
\partial_{v_2}(w_2v_2)&=\xi^{(i+1)(p-\ell)}w_2,\quad \partial_{v_2}(v_2w_2)=w_2,\quad \partial_{v_2}(v_1w_1)=0,\\
\partial_{w_1}(w_2v_2)&=0,\quad \partial_{w_1}(v_2w_2)=x_1\theta^{-1}\xi^{(k+1)(p-1-j)}v_1,\quad \partial_{w_1}(v_1w_1)=\xi^{-jk}v_1,\\
\partial_{w_2}(w_2v_2)&=v_2,\quad \partial_{w_2}(v_2w_2)=\xi^{(k+1)(p-j)}v_2,\quad \partial_{w_2}(v_1w_1)=0.
\end{align*}
Therefore, $\partial_{t}(r_2)=0$ for any $t\in\{v_1,v_2,w_1,w_2\}$ and hence $r_2=0$ in $\BN(V_{i,j}\bigoplus V_{k,\ell})$.

Hence the quotient $\BN$ of $\BN(V_{i,j})\otimes\BN(V_{k,\ell})$ by relations $w_1v_1-\xi^{-i\ell}v_1w_1=0$ and $r_1=0=r_2$ projects onto $\BN(V_{i,j}\bigoplus V_{k,\ell})$.

Let $B:=\{v_1^{n_1}v_2^{n_2}(w_1v_2)^{n_{12}}w_1^{n_3}w_2^{n_4}\mid n_1,\, n_{12},\, n_3\in\I_{0,1},\, n_2, \,n_4\in\I_{0,N-1}\}$. Note that $|B|=128$. We claim that the subspace $I$ linearly spanned by $B$ is a left ideal of $\BN$.  Then $B$ linearly  generates $\mathfrak{B}$ since $1\in I$. Indeed, it suffices to show that $gI\subset I$ for $g\in\{v_1,v_2,w_1,w_2\}$, which can be obtained easily from the defining relations of $\BN$.

We claim that $\dim\BN(V_{i,j}\bigoplus V_{k,\ell})\geq 8N^2=|B|$. Indeed, a direct computation shows that the Dynkin diagram of $\BN(X\bigoplus X_{i,j}\bigoplus X_{k,\ell})$ is \xymatrix@C+9pt{
\overset{-1}{\underset{v_1}{\circ}}\ar  @ {-}[r]^{(-1)^i\xi^{-j}}  & \overset{-1}{\underset{x
}{\circ}}\ar  @{-}[r]^{(-1)^k\xi^{-\ell}}
& \overset{-1}{\underset{w_1}{\circ}}}, which belongs to \cite[Table 2, Row 8]{H09} and is of standard $A_2$ type.  By \cite{An13,An15}, $\dim\BN(X\bigoplus X_{i,j}\bigoplus X_{k,\ell})=16N^2$.   Then  $\dim\BN(V_{i,j}\bigoplus V_{k,\ell})\geq\frac{1}{2}\dim \BN(X\bigoplus X_{i,j}\bigoplus X_{k,\ell})=8N^2$. Consequently, $\BN\cong\BN(V_{i,j}\bigoplus V_{k,\ell})$.
\end{proof}

\begin{pro}
Assume $V_{i,j}=\K v_1\bigoplus\K v_2 , \K_{\chi^k}=\K v_3\in {}_{\cH_{p,-1}}^{\cH_{p,-1}}\mathcal{YD}$ for $(i,j)\in\Lambda_p^{1}$ and odd number $k$ such that $(k+1)(pi-j)\equiv 0\mod 2p$. Then $\BN(V_{i,j}\bigoplus\K_{\chi^k})$ is generated by $v_1,\,v_2,\,v_3$, subject to the relations
\begin{gather}
v_1^2=0,\quad v_1v_2+\xi^{-j}v_2v_1=0,\quad v_2^N=0,\quad v_3^2=0, \text{~where~} N=\ord ((-1)^i\xi^{-j}),\label{eq:V-2-1-1}\\
(v_3v_1)^N+(-1)^{iN}(v_1v_3)^N=0,\label{eq:V-2-1-2}\\
(-1)^iv_1v_2v_3+v_1v_3v_2+(-1)^iv_2v_3v_1+v_3v_2v_1=0,\label{eq:V-2-1-3}\\
v_3v_2^2+((-1)^i+\xi^{-j})v_2v_3v_2+(-1)^i\xi^{-j}v_2^2v_3+(-1)^{i+1}\theta^{-2}\xi^{-(i+1)(2+j)}(1+\xi^{pi+j})v_1v_3v_1=0.\label{eq:V-2-1-4}
\end{gather}
\end{pro}
\begin{proof}
Since $\xi^{-ij}=-1$ and $\xi^{(k+1)(pi-j)}=1$,  $\xi^{(p-j)(i+1)}=\xi^{pi-j}$ and $\xi^{-kj}=\xi^j$, the braiding of $V_{i,j}\bigoplus\K_{\chi^k}$ is given as follows: $c(\left[\begin{array}{ccc} v_1\\v_2\\v_3\end{array}\right]\otimes\left[\begin{array}{ccc} v_1~v_2~v_3\end{array}\right])=$
\begin{align}\label{BraidingTwoone}
\left[\begin{array}{ccc}
   -v_1\otimes v_1    & -\xi^{-j}v_2\otimes v_1+(\xi^{pi-j}-1)v_1\otimes v_2 & \xi^{j}v_3\otimes v_1\\
   (-1)^{i+1}v_1\otimes v_2&\xi^{pi-j}v_2\otimes v_2+\theta^{-2}\xi^{-(i+1)(2+j)}(1+\xi^{pi+j})v_1\otimes v_1&-\xi^jv_3\otimes v_2\\
   (-1)^{i}v_1\otimes v_3 & (-1)^{(i+1)}v_2\otimes v_3 & - v_3\otimes v_3
         \end{array}\right].
\end{align}
Observe that $\K v_1\bigoplus\K v_3$ is a braided vector space of diagonal type. Then by \cite{An13}, $(v_3v_1+(-1)^iv_1v_3)^N=(v_3v_1)^N+(-1)^{iN}(v_1v_3)^N=0$ in $\BN(V_{i,j}\bigoplus\K_{\chi^k})$. Let $\alpha=\theta^{-2}\xi^{-(i+1)(2+j)}(1+\xi^{pi+j})$. It follows from a direct computation that
\begin{align*}
\partial_1(v_1v_2v_3)&=(-1)^i\xi^{-j}v_2v_3,\quad \partial_2(v_1v_2v_3)=-\xi^{-j}v_1v_3,\quad \partial_3(v_1v_2v_3)=-\xi^{2j}v_1v_2,\\
\partial_1(v_1v_3v_2)&=(-1)^{i+1}(\xi^{pi-j}-1)v_2v_3+v_3v_2,\  \partial_2(v_1v_3v_2)=(-1)^i\xi^{-j}v_1v_3,\  \partial_3(v_1v_3v_2)=\xi^jv_1v_2,\\
\partial_1(v_2v_3v_1)&=-v_2v_3,\quad \partial_2(v_2v_3v_1)=v_3v_1,\quad \partial_3(v_2v_3v_1)=\xi^{2j}v_1v_2,\\
\partial_1(v_3v_2v_1)&=-v_3v_2,\quad \partial_2(v_3v_2v_1)=(-1)^{i+1}v_3v_1,\quad \partial_3(v_3v_2v_1)=-\xi^jv_1v_2,\\
\partial_1(v_3v_2^2)&=(-1)^i\alpha v_3v_1,\quad\partial_2(v_3v_2^2)=(-1)^{i+1}(1+\xi^{pi-j})v_3v_2,\quad\partial_3(v_3v_2^2)=v_2^2,\\
\partial_1(v_2v_3v_2)&=(-1)^{i+1}\alpha v_1v_3,\quad\partial_2(v_2v_3v_2)=v_3v_2-\xi^{-j}v_2v_3,\quad\partial_3(v_3v_2^2)=-\xi^jv_2^2,\\
\partial_1(v_2^2v_3)&=\alpha v_1v_3,\quad\partial_2(v_2^2v_3)=(1+\xi^{pi-j})v_2v_3,\quad \partial_3(v_2^2v_3)=\xi^{2j}v_2^2,\\
\partial_1(v_1v_3v_1)&=v_3v_1+(-1)^{i+1}v_1v_3,\quad \partial_2(v_1v_3v_1)=0,\quad \partial_3(v_1v_3v_1)=0.
\end{align*}
It follows from a direct computation that relations \eqref{eq:V-2-1-3} and \eqref{eq:V-2-1-4} are annihilated by $\partial_i$, $i=1,2,3$.
Hence the quotient $\BN$ of $\BN(V_{i,j})\otimes\BN(\K_{\chi^k})$ by relations \eqref{eq:V-2-1-2}, \eqref{eq:V-2-1-3} and \eqref{eq:V-2-1-4} projects onto $\BN(V_{i,j}\bigoplus \K_{\chi^k})$.

Let $B:=\{v_1^{n_1}v_2^{n_2}(v_3v_1)^{n_{31}}(v_3v_2)^{n_{32}}v_3^{n_3}\mid n_2,\,n_{32}\in\I_{0,N-1},\,n_1,\,n_{31},\,n_3\in\I_{0,1}\}$. Note that $|B|=8N^2$. We claim that the subspace $I$ linearly spanned by $B$ is a left ideal of $\BN$.  Then $B$ linearly  generates $\mathfrak{B}$ since $1\in I$. Indeed,  it suffices to show that $gI\subset I$ for $g\in\{v_1,\,v_2,\,v_3\}$, which  can be obtained easily from the defining relations of $\BN$.

We claim that $\dim\BN(V_{i,j}\bigoplus \K_{\chi^k})\geq 8N^2=|B|$. Indeed, a direct computation shows that the Dynkin diagram of $\BN(X\bigoplus X_{i,j}\bigoplus Y_k)$ is \xymatrix@C+9pt{
\overset{-1}{{\circ}}\ar  @ {-}[r]^{(-)^i\xi^{-j}}  & \overset{-1}{{\circ}}\ar  @{-}[r]^{(-)^i\xi^{j}}
& \overset{-1}{{\circ}}},  which belongs to \cite[Table 2, Row 8]{H09} and is of standard $A_2$ type. Hence by \cite{An13,An15}, $\dim\BN(X\bigoplus X_{i,j}\bigoplus Y_k)=16N^2$. It follows that $\dim\BN(V_{i,j}\bigoplus \K_{\chi^k})\geq\frac{1}{2}\dim \BN(X\bigoplus X_{i,j}\bigoplus Y_k)=8N^2$. Consequently, $\BN\cong\BN(V_{i,j}\bigoplus \K_{\chi^k})$.
\end{proof}

\begin{pro}\label{pro:relations-Nichols-algebra-2-1-nonnnnnn}
Assume $V_{i,j}=\K v_1\bigoplus\K v_2 , \K_{\chi^k}=\K v_3\in {}_{\cH_{p,-1}}^{\cH_{p,-1}}\mathcal{YD}$, where $(i,j)\in\Lambda_p^{2}$ and $k$ is odd such that $(k-1)(pi-j)\equiv 0\mod 2p$. Set $ N=\ord(\xi^{-ij})$. If $N\in\{3,4\}$, then $\BN(V_{i,j}\bigoplus\K_{\chi^k})$ is generated by $v_1,v_2,v_3$, subject to the relations
\begin{gather}
v_1v_2+(-1)^{i+1}v_2v_1=0,\quad v_2^2+(-1)^{i+1}(\theta\xi^{i+1})^{-2}v_1^2=0,\quad
v_1^N=0,\quad v_3^2=0,\label{eq:V-2-1-1-1}\\
(-1)^{i+1}\xi^jv_1v_2v_3-\xi^jv_1v_3v_2+v_2v_3v_1+(-1)^iv_3v_2v_1=0,\label{eq:V-2-1-1-2}\\
v_3v_1^2+[(-1)^{i+1}-\xi^j]v_1v_3v_1+(-1)^i\xi^jv_1^2v_3=0,\label{eq:V-2-1-1-3}\\
\alpha(-1)^{Ni}\sum_{\ell=0}^{N-2}\xi^{2j\ell}(v_1v_3)^2(v_2v_3)^{N-2}+(-1)^{Ni}(v_2v_3)^N+(v_3v_2)^N=0,\label{eq:V-2-1-1-4}
\end{gather}
where $\alpha=\theta^{-2}\xi^{-(i+1)(2+j)}(1+\xi^{pi+j})$.
\end{pro}
\begin{proof}
Since $\xi^{(i+1)(p-j)}=-1$ and $\xi^{(k-1)(pi-j)}=1$,  $\xi^{-ij}=(-1)^i\xi^j$ and $\xi^{-kj}=\xi^{-j}$. The braiding of $V_{i,j}\bigoplus\K_{\chi^k}$ is given as follows: $c(\left[\begin{array}{ccc} v_1\\v_2\\v_3\end{array}\right]\otimes\left[\begin{array}{ccc} v_1~v_2~v_3\end{array}\right])=$
\begin{align*}
\left[\begin{array}{ccc}
   (-1)^i\xi^{j}v_1\otimes v_1    & (-1)^iv_2\otimes v_1+[(-1)^i\xi^{j}-1]v_1\otimes v_2 & \xi^{-j}v_3\otimes v_1\\
   \xi^{j}v_1\otimes v_2&-v_2\otimes v_2+\theta^{-2}\xi^{-(i+1)(2+j)}(1+\xi^{pi+j})v_1\otimes v_1&-\xi^{-j}v_3\otimes v_2\\
   (-1)^{i}v_1\otimes v_3 & (-1)^{(i+1)}v_2\otimes v_3 & - v_3\otimes v_3
         \end{array}\right].
\end{align*}
It is clear that relations \eqref{eq:V-2-1-1-1} hold in $\BN(V_{i,j}\bigoplus\K_{\chi^k})$. Observe that $\K v_1\bigoplus\K v_3$ is a braided vector space of diagonal type. Then by \cite{An13}, relation \eqref{eq:V-2-1-1-3} must hold in $\BN(V_{i,j}\bigoplus\K_{\chi^k})$.  It follows from a direct computation that
\begin{align*}
\partial_1(v_1v_2v_3)&=(-1)^i\xi^{j}v_2v_3,\quad \partial_2(v_1v_2v_3)=(-1)^iv_1v_3,\quad \partial_3(v_1v_2v_3)=-\xi^{-2j}v_1v_2,\\
\partial_1(v_1v_3v_2)&=(-1)^{i+1}(\xi^{pi+j}-1)v_2v_3+v_3v_2,\quad \partial_2(v_1v_3v_2)=-v_1v_3,\quad \partial_3(v_1v_3v_2)=\xi^{-j}v_1v_2,\\
\partial_1(v_2v_3v_1)&=(-1)^i\xi^jv_2v_3,\quad \partial_2(v_2v_3v_1)=v_3v_1,\quad \partial_3(v_2v_3v_1)=(-1)^{i+1}\xi^{-j}v_1v_2,\\
\partial_1(v_3v_2v_1)&=(-1)^i\xi^jv_3v_2,\quad \partial_2(v_3v_2v_1)=(-1)^{i+1}v_3v_1,\quad \partial_3(v_3v_2v_1)=(-1)^iv_1v_2.
\end{align*}
Then relations \eqref{eq:V-2-1-1-2} are annihilated by $\partial_i$, $i=1,2,3$. Then
 by induction, we have
\begin{align*}
\partial_1((v_3v_2)^N)&=-\alpha\sum_{\ell=0}^{N-2}(-1)^{Ni}\xi^{2j\ell}v_3v_1v_3(v_2v_3)^{N-2},\\
\partial_2((v_3v_2)^N)&=(-1)^{Ni+1}v_3(v_2v_3)^{N-1},\\
\partial_3((v_3v_2)^N)&=\sum_{\ell=0}^{N-1}\xi^{-\ell j}(v_3v_2)^{\ell}v_2(v_3v_2)^{N-1-\ell},\\
\partial_1((v_2v_3)^N)&=0,\quad \partial_2((v_2v_3)^N)=v_3(v_2v_3)^{N-1},\\
\partial_3((v_2v_3)^N)&=-\xi^{-Nj}\sum_{\ell=0}^{N-1}\xi^{\ell j}(v_2v_3)^{N-1-\ell}v_2(v_2v_3)^{\ell},\\
\partial_1((v_1v_3)^2(v_2v_3)^{N-2})&=v_3v_1v_3(v_2v_3)^{N-2},\quad
\partial_2((v_1v_3)^2(v_2v_3)^{N-2})=0,\\
\partial_3((v_1v_3)^2(v_2v_3)^{N-2})&=\xi^{-j}v_1^2v_3(v_2v_3)^{N-2}-\xi^{-2j}v_1v_3v_1(v_2v_3)^{N-2}+\xi^{-2j}(v_1v_3)^2\partial_3((v_2v_3)^{N-2}).
\end{align*}
Then it follows from a direct computation that relation \eqref{eq:V-2-1-1-4} is annihilated by $\partial_i$, $i=1,2,3$. Hence the quotient $\BN$ of $\BN(V_{i,j})\otimes\BN(\K_{\chi^k})$ by relations \eqref{eq:V-2-1-1-2}--\eqref{eq:V-2-1-1-4} projects onto $\BN(V_{i,j}\bigoplus \K_{\chi^k})$.

We claim that $\dim\BN(V_{i,j}\bigoplus \K_{\chi^k})\geq 8N^2=|B|$. Indeed, a direct computation shows that the Dynkin diagram of $\BN(X\bigoplus X_{i,j}\bigoplus Y_k)$ is \xymatrix@C+9pt{
\overset{-1}{ {\circ}}\ar  @ {-}[r]^{ \xi^{ij}}  & \overset{\xi^{-ij}}{ {\circ}}\ar  @{-}[r]^{ \xi^{ij}}
& \overset{-1}{ {\circ}}}, which belongs to \cite[Table 2, Row 8]{H09} and is of standard $A_2$ type. By \cite{An13,An15}, $\dim\BN(X\bigoplus X_{i,j}\bigoplus Y_k)=16N^2$. It follows that $\dim\BN(V_{i,j}\bigoplus \K_{\chi^k})\geq\frac{1}{2}\dim \BN(X\bigoplus X_{i,j}\bigoplus Y_k)=8N^2$. Consequently, $\BN\cong\BN(V_{i,j}\bigoplus \K_{\chi^k})$.
\end{proof}
\begin{rmk}
If we remove the condition $N\in\{3,4\}$ in Proposition \ref{pro:relations-Nichols-algebra-2-1-nonnnnnn}, then from the proof of Proposition \ref{pro:relations-Nichols-algebra-2-1-nonnnnnn}, relations \eqref{eq:V-2-1-1-1}--\eqref{eq:V-2-1-1-3} must hold in $\BN(V_{i,j}\bigoplus\K_{\chi^k})$. Moreover, in this case, the defining relations of $\BN(V_{i,j}\bigoplus\K_{\chi^k})$ in ${}_{\gr\A_{p,-1}}^{\gr\A_{p,-1}}\mathcal{YD}$ are given by
\begin{gather*}
[v_1,[v_{2},v_3]]=0,\quad [v_{2}v_1]=0,\quad [v_1v_1v_3]=0,\\
 v_{2}^2=0,\quad  [v_{2}v_3]^N=0,\quad  v_1^N=0,  \quad v_3^2=0.
\end{gather*}
In particular, relation \eqref{eq:V-2-1-1-4} in ${}_{\cH_{p,-1}}^{\cH_{p,-1}}\mathcal{YD}$ corresponds to the relation $[v_{2}v_3]^N=0$ in ${}_{\gr\A_{p,-1}}^{\gr\A_{p,-1}}\mathcal{YD}$.

It is very hard to show that relation \eqref{eq:V-2-1-1-4} holds when removing the condition because of the complicated commutation relations of PBW bases in ${}_{\cH_{p,-1}}^{\cH_{p,-1}}\mathcal{YD}$. We leave it to future work.

\begin{conj}\label{conj:pro:relations-Nichols-algebra-2-1-nonnnnnn}
The Nichols algebra $\BN(V_{i,j}\bigoplus\K_{\chi^k})$ in Proposition \ref{pro:relations-Nichols-algebra-2-1-nonnnnnn} is generated by $v_1,\,v_2$ and $v_3$, subject to relations \eqref{eq:V-2-1-1-1}--\eqref{eq:V-2-1-1-4}.
\end{conj}
\end{rmk}

\begin{pro}\label{pro-B2-non-Cartan-1}
Suppose $V_{i,j}=\K v_1\bigoplus\K v_2 \in{}_{\cH_{p,-1}}^{\cH_{p,-1}}\mathcal{YD}$ such that $pi-j-2ij\equiv 0\mod 2p$. Set $N=\ord(\xi^{-ij})$, if $N=3$ or $6$, then $\BN(V_{i,j})$ is generated by $v_1,v_2$, subject to the relations
\begin{gather}
v_1^N=0, \label{eq:R-B2-Nichols-alg-1}\\
(1+\xi^{-ij})v_1v_2v_1+(-1)^{i+1}\xi^{-ij}v_1^2v_2+(-1)^{i+1}v_2v_1^2=0,\label{eq:R-B2-Nichols-alg-2}\\
\alpha v_1^3+v_1v_2^2-(1-\xi^{ij})\xi^{-j(i+1)}v_2v_1v_2-\xi^{-j(i+2)}v_2^2v_1=0,\label{eq:R-B2-Nichols-alg-3}\\
v_2^{\frac{N}{2}}+\sum_{k=0}^{\frac{N}{2}-1}(-1)^{k(i+1)}v_1^kv_2v_1^{\frac{N}{2}-1-k},\  \text{if $N$ is even};\quad
v_2^{2N}=0,\ \text{if $N$ is odd}\label{eq:R-B2-Nichols-alg-4},
\end{gather}
where $\alpha=0$ if $N=3$, otherwise, $\alpha=-\theta^{-2}\xi^{-(i+1)(2+j)}(1+\xi^{2ij})(\xi^{-ij}-1)$.
\end{pro}
\begin{proof}
Observe that $(-)^i\xi^{-(2i+1)j}=1$. Then $\xi^{2ij}=(-1)^i\xi^{-j}$, $\xi^{-j(i+1)}=(-1)^i\xi^{ij}$ and $\xi^{(p-j)(i+1)}=-\xi^{ij}$. Let $A=\theta^{-2}\xi^{-(i+1)(2+j)}(1+\xi^{pi+j})$.
The braiding of $V_{i,j}$ is given by
  \begin{align*}
   c(\left[\begin{array}{ccc} v_1\\v_2\end{array}\right]\otimes\left[\begin{array}{ccc} v_1~v_2\end{array}\right])=
   \left[\begin{array}{ccc}
   \xi^{-ij}v_1\otimes v_1    & (-1)^i\xi^{ij}v_2\otimes v_1+(\xi^{-ij}-\xi^{ij})v_1\otimes v_2\\
   (-1)^i\xi^{-ij}v_1\otimes v_2   & -\xi^{ij}v_2\otimes v_2+Av_1\otimes v_1
         \end{array}\right].
  \end{align*}

Since $c(v_1\otimes v_1)=\xi^{-ij}v_1\otimes v_1$, it follows that $\partial_1(v_1^N)=0=\partial_2(v_1^N)$ and hence $v_1^N=0$ in $\BN(V_{i,j})$.  Using formula \eqref{eqSkew-1}, a direct computation shows that
\begin{align*}
\partial_1(v_2v_1^2)&=(-1)^i\xi^{-ij}(1+\xi^{-ij})v_2v_1,\quad \partial_2(v_2v_1^2)=v_1^2,\\
\partial_1(v_1v_2v_1)&=(-1)^i\xi^{-2ij}v_1v_2+[1+\xi^{-ij}-\xi^{ij}]v_2v_1,\quad \partial_2(v_1v_2v_1)=\xi^{-j(i+1)}v_1^2,\\
\partial_1(v_1^2v_2)&=\xi^{-ij}(1+\xi^{-ij})v_1v_2+\xi^{-j}(\xi^{-2ij}-1)v_2v_1,\quad\partial_2(v_1^2v_2)=\xi^{-2j(i+1)}v_1^2,\\
\partial_1(v_2^2v_1)&=Av_1^2+\xi^{-2ij}v_2^2,\quad \partial_2(v_2^2v_1)=(1-\xi^{ij})v_2v_1,\\
\partial_1(v_2v_1v_2)&=\xi^{-j(i+1)}Av_1^2+(-1)^i\xi^{-ij}Bv_2^2,\quad\partial_2(v_2v_1v_2)=v_1v_2-\xi^{-j}v_2v_1,\\
\partial_1(v_1v_2^2)&=\xi^{-ij}Av_1^2+(1+(1-\xi^{ij})(\xi^{-ij}-\xi^{ij}))v_2^2,\quad\partial_2(v_1v_2^2)=(1-\xi^{ij})\xi^{-j(i+1)}v_1v_2,\\
\partial_1(v_1^3)&=(1+\xi^{-ij}+\xi^{-2ij})v_1^2,\quad \partial_2(v_1^3)=0.
\end{align*}
Then it follows from a direct computation that relations \eqref{eq:R-B2-Nichols-alg-2}--\eqref{eq:R-B2-Nichols-alg-3} are  annihilated by $\partial_i$, $i=1,2$ and hence they must hold in $\BN(V_{i,j})$. Then by induction,  we have
\begin{align*}
\partial_1(v_2^n)&=A\sum_{k=1}^{n-1}((-1)^i\xi^{-ij})^{k-1}(n-k+1)_{-\xi^{ij}}v_2^{k-1}v_1v_2^{n-k-1},\quad \partial_2(v_2^n)=(n)_{-\xi^{ij}}v_2^{n-1},
\end{align*}
where
\begin{align*}
v_1v_2^n&=X_nv_2^{n-1}v_1v_2+\xi^{-j(i+2)} X_{n-1} v_2^n v_1-\alpha\sum_{i=0}^{n-1} X_{i+1}v_2^i v_1^3 v_2^{n-2-i}, \\
X_n&=\xi^{-(n-1)j(i+1)} (n-1)_{-\xi^{ij}}+(-1)^{n-1}\xi^{-(n-1)j}.
\end{align*}

If $N=3$ or $6$, then it follows from a direct computation that relation \eqref{eq:R-B2-Nichols-alg-4} is annihilated by $\partial_i$, $i=1,2$ and hence they must hold in $\BN(V_{i,j})$. Therefore, the quotient $\BN$ of $T(V_{i,j})$ by relations \eqref{eq:R-B2-Nichols-alg-1}--\eqref{eq:R-B2-Nichols-alg-4} projects onto $\BN(V_{i,j})$.

Let $B:=\{v_1^{n_1}(v_2v_1)^{n_{12}}v_2^{n_2}\mid n_1\in\I_{0,N-1},n_2\in\I_{0,N_2-1},n_{12}\in\I_{0,1}\}$, where $N_2=2N$ if $N$ is odd; otherwise $N_2=N/2$. We claim that the subspace $I$ linearly spanned by $B$ is a left ideal of $\BN$.  Then $B$ linearly generates $\mathfrak{B}$ since $1\in I$. Indeed, it suffices to show that $gI\subset I$ for $g\in\{v_1,v_2\}$, which  can be obtained easily from the defining relations of $\BN$.

We claim that $\dim\BN(V_{i,j})=|B|$. Indeed, a direct computation shows that the Dynkin diagram of $\BN(X\bigoplus X_{i,j})$ is \xymatrix@C+15pt{\overset{-1 }{{\circ}}\ar
@{-}[r]^{\xi^{pi-j}} & \overset{\xi^{-ij}}{{\circ}}} with $\xi^{-2ij}\xi^{pi-j}=1$, which belongs to \cite[Table 1, Row 4]{H09} and  is of standard type $B_2$. Hence by \cite{An13,An15}, $\dim\BN(X\bigoplus X_{i,j})=2|B|$. It follows that $\dim\BN(V_{i,j})\geq\frac{1}{2}\dim \BN(X\bigoplus X_{i,j})=|B|$. Consequently, $\BN\cong\BN(V_{i,j})$.
\end{proof}

\begin{rmk}\label{rmk-B2-non-Cartan-1}
If we remove the condition $N=3$ or $N=6$ in Proposition \ref{pro-B2-non-Cartan-1}, then from the proof of Proposition \ref{pro-B2-non-Cartan-1}, relations \eqref{eq:R-B2-Nichols-alg-1}--\eqref{eq:R-B2-Nichols-alg-3} must hold in $\BN(V_{i,j})$. Moreover, the defining relations of $\BN(V_{i,j})$ in ${}_{\gr\A_{p,-1}}^{\gr\A_{p,-1}}\mathcal{YD}$ are
\begin{gather*}
v_{2}^{N_{2}}=0, \  v_1^N=0,\quad [[v_{2},v_1],v_1]=0,\  (\ad_c v_{2})^2(v_1)=0,\  N_{2}=\ord(-q^{-1}),N=\ord(q).
\end{gather*}
In particular, relation \eqref{eq:R-B2-Nichols-alg-4} in ${}_{\cH_{p,-1}}^{\cH_{p,-1}}\mathcal{YD}$ corresponds to the relation $v_{2}^{N_{2}}=0$ in ${}_{\gr\A_{p,-1}}^{\gr\A_{p,-1}}\mathcal{YD}$.
\begin{conj}\label{conj-B2-non-Cartan-1}
The Nichols algebra $\BN(V_{i,j})$ in Proposition \ref{pro-B2-non-Cartan-1} is generated by $v_1$ and $v_2$, subject to relations \eqref{eq:R-B2-Nichols-alg-1}--\eqref{eq:R-B2-Nichols-alg-4}.
\end{conj}
\end{rmk}

\begin{pro}
Suppose $V_{i,j}=\K v_1\bigoplus\K v_2 \in{}_{\cH_{p,-1}}^{\cH_{p,-1}}\mathcal{YD}$ such that $(p+j)(i-1)\equiv 0\mod 2p$ and $3ij\equiv 0\mod 2p$. Then $\BN(V_{i,j})$ is generated by $v_1,v_2$, subject to the relations
\begin{gather}
v_1^3=0,\quad v_1v_2^2+(-1)^{i+1}v_2v_1v_2+v_2^2v_1=0,\label{eq:V-B2-no-1}\\
v_1^2v_2+(-1)^{i+1}\xi^{2j}v_1v_2v_1+\xi^{-2j}v_2v_1^2=0,\label{eq:V-B2-no-2}\\
v_2^3-\alpha(-1)^i(\xi^{-j}-\xi^{-2j})v_1^2v_2-\alpha(\xi^{-j}+\xi^{-2j})v_2v_1^2-\alpha v_1v_2v_1=0,\label{eq:V-B2-no-3}
\end{gather}
where $\alpha=\theta^{-2}\xi^{-(i+1)(2+j)}(1+\xi^{pi+j})$.
\end{pro}
\begin{proof}
Observe that $\xi^{-ij}=(-1)^{i+1}\xi^{-j}$, $\xi^{(i+1)(p-j)}=\xi^{-2j}$, $1+(-1)^{i+1}\xi^{-j}+\xi^{-2j}=0$ and $1+\xi^{-2j}+\xi^{-4j}=0$. The braiding of $V_{i,j}$ is given by
\begin{align*}
   c(\left[\begin{array}{ccc} v_1\\v_2\end{array}\right]\otimes\left[\, v_1~v_2\,\right])=
   \left[\begin{array}{ccc}
   (-1)^{i+1}\xi^{-j}v_1\otimes v_1    & (-1)^{i+1}\xi^{-2j}v_2\otimes v_1+(\xi^{-2j}-(-1)^{i}\xi^{-j})v_1\otimes v_2\\
   -\xi^{-j}v_1\otimes v_2   & \xi^{-2j}v_2\otimes v_2+\alpha v_1\otimes v_1
         \end{array}\right].
\end{align*}
Since $c(v_1\otimes v_1)=(-1)^{i+1}\xi^{-j}v_1\otimes v_1$ and $\xi^{-3j}=(-1)^{i+1}$, we have $v_1^3=0$ in $\BN(V_{i,j})$. Furthermore, using the braiding matrix given as above, it follows from a direct computation that
\begin{align*}
\partial_1(v_1v_2^2)&=(-1)^{i+1}\xi^{-j}\alpha v_1^2-\xi^{-2j}v_2^2,\quad \partial_2(v_1v_2^2)=(-1)^{i+1}(\xi^{-2j}+\xi^{-4j})v_1v_2,\\
\partial_1(v_2v_1v_2)&=(-1)^{i+1}\xi^{-2j}\alpha v_1^2,\quad \partial_2(v_2v_1v_2)=v_1v_2+(-1)^{i+1}\xi^{-4j}v_2v_1,\\
\partial_1(v_2^2v_1)&=\alpha v_1^2+\xi^{-2j}v_2^2,\quad \partial_2(v_2^2v_1)=(1+\xi^{-2j})v_2v_1,\\
\partial_1(v_1^2v_2)&=v_1v_2+[\xi^{-3j}+(-1)^{i+1}\xi^{-4j}]v_2v_1,\quad \partial_2(v_1^2v_2)=\xi^{-4j}v_1^2,\\
\partial_1(v_2v_1^2)&=-\xi^{-j}(1+(-1)^{i+1}\xi^{-j})v_2v_1,\quad \partial_2(v_2v_1^2)=v_1^2,\\
\partial_1(v_1v_2v_1)&=(-1)^i\xi^{-2j}v_1v_2,\quad \partial_2(v_1v_2v_1)=(-1)^{i+1}\xi^{-2j}v_1^2,\\
\partial_1(v_2^3)&=\alpha(1+\xi^{-2j})v_1v_2-\xi^{-j}\alpha v_2v_1,\quad \partial_2(v_2^3)=(1+\xi^{-2j}+\xi^{-4j})v_2^2=0.
\end{align*}
Then relations \eqref{eq:V-B2-no-1}--\eqref{eq:V-B2-no-3} are annihilated by $\partial_i$, $i=1,2$ and hence they must hold in $\BN(V_{i,j})$.
Therefore, the quotient $\BN$ of $\BN(V_{i,j})$ by relations \eqref{eq:V-B2-no-1}--\eqref{eq:V-B2-no-3} projects onto $\BN(V_{i,j})$.

Let $B:=\{v_1^{n_1}(v_2v_1)^{n_{12}}v_2^{n_2}\mid n_1,\,n_2\in\I_{0,2},\,n_{12}\in\I_{0,1}\}$. Note that $|B|=18$. We claim that  the subspace $I$ linearly spanned by $B$ is a left ideal of $\BN$.  Then $B$ linearly  generates $\mathfrak{B}$ since $1\in I$. Indeed, it suffices to show that $gI\subset I$ for $g\in\{v_1,v_2\}$, which can be obtained easily from the defining relations of $\BN$.

We claim that $\dim\BN(V_{i,j})\geq 18=|B|$. Indeed, a direct computation shows that the Dynkin diagram of $\BN(X\bigoplus X_{i,j})$ is \xymatrix@C+15pt{\overset{-1 }{{\circ}}\ar
@{-}[r]^{\xi^{pi-j}} & \overset{\xi^{-ij}}{{\circ}}} with $\xi^{-3ij}=1$ and $\xi^{-ij}=(-1)^{i+1}\xi^{-j}$, which belongs to \cite[Table 1, Row 6]{H09} and is of standard type $B_2$. Hence by \cite{An13,An15}, $\dim\BN(X\bigoplus X_{i,j})=18$. It follows that $\dim\BN(V_{i,j})\geq\frac{1}{2}\dim \BN(X\bigoplus X_{i,j})=18$. Consequently, $\BN\cong\BN(V_{i,j})$.
\end{proof}

\subsection{Liftings of Nichols algebras in ${}_{\cH_{p,-1}}^{\cH_{p,-1}}\mathcal{YD}$} We shall compute the liftings of some Nichols algebras in ${}_{\cH_{p,-1}}^{\cH_{p,-1}}\mathcal{YD}$ for $p\in\N$. In particular, we obtain families of finite-dimensional Hopf algebras without the dual Chevalley property, which constitute new examples of Hopf algebras.
\begin{defi}
Take $\mu\in\K$, $(i,j)\in\Lambda^3_p$. Let $\mathfrak{A}_{i,j}(\mu)$ be the algebra generated by $x,\,y,\,a,\,b$ satisfying the relations
\begin{gather*}
a^{2p}=1,\quad b^2=0,\quad  ba=\xi ab,\\
 ax=\xi^ixa,\quad   bx=\xi^ixb,\quad ay+ya=\Lam^{-1}xba^p,\quad  by+yb=xa^{p+1},\\
x^4=0,\quad xy+yx=\mu ba^{-1},\quad y^2+\theta^{-2}x^2=\frac{1}{2}\mu(1-a^p).
\end{gather*}
 \end{defi}

\begin{defi}
Set $\Lambda^4_p=\{(i,j)\in\Lambda_p\mid (p+j)(i-1)\equiv 0\mod 2p,\ 3ij\equiv 0\mod 2p,\ 3(i+1)\equiv 0\mod 2p\}$. Denote
by $\mathfrak{A}_{i,j}(\mu)$ $($with $(i,j)\in\Lambda_p^4$ and $\mu\in\K)$ the algebra generated by $v_1,\, v_2$, subject to the relations
\begin{gather*}
a^{2p}=1,\quad b^2=0,\quad  ba=\xi ab,\\
 ax=\xi^ixa,\quad   bx=\xi^ixb,\quad ay=\xi^{i+1}ya+\Lam^{-1}xba^p,\quad  by=\xi^{i+1}yb+xa^{p+1},\\
x^3=0,\quad x^2y+(-1)^{i+1}\xi^{2j}xyx+\xi^{-2j}yx^2=0, \\
 y^3-\alpha(-1)^i(\xi^{-j}-\xi^{-2j})x^2y-\alpha(\xi^{-j}+\xi^{-2j})yx^2-\alpha xyx=\mu(1-a^p),\\ xy^2+(-1)^{i+1}yxy+y^2x=-2\mu\xi^{1+i+j}((-1)^i-\xi^j)ba^{-1},
\end{gather*}
where $\alpha=\theta^{-2}\xi^{-(i+1)(2+j)}(1+\xi^{pi+j})$.
\end{defi}

$\mathfrak{A}_{i,j}(\mu)$ for $(i,j)\in\Lambda_p^3\cup\Lambda_p^4$ admits a Hopf algebra structure,  whose coalgebra structure is defined by
\begin{gather*}
\De(a)=a\otimes a+ \Lam^{-1}b\otimes ba^p,\quad \De(b)=b\otimes a^{p+1}+a\otimes b,\\
\Delta(x)=x\otimes 1+a^{-j}\otimes x+x_2\theta^{-1}ba^{-1-j}\otimes y,\\
\Delta(y)=y\otimes 1+a^{p-j}\otimes y+x_1\theta^{-1}ba^{p-j-1}\otimes x,
\end{gather*}
where $x_1=\theta^{-1}\xi^{p-1-i}((-)^i+\xi^j)$ and $x_2=\theta\xi^{p+1+i}((-1)^i-\xi^j)$.

\begin{lem}\label{lemDim-1}
For $(i,j)\in\Lambda^3_p$, a linear basis of $\mathfrak{A}_{i,j}(\mu)$  is given by $$\{y^rx^sa^tb^u\mid s\in\I_{0,3},t\in\I_{0,2p-1},r,u\in\I_{0,1}\}.$$ In particular, $\dim \mathfrak{A}_{i,j}(\mu)=32p$.
\end{lem}
\begin{proof}
We prove the statements for $\mathfrak{A}_{i,j}(\mu)$ by applying the Diamond Lemma \cite{B} with the order $y<x<a<b$. By the Diamond Lemma, it suffices to show that all overlaps ambiguities are resolvable, that is, the ambiguities can be reduced to the same expression by different substitution rules. To verify all the ambiguities are resolvable is tedious but straightforward.

It is clear that the overlapping pairs $\{a(x^4), (ax)x^3\}$, $\{b(x^4), (bx)x^3\}$, $\{x^4x^r, x^rx^4\}$ and $\{y^2y^s$, $ y^sy^2\}$ for $r\in\I_{1,4},s\in\I_{1,2}$ are resolvable.
It remains to verify the overlapping pairs $\{(xy)y, x(y^2)\}$, $\{x^3(xy),(x^4)y\}$, $\{(ay)y,a(y^2)\}$ and $\{(by)y,b(y^2)\}$ are resolvable.

Here we only show that $(xy)y=x(y^2), (x^4)y=x^{3}(xy)$ are resolvable and others are completely similar.

After a direct computation, we have that  $ba^{-1}y=yba^{-1}-xa^p$. Then
\begin{align*}
(xy)y&=(-yx+\mu ba^{-1})y=-yxy-\mu ba^{-1}y=y^2x-\mu yba^{-1}+\mu ba^{-1}y=y^2x-\mu xa^p\\
&=-\theta^{-2}x^3+\frac{1}{2}\mu(1-a^p)x-\mu xa^p=-\theta^{-2}x^3+\frac{1}{2}\mu x(1-a^p)=x(y^2).
\end{align*}
Note that $acx=xac$, then $x^2y=-xyx+\mu xba^{-1}=-(-yx+\mu ba^{-1})x+\mu xba^{-1}=yx^2$.
Therefore $x^{4}(xy)=yx^4=0=y(x^4)$. It follows that the overlaps $(xy)y=x(y^2), (x^4)y=x^3(xy)$ are resolvable.
\end{proof}

\begin{lem}\label{lemDim-3}
For $(i,j)\in\Lambda^4_p$, a linear basis of $\mathfrak{A}_{i,j}(\mu)$ is given by
  $$\{y^i(xy)^jx^kb^la^m\mid i,k\in\I_{0,2},\ j,l\in\I_{0,1},\ m\in\I_{0,2p-1}\}.$$
In particular, $\dim \mathfrak{A}_{i,j}(\mu)=72p$.
\end{lem}
\begin{proof}
We prove the Lemma by the Diamond Lemma \cite{B} with the order $y<xy<x<b<a$. Let $z:=xy$ for simplicity. It suffices to show that all overlaps ambiguities are resolvable, that is, the ambiguities can be reduced to the same expression by different substitution rules.
We show that the overlaps  $\{(ay)y^2,a(y^3)\}$ and $\{(by)y^2,b(y^3)\}$ are resolvable:
\begin{align*}
(ay)y^2&=(\xi^{i+1}ya+\lambda^{-1}xba^p)y^2\\
&=\xi^{2i+2}y^2ay+\lambda^{-1}\xi^{i+1}(yx+(-1)^{i+1}xy)ba^{p+1}y+(-1)^{i+1}\lambda^{-1}x^2ay\\
&=y^3a+\lambda^{-1}\xi^{2(i+1)}(y^2x+(-1)^{i+1}yxy+xy^2)ba^p+\\&\quad(-1)^{i+1}\xi^{i+1}\lambda^{-1}(x^2y+(-1)^{i+1}xyx+yx^2)a\\
&=\alpha(-1)^i(\xi^{-j}+\xi^{-2j})x^2ya+\alpha(\xi^{-j}+\xi^{-2j})yx^2a+\alpha xyxa+\mu(1-a^p)a\\&\quad+(-1)^{i+1}\xi^{i+1}\lambda^{-1}(x^2y+(-1)^{i+1}xyx+yx^2)a\\
&=\alpha(-1)^i(\xi^{-j}-\xi^{-2j})\xi^{3i+1}x^2ya+\alpha(\xi^{-j}+\xi^{-2j})\xi^{3i+1}yx^2a+\alpha\xi^{3i+1}xyxa+\mu(1-a^p)a\\
&=\alpha(-1)^i(\xi^{-j}-\xi^{-2j})x^2ya+\alpha(\xi^{-j}+\xi^{-2j})yx^2a+\alpha xyxa+\mu(1-a^p)a=a(y^3).
\end{align*}
\vspace{-0.5cm}
\begin{align*}
(by)y^2&=(\xi^{i+1}yb+xa^{p+1})y^2
=y^3b+\xi^{2(i+1)}(y^2x+(-1)^{i+1}yxy+xy^2)a^{p+1}\\&\quad+(-1)^{i+1}\xi^{i+1}\lambda^{-1}(yx^2+(-1)^{i+1}xyx+x^2y)b\\
&=y^3b-2\mu\xi^j((-1)^i-\xi^j)ba^p+(-1)^{i+1}\xi^{i+1}\lambda^{-1}(yx^2+(-1)^{i+1}xyx+x^2y)b\\
&=y^3b-2\mu ba^p+(-1)^{i+1}\xi^{i+1}\lambda^{-1}(y^2x+(-1)^{i+1}xyx+x^2y)b\\
&=\alpha(-1)^i(\xi^{-j}-\xi^{-2j})x^2yb+\alpha(\xi^{-j}+\xi^{-2j})yx^2b+\alpha xyxb+\mu(1-a^p)b-2\mu ba^p\\&\quad
(-1)^{i+1}\xi^{i+1}\lambda^{-1}(yx^2+(-1)^{i+1}xyx+x^2y)b\\
&=\alpha(-1)^i(\xi^{-j}-\xi^{-2j})\xi^{3i+1}x^2yb+\alpha(\xi^{-j}+\xi^{-2j})\xi^{3i+1}yx^2b+\alpha\xi^{3i+1}xyxb+\mu b+\mu a^pb\\
&=\alpha(-1)^i(\xi^{-j}-\xi^{-2j})bx^2y+\alpha(\xi^{-j}+\xi^{-2j})byx^2+\alpha bxyx+\mu b(1-a^p)=b(y^3).
\end{align*}
One can also show that the remaining overlaps $\{x^3x^r,x^rx^3\},\quad\{y^3y^r,y^ry^3\},\quad \{z^2z^h,z^hz^2\},$
\begin{gather*}
\{(ax)x^2, a(x^3)\},\quad \{(bx)x^2, b(x^3)\}, \quad
\{(az)z,a(z^2)\}, \quad \{(bz)z,b(z^2)\},\quad
\{(xz)z, x(z^2)\},\\ \{(x^3)z, x^2(xz)\},\quad \{(x^3)y,x^2(xy)\},\quad
\{(x(y^3),(xy)y^2\}, \quad \{z(y^3), (zy)y^2\},\quad
\{(z^2)y, z(zy)\}
\end{gather*}
are resolvable. We omit the details, since it is tedious but straightforward.
\end{proof}

\begin{pro}\label{pro:Lifting-H4p-2-dim-A2-1}
For $(i,j)\in\Lambda^3_p\cup\Lambda_p^4$, $\gr\mathfrak{A}_{i,j}(\mu)\cong \BN(V_{i,j})\sharp\cH_{p,-1}$.
\end{pro}
\begin{proof}
Here we only prove the Proposition for $(i,j)\in\Lambda^3_p$, being completely analogous for $(i,j)\in\Lambda^4_p$.  Denote by $\mathfrak{A}_{0}$ Hopf subalgebra of $\mathfrak{A}_{i,j}(\mu)$ generated by $a,b$. It is clear that $\mathfrak{A}_{0}\cong\cH_{p,-1}$. Let $\mathfrak{A}_{n}=\mathfrak{A}_{n-1}+\cH_{p,-1}\{y^ix^j\mid i+j=n,i\in\I_{0,3},j\in\I_{0,1}\}$ for $n\in\I_{0,4}$. Then $\{\mathfrak{A}_{n}\}_{n\in\I_{0,4}}$ is a coalgebra filtration of $\mathfrak{A}_{i,j}(\mu)$. Hence the coradical $(\mathfrak{A}_{i,j}(\mu))_{0}\subseteq \mathfrak{A}_{0}\cong\cH_{p,-1}$ and consequently $(\mathfrak{A}_{i,j}(\mu))_{[0]}\cong\cH_{p,-1}$. In particular, $\gr\mathfrak{A}_{i,j}(\mu)\cong R_{i,j}\sharp\cH_{p,-1}$, where $R_{i,j}$ is a connected Hopf algebra in ${}_{\cH_{p,-1}}^{\cH_{p,-1}}\mathcal{YD}$ and $\Pp(R_{i,j})=V_{i,j}$. Thus $R_{i,j}=\BN(V_{i,j})$ as Hopf algebras in ${}_{\cH_{p,-1}}^{\cH_{p,-1}}\mathcal{YD}$ since $\dim R_{i,j}=8=\dim\BN(V_{i,j})$ by Lemma \ref{lemDim-1}. Consequently, $\gr\mathfrak{A}_{i,j}(\mu)\cong \BN(V_{i,j})\sharp\cH_{p,-1}$.
\end{proof}

\begin{pro}\label{pro:lambda_3-liftings-1}
 Assume that $A$ is a finite-dimensional Hopf algebra over $\cH_{p,-1}$ such that $\gr A\cong \BN(V_{i,j})\sharp\cH_{p,-1}$, where $(i,j)\in\Lambda^3_p$. Then $A\cong \mathfrak{A}_{i,j}(\mu)$, for some $\mu\in\K$.
\end{pro}
\begin{proof}
Recall that if $(i,j)\in\Lambda^3_p$, then $(-1)^i=-1$, $\xi^{-ij}=-\xi^j$ and $\ord(\xi^{-ij})=4$. Hence the defining relations of $\BN(V_{i,j})$ can be rewritten as $x^4=0,xy+yx=0,y^2+\theta^{-2}x^2=0$. Observe that $1+\xi^{2j}=0$ for $(i,j)\in\Lambda^3$. Then $p-2j\equiv 0\mod 2p$ and hence $a^{-2j}=a^p$.

As $\Delta(x)=x\otimes 1+a^{2p-j}\otimes x+x_2\theta^{-1}ba^{2p-1-j}\otimes y$ and $\Delta(y)=y\otimes 1+a^{p-j}\otimes y+x_1\theta^{-1}ba^{p-j-1}\otimes x$, it follows from a direct computation that
\begin{gather*}
\Delta(y^2+\theta^{-2}x^2)=(y^2+\theta^{-2}x^2)\otimes 1+a^{p}\otimes (y^2+\theta^{-2}x^2),\\
\Delta(xy+yx)=(xy+yx)\otimes 1+1\otimes(xy+yx)-2ba^{-1}\otimes (y^2+\theta^{-2}x^2).
\end{gather*}
It follows that $y^2+\theta^{-2}x^2\in\Pp_{1,a^p}(A)=\Pp_{1,a^p}(\cH_{p,-1})=\K\{1-a^p,ba^{p-1}\}$ and  $y^2+\theta^{-2}x^2=\mu_1(1-a^p)+\mu_2ba^{p-1}$ for $\mu_1,\mu_2\in\K$. Denote $v:= xy+yx-\mu_12ba^{-1}$, then
\begin{align}
\Delta(v)=v\otimes 1+1\otimes v-2ba^{-1}\otimes \mu_2ba^{p-1}.\label{eq:quad-relations-1}
\end{align}

If $xy+yx=0$ admits non-trivial deformations, then there exist $\alpha_{i,j}^k\in\K$ for $i\in\I_{0,2p-1},j\in\I_{0,1},k\in\I_{1,3}$ such that
\begin{align*}
v=\sum_{i=0}^{2p-1}\sum_{j=0}^1\alpha_{i,j}^1b^ja^i+\alpha_{i,j}^2xb^ja^i+\alpha_{i,j}^3yb^ja^i.
\end{align*}
 Since $a(xy+yx)=\xi^{2i+1}(xy+yx)a$ and $b(xy+yx)=\xi^{2i+1}(xy+yx)b$, it follows from a direct computation that $\alpha_{i,0}^1=0=\alpha_{i,j}^k$ for $i\in\I_{0,3},j\in\I_{0,1},k\in\I_{2,3}$ and hence $v=\sum_{i=0}^{2p-1}\alpha_{i,1}^1ba^i$, which implies that
 \begin{align}
 \Delta(\sum_{i=0}^{2p-1}\alpha_{i,1}^1ba^i)=(\sum_{i=0}^{2p-1}\alpha_{i,1}^1ba^i)\otimes 1+1\otimes (\sum_{i=0}^{2p-1}\alpha_{i,1}^1ba^i)-2\mu_2 ba^{-1}\otimes ba^{p-1}.
 \end{align}
 It follows that equation \eqref{eq:quad-relations-1} holds only if $\mu_2=0$. Therefore, $y^2+\theta^{-2}x^2=0=v$ in $A$. Moreover, $\Delta(x^4)=\Delta(x)^4=x^4\otimes 1+1\otimes x^4$. Thus, there is an epimorphism from $\mathfrak{A}_{i,j}(2\mu_1)$ to $A$. Since $\dim A=\dim \mathfrak{A}_{i,j}(2\mu_1)$ by Lemma \ref{lemDim-1},  $A\cong\mathfrak{A}_{i,j}(2\mu_1)$ .
\end{proof}

\begin{pro}\label{pro:liftings-b2-n-H4p}
 Assume that $A$ is a finite-dimensional Hopf algebra over $\cH_{p,-1}$ such that $\gr A\cong\BN(V_{i,j})\sharp\cH_{p,-1}$, where $(i,j)\in\Lambda_p$ satisfying $(p+j)(i-1)\equiv 0\mod 2p$ and $3ij\equiv 0\mod 2p$. If $3(i+1)\equiv 0\mod 2p$, then $A\cong\mathfrak{A}_{i,j}(\mu)$ for $(i,j)\in\Lambda_p^4$; otherwise, $A\cong\gr A$.
\end{pro}
\begin{proof}
Let $X=v_1^2v_2+(-1)^{i+1}\xi^{2j}v_1v_2v_1+\xi^{-2j}v_2v_1^2$, $Y=v_2^3-\alpha(-1)^i(\xi^{-j}-\xi^{-2j})v_1^2v_2-\alpha(\xi^{-j}+\xi^{-2j})v_2v_1^2-\alpha v_1v_2v_1$ and $Z=v_1v_2^2+(-1)^{i+1}v_2v_1v_2+v_2^2v_1$ for short, where $\alpha=\theta^{-2}\xi^{-(i+1)(2+j)}(1+\xi^{pi+j})$. Then $\BN(V_{i,j})$
 is generated by $v_1,v_2$, subject to the relations
\begin{gather*}
v_1^3=0,\quad X=0,\quad Y=0,\quad Z=0.
\end{gather*}
 As $\Delta(v_1)=v_1\otimes 1+a^{-j}\otimes v_1+x_2\theta^{-1}ba^{-1-j}\otimes v_2$ and $\Delta(v_2)=v_2\otimes 1+a^{p-j}\otimes v_2+x_1\theta^{-1}ba^{p-1-j}\otimes v_1$, it follows from a direct computation that
\begin{align}
\Delta(v_1^3)=v_1^3\otimes 1+a^{-3j}\otimes v_1^3+\xi^{2j}x_2\theta^{-1}ba^{-1-3j}\otimes X,\label{eq:Standard-B2-1}\\
\Delta(X)=X\otimes 1+a^{p-3j}\otimes X,\quad \Delta(Y)=Y\otimes 1+a^{p-3j}\otimes Y,\label{eq:Standard-B2-2}\\
\Delta(Z)=Z\otimes 1+a^{-3j}\otimes Z+2(-1)^i\xi^j x_2\theta^{-1}ba^{-1-3j}\otimes Y+\beta ba^{-1-3j}\otimes X,\label{eq:Standard-B2-3}
\end{align}
where $\beta=2\xi^jx_2\alpha(\xi^{-j}-\xi^{-2j})+[(-1)^{i+1}+\xi^j]x_1\theta^{-1}$.  Observe that $-3j\equiv (i+1)p\mod 2p$. Then $a^{-3j}=a^{(i+1)p}$.

If $i\equiv 0\mod 2$, that is, $a^{p-3j}=1$, then by \eqref{eq:Standard-B2-2}, $X,Y$ are primitive elements of $A$ and hence $X=0=Y$ in $A$. By \eqref{eq:Standard-B2-1} and \eqref{eq:Standard-B2-3}, $v_1^3, Z\in\Pp_{1,a^p}(A)$, that is,
\begin{align*}
v_1^3=\alpha_1(1-a^p)+\alpha_2ba^{p-1},\quad Z=\beta_1(1-a^p)+\beta_2ba^{p-1},\quad \text{for some  }\alpha_1,\alpha_2,\beta_1,\beta_2\in\K.
\end{align*}
Since $av_1^3=\xi^{3i}v_1^3a$, $bv_1^3=\xi^{3i}v_1^3b$, $aZ=\xi^{3i+2}Za$ and $bZ=\xi^{3i+2}Zb$, it follows that
\begin{align*}
\alpha_1=0=\alpha_2,\quad \beta_1=0=\beta_2,
\end{align*}
which implies that $v_1^3=0=Z$ in $A$. Therefore, $A\cong\gr A$.

If $i\equiv 1\mod 2$, that is, $a^{p-3j}=a^p$, then by \eqref{eq:Standard-B2-2}, $X,Y\in\Pp_{1,a^p}(A)$, that is,
\begin{align}
X&=\mu_1(1-a^p)+\mu_2ba^{p-1},\quad \text{for some  }\mu_1,\mu_2\in\K,\label{eq:Standard-B2-4}\\
Y&=\nu_1(1-a^p)+\nu_2ba^{p-1},\quad \text{for some  }\nu_1,\nu_2\in\K.\label{eq:Standard-B2-5}
\end{align}
By \eqref{eq:Standard-B2-1} and \eqref{eq:Standard-B2-4}, we have
\begin{align*}
\Delta(v_1^3+\xi^{2j}x_2\theta^{-1}\mu_1 ba^{-1-3j})&=(v_1^3+\xi^{2j}x_2\theta^{-1}\mu_1 ba^{-1-3j})\otimes 1+1\otimes (v_1^3+\xi^{2j}x_2\theta^{-1}\mu_1 ba^{-1})\\&\quad+\xi^{2j}x_2\theta^{-1}ba^{-1}\otimes \mu_2ba^{p-1}.
\end{align*}
Since $av_1^3=\xi^{3i}v_1^3a$, $bv_1^3=\xi^{3i}v_1^3b$ and $3i\not\equiv -1\mod 2p$, a tedious computation on $A_{[2]}$ shows that $v_1^3=0$ in $A$ and hence the last equation holds only if $\mu_1=0=\mu_2$, which implies that $X=0$ in $A$. Then by \eqref{eq:Standard-B2-3} and \eqref{eq:Standard-B2-5}, we have
\begin{align*}
\Delta(Z+2\nu_1(-1)^i\xi^jx_2\theta^{-1}ba^{-1})&=(Z+2\nu_1(-1)^i\xi^jx_2\theta^{-1}ba^{-1})\otimes 1\\
&\quad +1\otimes (Z+2\nu_1(-1)^i\xi^jx_2\theta^{-1}ba^{-1})\\
&\quad +2(-1)^i\xi^jx_2\theta^{-1}ba^{-1}\otimes \nu_2ba^{p-1}.
\end{align*}
Then a tedious calculation on $A_{[2]}$ shows that the last equation holds only if $\nu_2=0$ and hence we have
\begin{align*}
Y=\nu_1(1-a^p),\quad Z=2\nu_1\xi^jx_2\theta^{-1}ba^{-1}.
\end{align*}
If $3i+3\not\equiv 0\mod 2p$, then by relation $aY=\xi^{3i+3}Ya$, we have $\nu_1=0$ and hence $Y=0=Z$ in $A$, which implies that $A\cong\gr A$. If $3i+3\equiv 0\mod 2p$, then the relations of $\mathfrak{A}_{i,j}(\mu)$  must hold in $A$ and hence there is an epimorphism from $\mathfrak{A}_{i,j}(\mu)$ for $(i,j)\in\Lambda_p^4$  to $A$. Then  $\dim \mathfrak{A}_{i,j}(\mu)=\dim A$ for $(i,j)\in\Lambda_p^4$ and hence $A\cong \mathfrak{A}_{i,j}(\mu)$.
\end{proof}

\begin{defi}
Assume that $p\in\N$ such that $p$ is even and $\frac{p}{2}$ is odd. For $\mu,\nu\in\K$ and $(i,j,k,\ell)=(p-1,\frac{p}{2},p-1,\frac{3p}{2})$, let $\mathfrak{A}_{i,j,k,\ell}(\mu,\nu)$ denote the algebra generated by the elements $a,b,x,y,z,t$ satisfying the relations
\begin{gather*}
a^{2p}=1, \quad b^2=0,\quad ba=\xi ab,\\
ax=\xi^ixa,\quad   bx=\xi^ixb,\quad ay+ya=\Lam^{-1}xba^p,\quad  by+yb=xa^{p+1},\\
x^4=0,\quad xy+yx=\mu ba^{-1},\quad y^2+\theta^{-2}x^2=\frac{1}{2}\mu(1-a^p),\\
az=\xi^kza,\quad   bz=\xi^kzb,\quad at+ta=\Lam^{-1}zba^p,\quad  bt+tb=za^{p+1},\\
z^4=0,\quad zt+tz=\nu ba^{-1},\quad t^2+\theta^{-2}z^2=\frac{1}{2}\nu(1-a^p),\\
zx-\xi^{\frac{p}{2}} xz=0,\quad tx+zy+yz+xt=0,\quad ty+yt+\theta^{-2}(1+\xi^{\frac{p}{2}})xz=0.
\end{gather*}
\end{defi}

$\mathfrak{A}_{i,j,k,\ell}(\mu,\nu)$ admits a Hopf algebra structure  whose coalgebra structure is defined by
\begin{gather*}
\De(a)=a\otimes a+ \Lam^{-1}b\otimes ba^p,\quad \De(b)=b\otimes a^{p+1}+a\otimes b,\\
\Delta(x)=x\otimes 1+a^{-j}\otimes x+x_2\theta^{-1}ba^{-1-j}\otimes y,\quad
\Delta(y)=y\otimes 1+a^{p-j}\otimes y+x_1\theta^{-1}ba^{p-j-1}\otimes x,\\
\Delta(z)=z\otimes 1+a^{-\ell}\otimes z+z_2\theta^{-1}ba^{-1-\ell}\otimes t,\quad
\Delta(t)=t\otimes 1+a^{p-\ell}\otimes t+z_1\theta^{-1}ba^{p-\ell-1}\otimes z,
\end{gather*}
where $x_1=\theta^{-1}\xi^{p-1-i}((-)^i+\xi^j)$, $x_2=\theta\xi^{p+1+i}((-1)^i-\xi^j)$, $z_1=\theta^{-1}\xi^{p-1-k}((-)^k+\xi^{\ell})$ and $z_2=\theta\xi^{p+1+k}((-1)^k-\xi^{\ell})$.

\begin{lem}\label{lemPBW-2}
Assume that $p\in\N$ such that $p$ is even and $\frac{p}{2}$ is odd. For any $\mu,\nu\in\K$ and $(i,j,k,\ell)=(p-1,\frac{p}{2},p-1,\frac{3p}{2})$, a linear basis of $\mathfrak{A}_{i,j,k,\ell}(\mu,\nu)$ is given by
\begin{align*}
\{t^{n_1}z^{n_2}(yt)^{n_{3}}y^{n_4}x^{n_5}b^{n_6}a^{n_7}\mid  n_2,n_4\in\I_{0,3},n_1,n_{3},n_5,n_6\in\I_{0,1},n_7\in\I_{0,2p-1}\}.
\end{align*}
In particular, $\dim \mathfrak{A}_{i,j,k,\ell}(\mu,\nu)=512p$.
\end{lem}
\begin{proof}
Similar to the proof of Lemma \ref{lemDim-3}.
\end{proof}

\begin{pro}
Assume that $p\in\N$ such that $p$ is even and $\frac{p}{2}$ is odd. For any $\mu,\nu\in\K$ and $(i,j,k,\ell)=(p-1,\frac{p}{2},p-1,\frac{3p}{2})$, we have $\gr\mathfrak{A}_{i,j,k,\ell}(\mu,\nu)\cong\BN(V_{i,j}\bigoplus V_{k,\ell})\sharp\cH_{p,-1}$.
\end{pro}
\begin{proof}
Similar to the proof of Proposition \ref{pro:Lifting-H4p-2-dim-A2-1}.
\end{proof}
\begin{pro}\label{pro:Lifting-H4p-2-dim-A2-2}
Suppose $A$ is a finite-dimensional Hopf algebra over $\cH_{p,-1}$ such that $\gr A\cong\BN(V_{i,j}\bigoplus V_{k,\ell})\sharp\cH_{p,-1}$, where $(i,j,k,\ell)=(p-1,\frac{p}{2},p-1,\frac{3p}{2})$. Then $A\cong\mathfrak{A}_{i,j,k,\ell}(\mu,\nu)$ for some $\mu,\nu\in\K$.
\end{pro}
\begin{proof}
Let $X:=tx+zy+yz+xt$ and $Y:=ty+yt+\theta^{-2}(1+\xi^{\frac{p}{2}})xz$ for short. It follows from a direct computation that
\begin{align}
\Delta(zx-\xi^{\frac{p}{2}} xz)=(zx-\xi^{\frac{p}{2}} xz)\otimes 1+1\otimes (zx-\xi^{\frac{p}{2}} xz)+(\xi^{\frac{p}{2}}-1)ba^{-1}\otimes X,\label{eqCV1113-2}\\
\Delta(X)=X\otimes 1+a^p\otimes X,\quad \Delta(Y)=Y\otimes 1+1\otimes Y-\theta^{-2}(1+\xi^{\frac{p}{2}})ba^{-1}\otimes X.\label{eqCV1113-1}
\end{align}
Then  $X=\alpha_1(1-a^p)+\alpha_2ba^{p-1}$ for $\alpha_1,\alpha_2\in\K$ by \eqref{eqCV1113-1}. Similar to the proof of the case $V\cong V_{i,j}$ for $(i,j)\in\Lambda_3^p$, a tedious computation on $A_{[1]}$ shows that equation \eqref{eqCV1113-2} holds only if $\alpha_2=0$. Then by \eqref{eqCV1113-2}, $zx-\xi^{\frac{p}{2}} xz+\alpha_1(\xi^{\frac{p}{2}}-1)ba^{-1}=0$ in $A$. Since $a(zx-\xi^{\frac{p}{2}} xz)=-(zx-\xi^{\frac{p}{2}} xz)a$ and $ba=\xi ab$, we have that $\alpha_1=0$ and hence the relations $X=0$, $Y=0$ and $zx-\xi^{\frac{p}{2}} xz=0$ hold in $A$. Moreover, as shown in the case $V\cong V_{i,j}$, we have
\begin{align*}
x^4=0,\quad xy+yx=\mu ba^{-1},\quad y^2+\theta^{-2}x^2=\frac{1}{2}\mu(1-a^p),\\
z^4=0,\quad zt+tz=\nu ba^{-1},\quad t^2+\theta^{-2}z^2=\frac{1}{2}\nu(1-a^p).
\end{align*}
Thus, there is an epimorphism from $\mathfrak{A}_{i,j,k,\ell}(\mu,\nu)$ to $A$. By Lemma \ref{lemPBW-2}, $\dim A=\dim \mathfrak{A}_{i,j,k,\ell}(\mu,\nu)$, it follows that $A\cong \mathfrak{A}_{i,j,k,\ell}(\mu,\nu)$.
\end{proof}

\begin{pro}\label{pro:Lifting-sum-one-dim-objects}
 Suppose $A$ is a finite-dimensional Hopf algebra over $\cH_{p,-1}$ such that $\gr A\cong\BN(V)\sharp\mathcal{K}$, where $V$ is isomorphic to $\bigoplus_{i=1}^n\K_{\chi^{n_i}}$ with odd numbers $n_i\in\I_{0,2p-1}$. Then $A\cong\gr A$.
\end{pro}
\begin{proof}
  Observe that $\BN(V)\sharp\cH_{p,-1}\cong\bigwedge V\sharp\cH_{p,-1}$. Let $V\cong\bigoplus_{i=1}^n\K_{\chi^{n_i}}=\bigoplus_{i=1}^n\K x_i$. Then $\Delta_A(x_i)=x_i\otimes 1+a^p\otimes x_i$. It follows from a direct computation that $x_i^2,x_ix_j+x_jx_i\in\Pp(A)$. Therefore, the relations in $\gr A$ must hold in $A$ and consequently $A\cong\gr A$.
\end{proof}
\begin{pro}\label{pro:Non-trivial-simple-object}
Suppose $A$ is a finite-dimensional Hopf algebra over $\cH_{p,-1}$ such that $\gr A\cong \BN(V_{i,j})\sharp\cH_{p,-1}$, where $(i,j)\in\Lambda_p^1$ or $\Lambda_p^2-\Lambda_p^3$. Then $A\cong \BN(V_{i,j})\sharp \cH_{p,-1}$.
\end{pro}
\begin{proof}
Assume that $(i,j)\in\Lambda^1_p$, that is, $1+\xi^{-ij}=0$. Then $\BN(V_{i,j})\sharp\cH_{p,-1}$ is generated as an algebra by $x,y,a,b$ satisfying the relations
\begin{gather}
a^{2p}=1,\quad b^2=0,\quad  ba=\xi ab, \\
ax=\xi^ixa,\quad   bx=\xi^ixb,\quad
ay-\xi^{i+1}ya=\Lam^{-1}xba^p,\quad  by-\xi^{i+1}yb=xa^{p+1},\\
x^2=0,\quad  xy+\xi^{-j}yx=0,\quad y^N=0,
\end{gather}
where $N=\ord ((-1)^i\xi^{-j})$.

If relation $x^2=0$ in $\BN(V_{i,j})$ has non-trivial deformations, then $x^2\in A_{[1]}$ and they must be linear combinations
of $\{a^i,ba^i, xa^i, ya^i,xba^i,yba^i\}_{i=0}^{2p-1}$. That is, there exist some elements $\alpha_i,\beta_i,\gamma_i,\phi_i,\mu_i,\nu_i\in\K$
for $i\in \I_{0,2p-1}$ such that
\begin{align*}
x^2=\sum_{n=0}^{2p-1}\alpha_n a^n+\beta_n ba^n+ \gamma_n xa^n+\phi_n xba^n+\mu_n ya^n+\nu_n yba^n.
\end{align*}
It follows from a direct computation that
\begin{align*}
x^2b&=\sum_{n=0}^{2p-1}\alpha_n\xi^{-n}ba^n+\gamma_n\xi^{-n}xba^n+\mu_n\xi^{-n}yba^n,\\
bx^2&=\sum_{n=0}^{2p-1}\alpha_nba^{n}+\gamma_n\xi^ixba^n+\mu_n\xi^{i+1}yba^n+\mu_nxa^{p+1+n}+\nu_n\xi^{p-1}xba^{p+1+n},\\
x^2a&=\sum_{n=0}^{2p-1}\alpha_n a^{n+1}+\beta_n ba^{n+1}+ \gamma_n xa^{n+1}+\phi_n xba^{n+1}+\mu_n ya^{n+1}+\nu_n yba^{n+1},\\
ax^2&=\sum_{n=0}^{2p-1}\alpha_n a^{n+1}+\beta_n\xi^{-1} ba^{n+1}+\gamma_n\xi^i xa^{n+1}+\xi^{i-1}\phi_nxba^{n+1}\\&\quad+\mu_n\xi^{i+1} ya^{n+1}+\xi^{-i}\mu_n xba^{p+n}+\xi^i\nu_n yba^{n+1}.
\end{align*}
Since $bx^2-\xi^{2i}x^2b=0$, we have that
\begin{align*}
(1-\xi^{2i-n})\alpha_n=0,\quad \mu_n=0,\quad \gamma_n(1-\xi^{2i-n})+\nu_{n-1-p}\xi^{p-1}=0.
\end{align*}
From $ax^2-\xi^{2i}x^2a=0$, we have that
\begin{align*}
(1-\xi^{2i})\alpha_n=0,\quad \beta_n(\xi^{-1}-\xi^{2i})=0,\quad \gamma_n=0,\quad (\xi^{-1}-\xi^i)\xi^i\phi_n=0,\quad\nu_n=0.
\end{align*}
Hence $\beta_n=\gamma_n=\mu_n=\nu_n=0$ for $n\in\I_{0,2p-1}$, $\phi_n=0$ for $n\in\I_{0,2p-2}$ and $\alpha_n=0$ for $n\in\I_{1,2p-1}$ and $x^2=\alpha_0+\phi_{2p-1}xba^{2p-1}$. Since $\epsilon(x^2)=0$, it follows that $\alpha_0=0$ and hence $x^2=\phi_{2p-1}xba^{2p-1}$. If $\phi_{2p-1}\neq 0$, then $\xi^i=\xi^{-1}$, that is, $i=2p-1$. Since $\xi^{ij}=-1$, it follows that $j=p$. It is impossible because $(i,j)\in\Lambda$. Hence $x^2=0$ in $A$.

Now we claim that $xy+\xi^{-j}yx=0$  in $A$. Indeed,
since $xy+\xi^{-j}yx\in A_{[1]}$, there exist some elements $\alpha_n,\beta_n,\gamma_n,\phi_n,\mu_n,\nu_n\in\K$ for $n\in \I_{0,2p-1}$ such that
\begin{align*}
xy+\xi^{-j}yx=\sum_{n=0}^{2p-1}\alpha_n a^n+\beta_n ba^n+ \gamma_n xa^n+\phi_n xba^n+\mu_n ya^n+\nu_n yba^n.
\end{align*}
it follows from a direct computation that
\begin{align*}
a(xy+\xi^{-j}yx)&=\xi^{2i+1}(xy+\xi^{-j}yx)a+\xi^i\Lam^{-1}(1+(-1)^i\xi^{-j})x^2ba^p,\\
b(xy+\xi^{-j}yx)&=\xi^{2i+1}(xy+\xi^{-j}yx)b+\xi^i(1+(-1)^i\xi^{-j})x^2a^{p+1}.
\end{align*}
Since $x^2=0$ in $A$, we have that
\begin{align*}
a(xy+\xi^{-j}yx)=\xi^{2i+1}(xy+\xi^{-j}yx)a,\quad
b(xy+\xi^{-j}yx)=\xi^{2i+1}(xy+\xi^{-j}yx)b,
\end{align*}
which implies that for all $n\in \I_{0,2p-1}$,
$\alpha_n=\beta_n=\gamma_n=\lambda_n=\mu_n=\nu_n=0$, which implies that the claim follows. Finally, we claim that $y^N=0$ in $A$. Observe that $\ord((-1)^i\xi^{-j})=N$. Then $\xi^{-jN}=(-1)^{iN}$, that is, $-jN\equiv piN \mod 2p$, which implies that $a^{(p-j)N}=a^{(i+1)Np}$. As $\Delta(y)=y\otimes 1+a^{p-j}\otimes y+x_1\theta^{-1}ba^{p-1-j}\otimes x$, we have
\begin{gather*}
\Delta(y^N)=y^N\otimes 1+a^{N(p-j)}\otimes y^N=y^N\otimes 1+a^{(i+1)Np}\otimes y^N.
\end{gather*}
If $(i+1)N\equiv 0\mod 2$, that is, $a^{(i+1)Np}=1$, then $y^N$ is a primitive element of $A$ and hence $y^N=0$ in $A$. If $(i+1)N\equiv 1\mod 2$, that is, $a^{(i+1)Np}=a^p$, then $y^N\in\Pp_{1,a^p}(A)=\Pp_{1,a^p}(\cH_{p,-1})=\K (1{-}a^p)\bigoplus\K ba^{p-1}$ and hence there exist $\alpha_1,\alpha_2\in\K$ such that
\begin{gather*}
y^N=\alpha_1(1-a^p)+\alpha_2ba^{p-1}.
\end{gather*}

Since $by^N=\xi^{(i+1)N}y^Nb$, it follows that $\alpha_1=0$ and hence $y^N=\alpha_2ba^{p-1}$. Since
\begin{align*}
(xy)y^{N-1}&=-\xi^{-j}y(xy)y^{N-2}=\cdots=(-1)^N\xi^{-Nj}y^Nx=\alpha_2(-1)^{(i+1)N}ba^{p-1}x,\\
x(y^N)&=\alpha_2xba^{p-1}=(-1)^i\alpha_2ba^{p-1}x,
\end{align*}
by the Diamond lemma, the overlapping pair $\{x, y^N\}$ is resolvable and hence we have
\begin{gather*}
-\alpha_2(1+(-1)^i)=\alpha_2[(-)^{(i+1)N}-(-1)^i]=0.
\end{gather*}
If $i\equiv 0\mod 2$, then $\alpha_2=0$, otherwise $(i+1)N\equiv 0\mod 2$, a contradiction. Hence $\alpha_2=0$ and the claim follows. Therefore, $A\cong\gr A$.

Assume that $(i,j)\in\Lambda^2_p-\Lambda^3_p$. Let $w:=y^2+(-1)^{i+1}(\theta\xi^{i+1})^{-2}x^2$ for short. A direct computation shows that
\begin{gather*}
a(xy+(-1)^{i+1}yx)=\xi^{2i+1}(xy+(-1)^{i+1}yx)a,\quad
b(xy+(-1)^{i+1}yx)=\xi^{2i+1}(xy+(-1)^{i+1}yx)b,\\
aw=\xi^{2(i+1)}wa+(1+\xi^{pi+j})\Lam^{-1}\xi^{(p+1)(i+1)}(xy+(-1)^{i+1}yx)ba^p,\\
bw=\xi^{2(i+1)}wb+(1+\xi^{pi+j})\xi^{(p+1)(i+1)}(xy+(-1)^{i+1}yx)a^{p+1}.
\end{gather*}

If $\xi^{2(i+1)}\neq 1$, that is, $\xi^{2i+1}\neq \xi^{-1}$, then a tedious computation shows that $xy+(-1)^{i+1}yx=0$ in $A$ and hence
\begin{align*}
aw=\xi^{2(i+1)}wa,\quad bw=\xi^{2(i+1)}wb,
\end{align*}
which implies that $w=0$ in $A$.

If $\xi^{2(i+1)}=1$ and $1+\xi^{2j}\neq 0$. Then $\xi^{2i+1}=\xi^{-1}$. A direct computation shows that
\begin{align*}
xy+(-1)^{i+1}yx=\sum_{n=0}^{2p-1}\beta_nba^n, \quad\text{ for some  } \beta_n\in\K.
\end{align*}
Hence, $aw=wa+(1+\xi^{pi+j})\Lam^{-1}\xi^{(p+1)(i+1)}(xy+(-1)^{i+1}yx)ba^p=wa$, which implies that
\begin{align*}
w=\sum_{n=0}^{2p-1}\tau_na^n,  \quad \text{for some } \tau_n\in\K.
\end{align*}
 Since $\Delta(x)=x\otimes 1+a^{2p-j}\otimes x+x_2\theta^{-1}ba^{2p-1-j}\otimes y$ and $\Delta(y)=y\otimes 1+a^{p-j}\otimes y+x_1\theta^{-1}ba^{p-j-1}\otimes x$, it follows from a direct computation that
 \begin{align*}
 \Delta(w)=w\otimes 1+a^{-2j}\otimes w+\theta^{-2}\xi^{p-1-i}(-1)^i(1+\xi^{2j})ba^{-1-2j}\otimes (xy+(-1)^{i+1}yx).
 \end{align*}
 Hence, we have
 \begin{align*}
 \Delta(\sum_{n=0}^{2p-1}\tau_na^n)=(\sum_{n=0}^{2p-1}\tau_na^i)\otimes 1+a^{2j}\otimes (\sum_{n=0}^{2p-1}\tau_na^n)+Aba^{-2j-1}\otimes (\sum_{n=0}^{2p-1}\beta_n ba^n),
 \end{align*}
 where $A=\theta^{-2}\xi^{p-1-i}(-1)^i(1+\xi^{2j})\neq 0$. It follows from a direct computation that $\beta_n=0=\tau_n$ for $n\in\I_{0,2p-1}$. Hence $w=0=xy+(-1)^{i+1}yx$ in $A$.  Finally, we claim that $x^N=0$ in $A$. Observe that $\xi^{-jN}=(-1)^{iN}$. Then $a^{-jN}=a^{iNp}$. As $\Delta(x)=x\otimes 1+a^{-j}\otimes x+x_2\theta^{-1}ba^{-1-j}\otimes v_2$, we have
 \begin{gather*}
 \Delta(x^N)=x^N\otimes 1+a^{-jN}\otimes x^N=x^N\otimes 1+a^{piN}\otimes x^N.
 \end{gather*}
 If $iN\equiv 0\mod 2$, that is, $a^{piN}=1$, then $x^N$ is a primitive element of $A$ and hence $x^N=0$ in $A$. If $iN\equiv 1\mod 2p$, that is, $a^{piN}=a^p$, then $x^N\in\Pp_{1,a^p}(A)=\Pp_{1,a^p}(\cH_{p,-1})=\K (1{-}a^p)\bigoplus \K ba^{p-1}$, which implies that
 \begin{gather*}
 x^N=\alpha_1(1-a^p)+\alpha_2ba^{p-1},\quad \alpha_1, \alpha_2\in\K.
 \end{gather*}
 Since $bx^N=\xi^{iN}x^Nb$ and $x^Ny=(-1)^{iN}yx^N=-yx^N$, it follows that $\alpha_1=0=\alpha_2$ and hence $x^N=0$ in $A$. Consequently, the claim follows. Therefore, $A\cong\gr A$.
\end{proof}

\subsection{Some classification results} We give some classification results. In particular, we obtain a complete classification of finite-dimensional Hopf algebras whose Hopf coradical is isomoprhic to $\cH_{p,-1}$ with a prime number $p>5$.
\begin{thm}\label{thm:lifting-braiding-object-quad-simple}
 Assume $A$ is a finite-dimensional Hopf algebra over $\cH_{p,-1}$, whose infinitesimal braiding $V$ is an indecomposable object in ${}_{\cH_{p,-1}}^{\cH_{p,-1}}\mathcal{YD}$. If the diagram of $A$ admits non-trivial quadratic relations, then $A$ is isomorphic to one of the following Hopf algebras
\begin{itemize}
  \item \ $\bigwedge\K_{\chi^k}\sharp\cH_{p,-1}$, for odd $k\in\I_{0,2p-1}$,
  \item \ $\BN(V_{i,j})\sharp\cH_{p,-1}$, for $(i,j)\in\Lambda_p^1\cup\Lambda^2_p-\Lambda^3_p$,
  \item \ $\mathfrak{A}_{i,j}(\mu)$, for some $\mu\in\K$ and $(i,j)\in\Lambda_p^3$.
\end{itemize}
\end{thm}
\begin{proof}
Recall that $\cH_{p,-1}$ is a basic Hopf algebra, whose dual Hopf algebra is a cocycle deformation of the associated graded Hopf algebra. Then by \cite[Theorem 1.3]{AA18b} and Proposition 4.2, the diagram of $A$ is a Nichols algebra over a simple object in ${}_{\cH_{p,-1}}^{\cH_{p,-1}}\mathcal{YD}$. Hence by Theorem \ref{thmsimplemoduleD}, $V\cong \K_{\chi^k}$ for odd $k\in\I_{0,2p-1}$ or $V_{i,j}$ for some $(i,j)\in\Lambda_p$.
By Remark \ref{rmk:relations-Nichols-algebra-2-dim-simple-quad}, $\BN(V_{i,j})$ admits non-trivial quadratic relations if and only if $(i,j)\in\Lambda_p^1\cup\Lambda_p^2$. Then the Theorem follows from Propositions \ref{pro:lambda_3-liftings-1}, \ref{pro:Lifting-sum-one-dim-objects} and \ref{pro:Non-trivial-simple-object}.
\end{proof}

\begin{cor}\cite[Theorem B]{GG16}\label{cor:lifting-braiding-object-quad-simple}
 Assume $A$ is a finite-dimensional Hopf algebra over $\cH_{2,-1}$, whose infinitesimal braiding is an indecomposable object in ${}_{\cH_{2,-1}}^{\cH_{2,-1}}\mathcal{YD}$. Then $A$ is isomorphic to one of the following Hopf algebras
\begin{itemize}
  \item $\bigwedge\K_{\chi^k}\sharp\cH_{2,-1}$ for $k\in\{1,3\}$,
  \item $\BN(V_{2,j})\sharp\cH_{2,-1}$ for $j\in\{1,3\}$,
  \item $\mathfrak{A}_{1,j}(\mu)$ for $j\in\{1,3\}$.
\end{itemize}
\end{cor}
\begin{proof}
It follows from Theorems \ref{thm-finite-braided-vector-spaces-H} and \ref{thm:lifting-braiding-object-quad-simple}.
\end{proof}

\begin{pro}\label{pro:Liftings-Lambda1-Lambda2--lambda3-two-simple-objects}
 Assume $A$ is a finite-dimensional Hopf algebra  over $\cH_{p,-1}$ such that $\gr A\cong\BN(V_{i,j}\bigoplus V_{k,\ell})\sharp\cH_{p,-1}$, where $(i,j),(k,\ell)\in\Lambda_p^1$ $($resp. $\Lambda_p^2-\Lambda^3_p$$)$ satisfy $kj+i\ell\equiv 0\mod 2p$ and $p(i+k)+j+\ell\equiv 0\mod 2p$. Then $A\cong\BN(V_{i,j}\bigoplus V_{k,\ell})\sharp\cH_{p,-1}$.
\end{pro}
\begin{proof}
By Proposition \ref{pro:Non-trivial-simple-object}, the relations of $\BN(V_{i,j})$ for $(i,j)\in\Lambda^1_p$ (resp. $\Lambda_p^2-\Lambda^3_p$) hold in $A$.

Let $X= w_2v_1-\xi^{(k+1)j}v_1w_2-(-1)^i\xi^{i-k}(w_1v_2-\xi^{-(i+1)\ell}v_2w_1)$, $Y=w_2v_2-\xi^{(i+1)(p-\ell)}v_2w_2-\theta^{-2}\xi^{(i+1)(p-1-\ell)+(p-1-k)}((-1)^k+\xi^{\ell})v_1w_1$ and $Z:= w_1v_1-\xi^{-i\ell}v_1w_1$ for short.
It follows from a direct computation that
\begin{gather}
\Delta(X)=X\otimes 1+a^{p-(j+\ell)}\otimes X=X\otimes 1+a^{(i+k+1)p}\otimes X,\label{eqCVijk-0}\\
\Delta(Y)=Y\otimes 1+a^{-(j+\ell)}\otimes Y+\theta^{-2}\xi^{\ell-1-i}((-)^i+\xi^j)ba^{-1-(j+\ell)}\otimes X,\label{eqCVijk-1}\\
\Delta(Z)=Z\otimes 1+a^{-(j+\ell)}\otimes Z-\xi^{1+k}((-1)^k-\xi^{\ell})ba^{-1-(\ell+j)}\otimes X.\label{eqCVijk-2}
\end{gather}
Since $p(i+k)+j+\ell\equiv 0\mod 2p$, it follows that $\xi^{-(j+\ell)}=(-1)^{i+k}$ and $a^{-(j+\ell)}=a^{(i+k)p}$.

We claim that $i+k\equiv 0\mod 2$. If $(i,j),(k,\ell)\in\Lambda_p^1$, then we have $\xi^{-ij}=-1$ and $\xi^{-k\ell}=-1$. moreover,
\begin{align*}
\xi^{-i\ell}=(-1)^{i^2+ik}\xi^{ij}=(-1)^{i+ik+1},\quad
\xi^{-kj}=(-)^{ik+k^2}\xi^{k\ell}=(-1)^{k+ik+1},
\end{align*}
which implies that $(-1)^{i+k}=\xi^{-i\ell-kj}=1$ and hence $i+k\equiv 0\mod 2$. If $(i,j),(k,\ell)\in\Lambda_p^2$, then by Remark \ref{rmkDmoddual}, $(-i-1,-j-p),(-k-1,-\ell-p)\in\Lambda^1_p$. Let $i_1=-i-1,j_1=-j-p$. It is easy to see that  $k_1j_1+i_1\ell_1\equiv 0\mod 2p$ and $p(i_1+k_1)+j_1+\ell_1\equiv 0\mod 2p$. Then from the preceding proof, $i_1+k_1\equiv 0\mod 2$ and hence $i+k\equiv 0\mod 2$.

 Since $i+k\equiv 0\mod 2$,  by \eqref{eqCVijk-0}, $X=\alpha_1(1-a^p)+\alpha_2ba^{p-1}$ for some  $\alpha_1,\alpha_2\in\K$. Let $r_1=Z-\alpha_1\xi^{1+k}((-1)^k-\xi^{\ell})ba^{-1}$. \eqref{eqCVijk-2} implies that
\begin{align}
\Delta(r_1)=r_1\otimes 1+1\otimes r_1-\alpha_2\xi^{1+k}((-1)^k-\xi^{\ell})ba^{-1}\otimes ba^{p-1}.\label{eqV2123-3}
\end{align}

If relation $Z=0$ admits non-trivial deformations, then $r_1\in A_{[1]}$, that is, there exist $\alpha_{i,j}^k\in\K$ for $i\in\I_{0,2p-1},j\in\I_{0,1},k\in\I_{1,5}$ such that
\begin{align*}
Z=\sum_{i=0}^{2p-1}\sum_{j=0}^1\alpha_{i,j}^1b^ja^i+\alpha_{i,j}^2v_1b^ja^i+\alpha_{i,j}^3v_2b^ja^i+\alpha_{i,j}^4w_1b^ja^i+\alpha_{i,j}^5w_2b^ja^i.
\end{align*}
Since $aZ=\xi^{i+k}Za$, $bZ=\xi^{i+k}Zb$ and $\epsilon(Z)=0$, it follows from a direct computation that $\alpha_{i,j}^k=0$ for $i\in\I_{0,2p-1},j\in\I_{0,1},k\in\I_{1,5}$, which implies that $Z=0$ in $A$. Then $r_1=-\alpha_1\xi^{1+k}((-1)^k-\xi^{\ell})ba^{-1}$ and hence \eqref{eqV2123-3} can be rewritten by
\begin{align*}
\Delta(-\alpha_1\xi^{1+k}((-1)^k-\xi^{\ell})ba^{-1})&=-\alpha_1\xi^{1+k}((-1)^k-\xi^{\ell})ba^{-1}\otimes 1+1\otimes -\alpha_1\xi^{1+k}((-1)^k-\xi^{\ell})ba^{-1}\\&\quad-\alpha_1\xi^{1+k}((-1)^k-\xi^{\ell})ba^{-1}\otimes\alpha_2ba^{p-1}.
\end{align*}
It follows from a direct computation that the last equation holds only if $\alpha_1=0=\alpha_2$, which implies that $X=0$ in $A$ and hence $Y=0$ in $A$. Therefore $A\cong\gr A$.
\end{proof}

\begin{pro}\label{pro:Liftings-Lambda1-two-simple-objects}
Suppose $A$ is a finite-dimensional Hopf algebra  over $\cH_{p,-1}$ such that $\gr A\cong\BN(V_{i,j}\bigoplus\K_{\chi^k})\sharp\cH_{p,-1}$, where $(i,j)\in\Lambda_p^1$  and $k\in\I_{0,2p-1}$ is odd such that $(k+1)(pi-j)\equiv 0\mod 2p$. Then $A\cong\BN(V_{i,j}\bigoplus\K_{\chi^k})\sharp\cH_{p,-1}$.
\end{pro}
\begin{proof}
 Similar to the proof of Proposition \ref{pro:Non-trivial-simple-object}, the relations $v_1^2=0$, $v_1v_2+\xi^{-j}v_2v_1=0$ and $v_2^N=0$ hold in $A$, where $N=\ord((-1)^i\xi^{-j})$. As $\Delta(v_3)=v_3\otimes 1+a^p\otimes v_3$, we have $\Delta(v_3^2)=v_3^2\otimes 1+1\otimes v_3^2$, which impiles that $v_3^2=0$ in $A$. Let $X=(-1)^iv_1v_2v_3+v_1v_3v_2+(-1)^iv_2v_3v_1+v_3v_2v_1$, $Y=v_3v_2^2+((-1)^i+\xi^{-j})v_2v_3v_2+(-1)^i\xi^{-j}v_2^2v_3+(-1)^{i+1}\theta^{-2}\xi^{-(i+1)(2+j)}(1+\xi^{pi+j})v_1v_3v_1$ and $Z=(v_3v_1)^N+(-1)^{iN}(v_1v_3)^N$ for short.

 It follows from a direct computation that
\begin{gather*}
aX=\xi^{2i+k+1}Xa,\quad bX=\xi^{2i+k+1}Xb,\\
aY=\xi^{2i+2+k}Ya+\xi^{i+k+1}\lambda^{-1}(1+\xi^{pi-j})Xba^p,\quad bY=\xi^{2i+2+k}Yb+\xi^{i+k+1}(1+\xi^{pi-j})Xa^{p+1},\\
\Delta(X)=X\otimes 1+a^{-2j}\otimes X-\xi^{i+j+1}((-1)^i-\xi^j)ba^{-1-2j}\otimes Y,\quad
\Delta(Y)=Y\otimes 1+a^{p-2j}\otimes Y.
\end{gather*}
Then similar to the proof of Proposition \ref{pro:Non-trivial-simple-object}, a tedious computation on $A_{[2]}$ shows that $X=0=Y$ in $A$. Observe that $(-1)^{iN}\xi^{-jN}=1$. Then $a^{-jN}=a^{iNp}$.  Furthermore, it follows that
\begin{gather*}
\Delta(Z)=Z\otimes 1+a^{N(p-j)}\otimes Z=Z\otimes 1+a^{(i+1)Np}\otimes Z.
\end{gather*}
If $(i+1)N\equiv 0\mod 2$, then $a^{(i+1)Np}=1$ and hence $Z$ is a primitive element of $A$, which implies that $Z=0$ in $A$. If $(i+1)N\equiv 1\mod 2$, then $a^{(i+1)Np}=a^p$ and hence $Z\in\Pp_{1,a^p}(A)=\Pp_{1,a^p}(\cH_{p,-1})=\K (1{-}a^p)\bigoplus\K ba^{p-1}\bigoplus \K v_3$, that is,
\begin{gather*}
Z=\alpha_1(1-a^p)+\alpha_2ba^{p-1}+\alpha_3v_3.
\end{gather*}
for some $\alpha_1,\alpha_2,\alpha_3\in\K$. Observe that $k+i\equiv 0\mod 2$. Then by the equations  $aZ=\xi^{N(k+i)}Za$ and $bZ=\xi^{N(k+i)}Zb$, it follows that $\alpha_1=0=\alpha_2=\alpha_3$ and hence $Z=0$ in $A$. Therefore, $\gr A\cong A$.
\end{proof}

\begin{pro}\label{pro:Lifting-Lambda3-k-1-1}
Suppose $A$ is a finite-dimensional Hopf algebra  over $\cH_{p,-1}$ such that $\gr A\cong\BN(V_{i,j}\bigoplus\K_{\chi^k})\sharp\cH_{p,-1}$, where $(i,j)\in\Lambda_p^3$ and $k\in\I_{0,2p-1}$ is odd such that $(k-1)(pi-j)\equiv 0\mod 2p$. If $k\neq 1$, then $A\cong\BN(V_{i,j}\bigoplus\K_{\chi^k})\sharp\cH_{p,-1}$.
\end{pro}
\begin{proof}
Set $r_1=v_1v_2+(-1)^{i+1}v_2v_1$, $r_2=v_2^2+(-1)^{i+1}(\theta\xi^{i+1})^{-2}v_1^2$, $X=
(-1)^{i+1}\xi^jv_1v_2v_3-\xi^jv_1v_3v_2+v_2v_3v_1+(-1)^iv_3v_2v_1$, $Y=
v_3v_1^2+[(-1)^{i+1}-\xi^j]v_1v_3v_1+(-1)^i\xi^jv_1^2v_3$ and $Z:=\alpha \sum_{\ell=0}^{2}\xi^{2j\ell}(v_1v_3)^2(v_2v_3)^{2}+(v_2v_3)^4+(v_3v_2)^4$ for short.
From the proof of Proposition \ref{pro:lambda_3-liftings-1}, we have $v_3^2\in\Pp(A)$ and
\begin{align*}
\Delta(r_1)=r_1\otimes 1+1\otimes r_1-2ba^{-1}\otimes r_2,\quad \Delta(r_2)=r_2\otimes 1+a^p\otimes r_2.
\end{align*}
Since $\Pp(A)=0$ and $\Pp_{1,a^p}(A)=\K (1{-}a^p)\bigoplus \K ba^{p-1}\bigoplus \K v_3$, it follows that $v_3^2=0$ in $A$ and
\begin{gather}
r_2=\alpha_1(1-a^p)+\alpha_2ba^{p-1}+\alpha_3v_3,\quad\text{which implies that}\notag\\
\Delta(r_1)=r_1\otimes 1+1\otimes r_1-2ba^{-1}\otimes [\alpha_1(1-a^p)+\alpha_2ba^{p-1}+\alpha_3v_3]\label{eq:H4p-Lifting-A_3-1}.
\end{gather}
Observe that $r_1\in A_{[1]}$. Then there are $\alpha_{m,n}^k$ with $m\in\I_{0,2p-1},j\in\I_{0,1},k\in\I_{1,4}$ such that
\begin{align*}
r_1=\sum_{i=0}^{2p-1}\sum_{j=0}^1\alpha_{i,j}^1b^ja^i+\alpha_{i,j}^2v_1b^ja^i+\alpha_{i,j}^3v_2b^ja^i+\alpha_{i,j}^4v_3b^ja^i.
\end{align*}
Since $ar_1=\xi^{2i+1}r_1a$ and $br_1=\xi^{2i+1}r_1b$, it follows from a direct computation that $\alpha_{i,0}^1=\alpha_{i,j}^2=\alpha_{i,j}^3=\alpha_{i,j}^4=0$ for $i\in\I_{0,2p-1},j\in\I_{0,1}$, that is,
\begin{align*}
r_1=\sum_{i=0}^{2p-1}\alpha_{i,1}^1ba^i.
\end{align*}
Then from the last equation and \eqref{eq:H4p-Lifting-A_3-1}, we have $\alpha_2=0=\alpha_3$ and hence
\begin{align}\label{eq:r1r2-1}
r_1=2\alpha_1ba^{-1},\quad r_2=\alpha_1(1-a^p).
\end{align}

It  follows from a direct computation that
\begin{align}
\begin{split}\label{eq:r1r2-2}
\Delta(X)&=X\otimes 1+a^{-2j}\otimes X+(1+\xi^{2j})\theta^{-2}\xi^{-1-2j}ba^{-1-2j}\otimes Y+\\
&\quad (-1)^i\xi^jx_2\theta^{-1}ba^{-1-2j}\otimes v_3r_2+(-1)^{i+1}\xi^{j}x_2\theta^{-1}ba^{-1-2j}\otimes r_2v_3\\&\quad
+(-1)^i(\xi^{-j}-\xi^j)x_2\theta^{-1}v_3ba^{p-1-2j}\otimes r_2+[(-1)^i\xi^{-j}-1]v_3a^{p-2j}\otimes r_1\\
&=X\otimes 1+a^p\otimes X+(-1)^i\xi^jx_2\theta^{-1}ba^{p-1}\otimes v_3r_2+(-1)^{i+1}\xi^{j}x_2\theta^{-1}ba^{p-1}\otimes r_2v_3\\&\quad
+(-1)^i(\xi^{-j}-\xi^j)x_2\theta^{-1}v_3ba^{-1}\otimes r_2+[(-1)^i\xi^{-j}-1]v_3\otimes r_1,
\end{split}\\
\begin{split}
\Delta(Y)&=Y\otimes 1+a^{p-2j}\otimes Y+\xi^{1+i}(1-\xi^{2j})ba^{p-1-2j}\otimes X\\
&=Y\otimes 1+1\otimes Y+2\xi^{1+i}ba^{-1}\otimes X.\label{eq:r1r2-3}
\end{split}
\end{align}
Let $r:=X+2\alpha_1(-1)^{i+1}\xi^jx_2\theta^{-1}v_3ba^{-1}$ for short. By \eqref{eq:r1r2-1} and \eqref{eq:r1r2-2},
\begin{align*}
\Delta(r)=r\otimes 1+a^p\otimes r,
\end{align*}
which implies that $r=\beta_1(1-a^p)+\beta_2ba^{p-1}+\beta_3v_3$ for some $\beta_1,\beta_2,\beta_3\in\K$, that is,
\begin{align}
X=2\alpha_1[(-1)^i\xi^{-j}-1]v_3ba^{-1}+\beta_1(1-a^p)+\beta_2ba^{p-1}+\beta_3v_3.
\end{align}
Then by \eqref{eq:r1r2-3} and the last equation, we have
\begin{align}
\begin{split}
\Delta(Y+2\beta_1\xi^{i+1}ba^{-1})&=(Y+2\beta_1\xi^{i+1}ba^{-1})\otimes 1+1\otimes (Y+2\beta_1\xi^{i+1}ba^{-1})+\\
&\quad 4\xi^{1+i}\alpha_1[(-1)^i\xi^{-j}-1]ba^{-1}\otimes v_3ba^{-1}+2\beta_2\xi^{1+i}ba^{-1}\otimes ba^{p-1}\\&\quad+2\beta_3\xi^{1+i}ba^{-1}\otimes v_3.
\end{split}
\end{align}
Since $aY=\xi^{2i+k}Ya=\xi^{k-2}Ya$ and $bY=\xi^{k-2}Yb$, a tedious computation on $A_{[2]}$ shows that the last equation hold only if $\alpha_1=0=\beta_2=\beta_3$ and hence
\begin{align*}
r_1=0=r_2,\quad X=\beta_1(1-a^p),\quad Y=2\beta_1ba^{-1}.
\end{align*}

Observe that $ab=\xi^{-1}ba$. If $k\neq 1$, then $\beta_1=0$ and hence $X=0=Y$ in $A$. Note that $(i,j)\in\Lambda_p^3$. Then $a^{-4j}=1$ and hence
\begin{align*}
\Delta(Z)=Z\otimes 1+1\otimes Z+o(Z),
\end{align*}
where $o(Z)\in r(Z)\otimes A+A\otimes r(Z)$ and $r(Z)$ belongs to the Hopf ideal of $A$ generated by $r_1,r_2,X,Y$. Since $r_1=r_2=X=Y=0$ in $A$, it follows that $r(Z)=0$ and hence $o(Z)=0$, which implies that $Z\in\Pp(A)$. Therefore, $Z=0$ in $A$ and then $A\cong\gr A$.
\end{proof}

\begin{rmk}\label{rmk:Liftings-over-cH2}
Keep the notations in Proposition \ref{pro:Lifting-Lambda3-k-1-1}. If $k=1$, then from the proof of Proposition \ref{pro:Lifting-Lambda3-k-1-1}, the defining relations of $A$ are of the following form:
\begin{gather*}
a^{2p}=1,\quad b^2=0,\quad  ba=\xi ab,\\
 av_1=\xi^iv_1a,\quad   bv_1=\xi^iv_1b,\quad av_2+v_2a=\Lam^{-1}v_1ba^p,\quad  bv_2+v_2b=v_1a^{p+1},\\
 v_3^2=0,\quad av_3=\xi v_3a,\quad  bv_3=\xi v_3b,\\
v_1^4=0,\quad v_1v_2+v_2v_1=0,\quad v_2^2+\theta^{-2}v_1^2=0,\\
(-1)^{i+1}\xi^jv_1v_2v_3-\xi^jv_1v_3v_2+v_2v_3v_1+(-1)^iv_3v_2v_1=\beta_1(1-a^p),\\
v_3v_1^2+[(-1)^{i+1}-\xi^j]v_1v_3v_1+(-1)^i\xi^jv_1^2v_3=2\beta_1ba^{-1},\\
\alpha \sum_{\ell=0}^{2}\xi^{2j\ell}(v_1v_3)^2(v_2v_3)^{2}+(v_2v_3)^4+(v_3v_2)^4=o(\beta_1),
\end{gather*}
where $o(\beta_1)\in A_{[7]}$ depends on the parameter $\beta_1\in\K$.  Let $\mathfrak{A}_{i,j,1}(\beta_1):=A$ for short. To determine the structure of $\mathfrak{A}_{i,j,1}(\beta_1)$, it remains to compute $o(\beta_1)$, which is very tedious and complicated because of the complicated commutation relations and comultiplications of generators.
If $\beta_1=0$, then $o(\beta_1)=0$ and hence $A\cong\gr A$. We will study the case $\beta_1\neq 0$ in a subsequent work.
\end{rmk}

\begin{pro}\label{pro:Lifting-Lambda3-k-1-2}
   Assume $A$ is a finite-dimensional Hopf algebra  over $\cH_{p,-1}$ such that $\gr A\cong\BN(V_{i,j}\bigoplus\K_{\chi^k})\sharp\cH_{p,-1}$, where $(i,j)\in\Lambda_p^2-\Lambda^3_p$ and $k\in\I_{0,2p-1}$ is odd such that $(k-1)(pi-j)\equiv 0\mod 2p$. Assume that Conjecture \ref{conj:pro:relations-Nichols-algebra-2-1-nonnnnnn} holds or $p$ is a prime number. Then $A\cong\BN(V_{i,j}\bigoplus\K_{\chi^k})\sharp\cH_{p,-1}$.
\end{pro}
\begin{proof}
Let $r_1=v_1v_2+(-1)^{i+1}v_2v_1$, $r_2=v_2^2+(-1)^{i+1}(\theta\xi^{i+1})^{-2}v_1^2$, $X=
(-1)^{i+1}\xi^jv_1v_2v_3-\xi^jv_1v_3v_2+v_2v_3v_1+(-1)^iv_3v_2v_1$ and $Y=
v_3v_1^2+[(-1)^{i+1}-\xi^j]v_1v_3v_1+(-1)^i\xi^jv_1^2v_3$ for short.  Recall that $\Delta(v_1)=v_1\otimes 1+a^{-j}\otimes v_1+x_2\theta^{-1}ba^{-1-j}\otimes v_2$ and $\Delta(v_2)=v_2\otimes 1+a^{p-j}\otimes v_2+x_1\theta^{-1}ba^{p-1-j}\otimes v_1$. Observe that $a^{-2j}\neq a^p$ for $(i,j)\in\Lambda_p^2-\Lambda_p^3$. Since $(i,j)\not\in\Lambda_p^3$, similar to the proof of Proposition \ref{pro:Non-trivial-simple-object}, $r_1=0=r_2$. Then
it follows from a direct computation that
\begin{align}
\begin{split}
\Delta(X)&=X\otimes 1+a^{-2j}\otimes X+(1+\xi^{2j})\theta^{-2}\xi^{-1-2j}ba^{-1-2j}\otimes Y+\\
&\quad (-1)^i\xi^jx_2\theta^{-1}ba^{-1-2j}\otimes v_3r_2+(-1)^{i+1}\xi^{j}x_2\theta^{-1}ba^{-1-2j}\otimes r_2v_3\\&\quad
+(-1)^i(\xi^{-j}-\xi^j)x_2\theta^{-1}v_3ba^{p-1-2j}\otimes r_2+[(-1)^i\xi^{-j}-1]v_3a^{p-2j}\otimes r_1,\\
&=X\otimes 1+a^{-2j}\otimes X+(1+\xi^{2j})\theta^{-2}\xi^{-1-2j}ba^{-1-2j}\otimes Y
\end{split}\\
\Delta(Y)&=Y\otimes 1+a^{p-2j}\otimes Y+\xi^{1+i}(1-\xi^{2j})ba^{p-1-2j}\otimes X,\\
aX&=\xi^{2i+1+k}X+\lambda^{-1}\xi^{k+i}Yba^p,\quad bX=\xi^{2i+k+1}Xb+\xi^{k+i}Ya^{p+1},\\
aY&=\xi^{2i+k}Ya,\quad bY=\xi^{2i+k}Yb.
\end{align}
Then similar to the proof of Proposition \ref{pro:Non-trivial-simple-object}, a tedious computation on $A_{[2]}$ shows that $X=0=Y$ in $A$. Then by definition,
\begin{align*}
\Delta(v_1^N)=v_1^N\otimes 1+a^{-Nj}\otimes v_1^N+o(v_1^N)\otimes A+A\otimes o(v_1^N),
\end{align*}
where $o(v_1^N)$ belongs to the ideal generated by $r_1,r_2,X,Y$. Since $r_1=r_2=X=Y=0$ in $A$, it follows that $o(v_1^N)=0$ in $A$. Observe that $a^{-Nj}=a^{piN}$. If $iN\equiv 0\mod 2$, then $v_1^N\in\Pp(A)$, which implies that $v_1^N=0$ in $A$. If $iN\not\equiv 0\mod 2$, then $v_1^N\in\Pp_{1,a^p}(A)=\K (1{-}a^p)\bigoplus \K ba^{p-1}\bigoplus \K v_3$, which implies that
\begin{align*}
v_1^N=\alpha_1(1-a^p)+\alpha_2ba^{p-1}+\alpha_3v_3,
\end{align*}
for some $\alpha_1,\alpha_2,\alpha_3\in\K$.
By Proposition \ref{proYD-2}, $av_1^N=\xi^{iN}v_1^Na,~bv_1^N=\xi^{iN}v_1^Nb$. By Proposition \ref{proYD-1}, $av_3=\xi^{k}v_3a,~bv_3=\xi^{k}v_3b$. Then it follows from a direct computation that
\begin{align*}
\alpha_1=0,\quad (\xi^{-1}-\xi^{iN})\alpha_2=0,\quad (\xi^k-\xi^{iN})\alpha_3=0.
\end{align*}
Observe that $\xi^{-kj}=\xi^{-j}$. If $\xi^k=\xi^{iN}$, then $\xi^{kj}=1$ and so $\xi^{-j}=1$, a contradiction. Hence $\alpha_3=0$. Similarly, $\alpha_2=0$. Consequently, $v_1^N=0$ in $A$.

Assume that Conjecture \ref{conj:pro:relations-Nichols-algebra-2-1-nonnnnnn} holds. Let $$Z:=\alpha(-1)^{Ni}\sum_{\ell=0}^{N-2}\xi^{2j\ell}(v_1v_3)^2(v_2v_3)^{N-2}+(-1)^{Ni}(v_2v_3)^N+(v_3v_2)^N.$$
Then
\begin{align*}
\Delta(Z)=Z\otimes 1+a^{-Nj}\otimes Z.
\end{align*}
Observe that $aZ=\xi^{(i+k+1)N}Za$ and $bZ=\xi^{(i+k+1)N}Zb$. Similar to the case $v_1^N$,  $Z=0$ in $A$. Consequently, the defining relations of $\gr A$ hold in $A$. Therefore, $A\cong\gr A$.

Assume that $p$ is a prime number. Let $Z$ be the relation of $\BN(V_{i,j}\bigoplus\K_{\chi^k})$ in ${}_{\cH_{p,-1}}^{\cH_{p,-1}}\mathcal{YD}$ associated to the relation~$[v_{2}v_3]^N=0$ in ${}_{\gr\A_{p,-1}}^{\gr\A_{p,-1}}\mathcal{YD}$.  Then by definition, $\Delta(Z)=Z\otimes 1+a^{-Nj}\otimes Z$. By Proposition \ref{thm-finite-braided-vector-spaces-H}, $p>2$ and $a^{-Nj}=1$. Hence $Z\in\Pp(A)$, which implies that $Z=0$ in $A$. Consequently, the defining relations of $\gr A$ must hold in $A$. Therefore, $A\cong\gr A$.
\end{proof}

\begin{pro}\label{pro:liftings-b2-Hp1-1}
 Let $A$ be a finite-dimensional Hopf algebra  over $\cH_{p,-1}$ such that $\gr A\cong\BN(V_{i,j})\sharp\cH_{p,-1}$, where $(i,j)\in\Lambda_p$ satisfying    $pi-j-2ij\equiv 0\mod 2p$.  Let $N=\ord(\xi^{-ij})$. Assume that Conjecture~\ref{conj-B2-non-Cartan-1}~holds~(e.g.~$N=3$~or~$6$)~or~$p$~is a prime number.  Then~$A\cong\BN(V_{i,j})\sharp\cH_{p,-1}$.
\end{pro}

\begin{proof}
Let $X=(1+\xi^{-ij})v_1v_2v_1+(-1)^{i+1}\xi^{-ij}v_1^2v_2+(-1)^{i+1}v_2v_1^2$, $Y=\alpha v_1^3+v_1v_2^2-(1-\xi^{ij})\xi^{-j(i+1)}v_2v_1v_2-\xi^{-j(i+2)}v_2^2v_1$ for short, where $\alpha=\theta^{-2}\xi^{-(i+1)(2+j)}(1+\xi^{pi+j})$.

Recall that $\Delta(v_1)=v_1\otimes 1+a^{-j}\otimes v_1+x_2\theta^{-1}ba^{-1-j}\otimes v_2$ and $\Delta(v_2)=v_2\otimes 1+a^{p-j}\otimes v_2+x_1\theta^{-1}ba^{p-1-j}\otimes v_1$.

Assume that $N=3$. Then $\alpha=0$ and it follows from a direct computation that
\begin{gather*}
\Delta(X)=X\otimes 1+a^{p-3j}\otimes X,\\
\Delta(Y)=Y\otimes 1+a^{-3j}\otimes Y+(-1)^{i+1}[(-1)^i+(-1)^{i+1}\xi^{-2ij}+\xi^{-j(i+1)}]x_1\theta^{-1}ba^{-1-3j}\otimes X,\\
\Delta(v_1^3)=v_1^3\otimes 1+a^{-3j}\otimes v_1^3+(-1)^{i+1}x_2\theta^{-1}ba^{-1-3j}\otimes X.
\end{gather*}
If $i\in\I_{0,2p-1}$ is odd, then $\xi^{-3j}=-1$ and so $a^{-3j}=a^p$, which implies that $X\in\Pp(H)$. Hence $X=0$ in $A$ and $v_1^3, Y\in\Pp_{1,a^p}(A)$, which implies that  there exist $\alpha_1,\cdots,\alpha_4\in\K$ such that
\begin{align*}
v_1^3=\alpha_1(1-a^p)+\alpha_2ba^{p-1},\quad Y=\alpha_3(1-a^p)+\alpha_4ba^{p-1}.
\end{align*}

Observe that $av_1^3=\xi^{3i}v_1^3a$ and $bv_1^3=\xi^{3i}v_1^3b$. Then it follows that $\alpha_1=0=\alpha_2$, that is, $v_1^3=0$ in $A$. Similarly, $Y=0$ in $A$. Finally, $\Delta(v_2^6)=v_2^6\otimes 1+1\otimes v_2^6$, which implies that $v_2^6=0$ in $A$. If $i\in\I_{0,2p-1}$ is even, then $\xi^{-3j}=1$, that is, $a^{-3j}=1$, and so $X\in\Pp_{1,a^p}(A)$. Hence  $X=\beta_2ba^{p-1}+\beta_1(1-a^p)$ for $\beta_1,\beta_2\in\K$. Then
\begin{align*}
\Delta(v_1^3-\beta_1x_2\theta^{-1}ba^{-1})&=(v_1^3-\beta_1x_2\theta^{-1}ba^{-1})\otimes 1+1\otimes(v_1^3-\beta_1x_2\theta^{-1}ba^{-1})\\&\quad-\beta_2x_2\theta^{-1}ba^{-1}\otimes ba^{p-1}.
\end{align*}
Observe that $av_1^3=\xi^{3i}v_1^3a$ and $bv_1^3=\xi^{3i}v_1^3b$. Similar to the proof of Proposition \ref{pro:Non-trivial-simple-object}, a tedious computation on $A_{[2]}$ shows that  $\beta_1=0=\beta_2$, which implies that $v_1^3=0=X$ in $A$. Then $Y\in\Pp(A)$, which implies that $Y=0$ in $A$. Finally, it follows from a direct computation that $v_2^6\in\Pp(A)$, which implies that $v_2^6=0$ in $A$. Consequently, $A\cong\gr A$.

Assume that $N\neq 3$. Similar to the proof of the case $N=3$,  $X=0=Y$ in $A$.

If $N$ is odd, the by definition,
\begin{align*}
\Delta(v_1^N)=v_1^N\otimes 1+a^{-Nj}\otimes v_1^N+o(v_1^N)\otimes A+A\otimes o(v_1^N),
\end{align*}
where $o(v_1^N)$ belongs to the Hopf ideal generated by  $X,Y$. Since  $X=0=Y$ in $A$, it follows that $o(v_1^N)=0$. Observe that $\xi^{2ij}=(-1)^i\xi^{-j}$ and $N$ is odd. Then $a^{-Nj}=a^{pi}$.  If $i$ is even, then $v_1^N\in\Pp(A)$, which implies that $v_1^N=0$ in $A$. If $i$ is odd, then $v_1^N\in\Pp_{1,a^p}(A)=\K (1{-}a^p)\bigoplus \K ba^{p-1}$, which implies that
\begin{align*}
v_1^N=\mu_1(1-a^p)+\mu_2ba^{p-1}.
\end{align*}
for some $\mu_1,\mu_2\in\K$. Observe that $av_1^N=\xi^{iN}v_1^Na$, $bv_1^N=\xi^{iN}v_1^Nb$. It follows from a direct computation that
\begin{align*}
\mu_1=0,\quad (\xi^{iN}-\xi^{-1})\mu_2=0.
\end{align*}
If $\xi^{iN}=\xi^{-1}$, then $\xi^{-j}=\xi^{ijN}=1$, a contradiction. Hence $\mu_2=0$. Consequently, $v_1^N=0$ in $A$. Let $Z$ be the relation of $\BN(V_{i,j})$ in ${}_{ \cH_{p,-1}}^{ \cH_{p,-1}}\mathcal{YD}$ associated to the relation $v_{2}^{2N}=0$ in ${}_{\gr\A_{p,-1}}^{\gr\A_{p,-1}}\mathcal{YD}$.  Then
\begin{align*}
\Delta(Z)=Z\otimes 1+a^{-2Nj}\otimes Z+o(Z)\otimes A+A\otimes o(Z),
\end{align*}
where $o(Z)$ belongs to the ideal generated by $X, Y, v_1^N$. Since $X=Y=v_1^N=0$ in $A$, then $o(A)=0$. As $a^{-Nj}=a^{pi}$,
\begin{align*}
\Delta(Z)=Z\otimes 1+1\otimes Z,
\end{align*}
which implies that $Z=0$ in $A$. Consequently, the defining relations of $\gr A$ hold in $A$. So, $A\cong\gr A$.

If $N$ is even, the the proof follows from the same line.
\end{proof}

\begin{thm}\label{thm:f.d.HopfH4pp7-1}
Suppose that $p>5$ is a prime number. Let $A$ be a finite-dimensional Hopf algebra over $\cH_{p,-1}$. Then $A$ is a basic Hopf algebra.
\end{thm}
\begin{proof}
By \cite[Theorem 1.3]{AA18b} and Proposition 4.2, the diagram of $A$~is a Nichols algebra over a semisimple object in ${}_{\cH_{p,-1}}^{\cH_{p,-1}}\mathcal{YD}$.  Let $V$ be the diagram of $A$. Then by Theorem \ref{thm-finite-braided-vector-spaces-H}, $V$ is isomorphic either to
\begin{itemize}
  \item  $\bigoplus_{k=1}^n\K_{\chi^{i_k}}$, for some $n\in\N$ and $i_k\in\I_{0,2p-1}$ is odd,
\end{itemize}
or to $(\bigoplus_{k=0}^n\K_{\chi^p})\bigoplus W$, where $W$ is isomorphic to one of the following objects:
\begin{itemize}
\item[(1)] \ $V_{p,j}$,  for odd number $j\neq p$;
\item[(2)] \ $V_{-p-1,j}$,  for odd number $j\neq 0$;
\item[(3)] \ $V_{\frac{1}{2}(p-1),j}$,  if $\frac{1}{2}(p-1)\equiv 0\mod 2$, then $j\neq 0$ is even, otherwise $j\neq p$ is odd;
\item[(4)] \ $V_{\frac{1}{2}(p-1)+p,j}$, if $\frac{1}{2}(p-1)\equiv 0\mod 2$, then $j\neq p$ is odd, otherwise $j\neq 0$ is even;
\item[(5)] \ $V_{p,j}\otimes\K_{\chi^{-1}}$, for odd number $j\neq p$;
\item[(6)] \ $V_{-p-1,j}\otimes\K_{\chi^{}}$,  for even number $j\neq 0$;
\item[(7)] \ $V_{p,j}\bigoplus V_{p,-j}$, for odd number $j\neq p$;
\item[(8)] \ $V_{-p-1,j}\bigoplus V_{-p-1,-j}$, for even number $j\neq 0$.
\end{itemize}
In particular, by \cite[Theorem 2.2.]{G00}, $\BN(V)\cong \BN(\bigoplus_{k=0}^n\K_{\chi^p})\otimes \BN(W)$.

If $V\cong \bigoplus_{k=1}^n\K_{\chi^{i_k}}$, then by Proposition \ref{pro:Lifting-sum-one-dim-objects}, $A\cong\gr A$.

If $V$ is isomorphic to one of the Hopf algebras described in $(1)$--$(2)$, then by Proposition \ref{pro:Non-trivial-simple-object}, $A\cong\gr A$.

If $V$ is isomorphic to one of the Hopf algebras described in $(3)$--$(4)$, then by Proposition \ref{pro:liftings-b2-Hp1-1}, $A\cong\gr A$.

If $V$ is isomorphic to one of the Hopf algebras described in $(5)$--$(6)$, then by Proposition \ref{pro:Lifting-Lambda3-k-1-1}~and~\ref{pro:Lifting-Lambda3-k-1-2}, $A\cong\gr A$.

If $V$ is isomorphic to one of the Hopf algebras described in $(7)$--$(8)$, then by Proposition \ref{pro:Liftings-Lambda1-Lambda2--lambda3-two-simple-objects}, $A\cong\gr A$.

From the preceding discussion, it remains to show that $A\cong\gr A$ for $V\cong\K_{\chi^p}\bigoplus W$, where $W$ is isomorphic to one of the Hopf algebras described in $(1)$--$(4)$.

Assume that $V\cong \K_{\chi^p}\bigoplus V_{p,j}:=\K v_1\bigoplus\K v_2 \bigoplus\K v_3$ for odd number $j\neq p$.  Then $\BN(V)\cong \bigwedge \K_{\chi^p}\otimes\BN(V_{p,j})$. By Proposition \ref{pro:Lifting-sum-one-dim-objects}, $v_3^2=0$ in $A$.  Now we claim that $v_1v_3+v_3v_1=0$ in $A$. Suppose that relation $v_1v_3+v_3v_1=0$ admits non-trivial deformations. Then $v_1v_3+v_3v_1\in A_{[1]}$ and hence there are $\alpha_{i,j}^k\in\K$, $i\in\I_{0,1},\,j\in\I_{0,2p-1},\,k\in\I_{0,3}$, such that
\begin{align*}
v_1v_3+v_3v_1=\alpha_{i,j}^0b^ia^j+\alpha_{i,j}^1v_1b^ia^j+\alpha_{i,j}^2v_2b^ia^j+\alpha_{i,j}^3v_3b^ia^j.
\end{align*}
Observe that $a(v_1v_3+v_3v_1)=(v_1v_3+v_3v_1)a$, $b(v_1v_3+v_3v_1)=(v_1v_3+v_3v_1)b$ and $\epsilon(v_1v_3+v_3v_1)=0$. It follows from a direct computation that  $\alpha_{i,j}^k=0$ for $i\in\I_{0,1},j\in\I_{0,2p-1},k\in\I_{0,3}$. Thus the claim follows.  As $\Delta(v_2)=v_2\otimes 1+a^{p-j}\otimes v_2+x_1\theta^{-1}ba^{p-1-j}\otimes v_1$ and $\Delta(v_3)=v_3\otimes 1+a^p\otimes v_3$, it follows from a direct computation that
\begin{align*}
\Delta(v_2v_3+v_3v_2)=(v_2v_3+v_3v_2)\otimes 1+a^{-j}\otimes (v_2v_3+v_3v_2).
\end{align*}
Similar to the proof of Proposition \ref{pro:lambda_3-liftings-1}, a tedious computation on $A_{[1]}$ shows that $v_2v_3+v_3v_2=0$ in $A$. Then similar to the proof of Proposition \ref{pro:Lifting-Lambda3-k-1-1}  or \ref{pro:Lifting-Lambda3-k-1-2}, the defining relations of  $\BN(V_{p,j})$ hold in $A$. Therefore,  $A\cong\gr A$.

Assume that $W$ is isomorphic to one of the Hopf algebras described in $(2)$--$(4)$, the proof follows from the same lines.
\end{proof}

\begin{thm}\cite[Theorem B]{X18b}\label{thm:f.d.Hopf-H12}
 Suppose $A$ is a finite-dimensional Hopf algebra over $\cH_{3,-1}$ such that the infinitesimal braiding is an indecomposable  object in ${}_{\cH_{3,-1}}^{\cH_{3,-1}}\mathcal{YD}$. Then $A$ is isomorphic to
\begin{itemize}
  \item[(a)] \ $\bigwedge\K_{\chi^{k}}\sharp \cH_{3,-1}$, for $k\in\{1,3,5\}$;
  \item[(b)] \ $\BN(V_{3,1})\sharp \cH_{3,-1}$; $\BN(V_{3,5})\sharp \cH_{3,-1}$; $\BN(V_{2,2})\sharp \cH_{3,-1}$; $\BN(V_{2,4})\sharp \cH_{3,-1}$;
  \item[(c)] \ $\BN(V_{4,1})\sharp\cH_{3,-1}$; $\BN(V_{4,5})\sharp\cH_{3,-1}$;
  \item[(d)] \ $\BN(V_{1,1})\sharp\cH_{3,-1}$; $\BN(V_{1,5})\sharp\cH_{3,-1}$; $\BN(V_{4,2})\sharp\cH_{3,-1}$; $\BN(V_{4,4})\sharp\cH_{3,-1}$;
  \item[(e)] \ $\mathfrak{A}_{1,2}(\mu)$, for some $\mu\in\K$;
  \item[(f)] \ $\mathfrak{A}_{1,4}(\mu)$, for some $\mu\in\K$.
\end{itemize}
\end{thm}
\begin{proof}
By \cite[Theorem 1.3]{AA18b} and Proposition 4.2, the diagram of $A$ is a Nichols algebra over a simple object in ${}_{\cH_{3,-1}}^{\cH_{3,-1}}\mathcal{YD}$. Let $V$ be the infinitesimal braiding of $A$. Then by Theorem \ref{thm-finite-braided-vector-spaces-H}, $V$ is isomorphic to $\K_{\chi^{k}}$ for $k\in\{1,3,5\}$, $V_{1,1}$, $V_{4,2}$, $V_{3,1}$, $V_{2,2}$, $V_{1,4}$, $V_{4,5}$, $V_{2,4}$, $V_{3,5}$, $V_{4,4}$, $V_{1,5}$, $V_{4,1}$ or $V_{1,2}$.

If $V\cong \K_{\chi^{k}}$ for $k\in\{1,3,5\}$, then by Proposition \ref{pro:Lifting-sum-one-dim-objects}, $A\cong \bigwedge\K_{\chi^{k}}\sharp \cH_{3,-1}$.

If $V\cong V_{3,j}$ or $V_{2,j+5}$ for $j\in\{1,5\}$, then by Proposition \ref{pro:Non-trivial-simple-object}, $A\cong\BN(V)\sharp\cH_{3,-1}$.

If $V\cong V_{1,j}$ or $V_{4,j+5}$ for $j\in\{1,5\}$, then by Proposition \ref{pro:liftings-b2-Hp1-1}, $A\cong\BN(V)\sharp\cH_{3,-1}$.

If $V\cong V_{4,j}$, then by Proposition \ref{pro:liftings-b2-n-H4p}, $A\cong\gr A$. If $V\cong V_{1,j}$ for $j\in\{1,4\}$, then by Proposition \ref{pro:liftings-b2-n-H4p}, $A\cong\mathfrak{A}_{1,j}(\mu)$ for some $\mu\in\K$.

The Hopf algebras from different families are pairwise non-isomorphic since their infinitesimal braidings are pairwise non-isomorphic as Yetter-Drinfeld modules over $\cH_{3,-1}$.
\end{proof}

\begin{thm}\label{thm:f.d.Hopf-H16}
 Assume $A$ is a finite-dimensional Hopf algebra over $\cH_{4,-1}$. If the diagram of $A$ admits non-trivial quadratic relations, then $A$ is a basic Hopf algebra.
\end{thm}
\begin{proof}
Let $V$ be the infinitesimal braiding of $A$.  By \cite[Theorem 1.3]{AA18b} and Proposition 4.2, the diagram of $A$ is a Nichols algebra over a semisimple object in ${}_{\cH_{p,-1}}^{\cH_{p,-1}}\mathcal{YD}$.

Assume that $V$ is a simple object in ${}_{\cH_{4,-1}}^{\cH_{4,-1}}\mathcal{YD}$ such that $\BN(V)$ admits non-trivial quadratic relations. Then by Proposition \ref{pro:Nichols-algebra-simple-p=4-1} and Remark \ref{rmk:relations-Nichols-algebra-2-dim-simple-quad}, $V$ is isomorphic to one of the following objects
\begin{itemize}
\item $\K_{\chi^{i_k}}$ with $i_k\in\{1,3,5,7\}$;

\item  $V_{i,j}$, for $(i,j)\in\{(2,2),(2,6),(6,2),(6,6),(4,1),(4,3),(4,5),(4,7)\}$, i.e., $(i,j)\in\Lambda_4^1$;

\item  $V_{i,j}$, for $(i,j)\in\{(5,2),(5,6), (1,2),(1,6),(3,3),(3,1),(3,7),(3,5)\}$, i.e., $(i,j)\in\Lambda_4^2-\Lambda_4^3$.
\end{itemize}
Then by Propositions \ref{pro:Lifting-sum-one-dim-objects} and \ref{pro:Non-trivial-simple-object}, $A\cong\gr A$.

Assume that $V$ is a non-simple semisimple object in ${}_{\cH_{4,-1}}^{\cH_{4,-1}}\mathcal{YD}$. Then by Theorem \ref{thm-finite-braided-vector-spaces-H-p=4}, $\BN(V)$ admits non-trivial quadratic relations. By Propositions \ref{pro:Lifting-sum-one-dim-objects},  \ref{pro:Liftings-Lambda1-Lambda2--lambda3-two-simple-objects}, \ref{pro:Liftings-Lambda1-two-simple-objects} and \ref{pro:Lifting-Lambda3-k-1-2}, it follows that $A\cong\gr A$.
\end{proof}

\vskip10pt \centerline{\bf ACKNOWLEDGMENT}

\vskip10pt
The authors are partially supported by the NSFC (Grants No. 11926353, 12171155, 12071094) and in part by Science and Technology Commission of Shanghai Municipality (No. 18dz2271000). The first author is grateful to the CSC foundation of China for support and to Prof. Giovanna Carnovale for her kind hospitality during Xiong's one-year visit from October, 2017 to September, 2018 at University of Padova, Italy,  to Prof. G.A. Garcia for his comments
and introducing him the GAP software for use, and to Prof. Quanshui Wu for his kind support for his visiting at the Shanghai Center for Mathematical Sciences.

\end{document}